\documentclass[11pt,reqno]{amsart}
\usepackage[utf8]{inputenc} 
\usepackage[T1]{fontenc}    
\usepackage[english]{babel} 
\usepackage{amsmath}
\usepackage{amsthm}
\usepackage{amsfonts}
\usepackage{amsbsy,amsfonts,amsmath,amssymb,amscd,amsthm}
\usepackage{mathrsfs,float,color}
\usepackage[mathcal]{euscript}
\usepackage{enumitem}
\usepackage{mathtools}
\usepackage[colorlinks=true,linkcolor=blue,citecolor=green]{hyperref}
\usepackage{tikz}
\usepackage{tikz-cd}
\usepackage{xcolor}

\numberwithin{equation}{section}

\usepackage[a4paper,text={6.1in,8in},centering,top=1.5in]{geometry}
\small\normalsize
\newtheorem{mainthm}{Theorem}
\newtheorem{theorem}{Theorem}[section] 
\newtheorem{proposition}[theorem]{Proposition}
\newtheorem{lemma}[theorem]{Lemma}
\newtheorem{corollary}[theorem]{Corollary}      
\newtheorem{remark}[theorem]{Remark}

\theoremstyle{plain}
\newtheorem{mainprop}{Proposition}
\newtheorem{maincor}[mainprop]{Corollary}

\newtheorem{lem}[theorem]{Lemma}

\newtheorem{prop}[theorem]{Proposition}

\theoremstyle{definition}

\newtheorem{rem}[theorem]{Remark}

\newcommand{\PP}{\mathrm{P}}

\newcommand{\cE}{\mathcal{E}}

\newcommand{\Prob}{\mathrm{Prob}}

\title{Integral representation of Lyapunov exponents}
\author[Barrientos]{Pablo G. Barrientos}
\address{\centerline{Instituto de Matem\'atica e Estat\'istica, UFF}
    \centerline{Niter\'oi,  Brazil}}
\email{pgbarrientos@id.uff.br}
\author[Nisoli]{Isaia Nisoli}
\address{\centerline{Instituto de Matem\'atica, Universidade Federal do Rio de Janeiro}
    \centerline{Av.\ Athos da Silveira Ramos 149, Rio de Janeiro, RJ 21941-909, Brazil}}\email{nisoli@im.ufrj.br}

\begin{document}

\begin{abstract}
We introduce a new operator-theoretic construction of Lyapunov growth for Markov-driven systems. 
The construction is based on an abstract variational principle for asymptotic growth rates arising from a subadditive process generated by an intertwining pair of Markov operators on a measurable bundle with compact fibers: for each invariant base measure, the fiberwise-maximal growth rate equals the supremum of the fiber integral over invariant lifts, and this supremum is attained on an ergodic lift. For random linear bundle morphisms driven by Markovian place-dependent noise, this yields Lyapunov exponents defined intrinsically from the initial law on the state-phase space.  We prove that these exponents coincide with the classical path-space skew-product Lyapunov spectrum, showing that the pointwise exponents depend only on the current noise state and initial position, not on the full noise realization.  We also obtain new asymptotic formulas for the sums of Lyapunov exponents as conditional annealed growth rates along individual directions.  As further consequences of this variational principle, we recover and extend to singular linear bundle morphisms the classical projective variational formulas.
 \end{abstract}

\maketitle

\thispagestyle{empty}
\section{Introduction}
\enlargethispage{1cm}

Lyapunov exponents quantify asymptotic growth rates and play a central role in smooth and random dynamics, probability, and applications.  In many settings, their existence follows from Kingman’s subadditive ergodic theorem~\cite{ruelle1979ergodic}, yet one often needs \emph{representation formulas} that express these exponents as integrals of explicit observables against invariant (or stationary) measures.  Classical instances include Furstenberg-Kifer formulas for random matrix products and the projective integral formulas for sums of Lyapunov exponents of cocycles~\cite{Fur:63,Kif:86}. These formulas provide variational principles, open the door to stability results via perturbation theory for operators, and suggest computational schemes. 

We prove an abstract operator-theoretic variational principle for asymptotic growth rates driven by Markov operators. The abstract framework consists of a measurable bundle \(\pi:\hat X\to X\) with compact fibers, a pair of Markov operators \(P\) on \(B(X)\) and \(\hat P\) on \(B(\hat X)\) intertwining through \(\pi\) via the commutative diagram
\[
\begin{CD}
B(X) @>{ P}>> B( X)\\
@V{\pi^*}VV @VV{\pi^*}V\\
B(\hat X) @>>{\hat P}> B( \hat X)
\end{CD}
\qquad\text{with}\qquad
\pi^*\phi:=\phi\circ\pi,
\]
where \(B(Z)\) denotes the Banach lattice of bounded Borel functions on a space \(Z\). Theorem~\ref{mainthm:A} shows that fiberwise maximization sends additive processes on \(\hat X\) to \(P\)-subadditive sequences on the base.  Theorem~\ref{thm:VP} proves that the resulting growth rate admits a variational representation:
for each \(P\)-invariant probability \(\mu\) on \(X\), the growth rate equals the supremum of \(\int \hat\phi_1\,d\hat\mu\) over all \(\hat P\)-invariant lifts \(\hat\mu\) of \(\mu\), and the supremum is attained on an ergodic lift.
Together with the Markov-operator versions of Kingman's theorem
and its uniform counterpart proved in~\cite{BM},
Theorem~\ref{thm:VP} may be viewed as a relative variational
principle for Lyapunov growth. In the classical relative
variational principle, one fixes an invariant measure on a factor
and optimizes entropy or pressure over its invariant
lifts~\cite{LW77}; here the base measure $\mu$ is fixed, while the
optimization is performed over $\hat P$-invariant lifts of $\mu$
and involves the integral of the fiber potential. When the base
measure is also allowed to vary, the resulting maximal-growth
problem is related, in the setting of matrix cocycles, to the
ergodic theory of joint spectral radii and their associated Mather
sets~\cite{Mor13}. The uniform and dual
counterparts developed in Appendices~\ref{appendix-A} and~\ref{app:dual-min-growth} extend this viewpoint,
respectively, to unconstrained maximal growth and to minimal
directional growth.


The first application developed in this paper concerns the Lyapunov spectra of finite-dimensional random linear bundle morphisms
\(F:T\times\mathcal E\to\mathcal E\) covering a random map
\(f:T\times X\to X\) driven by Markovian place-dependent noise.
Classically, one passes to the path space
\(\Omega=T^{\mathbb N}\) and studies the associated skew-product
cocycle with respect to a driving measure \(m\) on
\(\Omega\times X\). In this formulation, the pointwise Lyapunov
exponents \(\lambda_k(\omega,x)\) depend, a priori, on the entire
noise realization \(\omega=(\omega_i)_{i\geq0}\). We show that
passing to the path-space cocycle is not necessary to define the
spectrum. Instead, the Lyapunov exponents
\(\lambda_k(\omega_0,x)\) can be constructed intrinsically from an
invariant initial law \(\nu\) on the state-phase space
\(Z=T\times X\), using the Markov operator \(P\) on \(Z\) associated
with \(f\). The corresponding growth rate is described by the
abstract variational principle above. Theorem~\ref{mainthm:B} proves
that this intrinsic spectrum agrees with the classical path-space
spectrum:
\[
\lambda_k(\omega,x)=\lambda_k(\omega_0,x)
\quad\text{for $m$-a.e.~$(\omega,x)$}
\quad\text{and}\quad
\lambda_k(m)=\lambda_k(\nu),
\]
where
\[
\lambda_k(m):=\int\lambda_k(\omega,x)\,dm
\qquad\text{and}\qquad
\lambda_k(\nu):=\int\lambda_k(\omega_0,x)\,d\nu.
\]
Moreover, Corollary~\ref{maincor:B} gives the asymptotic
representation
\[
\sum_{i=1}^k \lambda_i(m)
=
\lim_{n\to\infty}
\sup_{[v]}
\frac{1}{n}
\mathbb{E}\bigg[
\log\frac{\bigl\|\wedge^k F^n_{(\omega,x)}v\bigr\|}{\|v\|}
\biggm|
(\omega_0,x,[v])
\bigg]
\quad\text{for $\nu$-a.e.~$(\omega_0,x)$},
\]
where the conditional expectation averages over the future noise
starting from the indicated lifted state. At finite time,
maximizing over directions and averaging over the noise need not
commute. The corollary shows that, after normalization, the two
procedures nevertheless have the same asymptotic growth rate. In
particular, the operator norm can be replaced by the averaged growth
along individual directions, producing a more explicit
representation that is well suited to both theoretical estimates
and computation. In the Bernoulli case, the dependence on the
current noise state also disappears, and one recovers the classical
\(\omega\)-independence of the Lyapunov spectrum
\cite{liu2006smooth}.


\subsection{Main abstract result}\label{ss:thmA-intro}

Let \(\pi:\hat X\to X\) be a \emph{standard Borel bundle with compact metric fibers}. This means that \(\pi\) is a measurable surjection between standard Borel spaces \(\hat X\) and \(X\), and there exist a compact metric space \(Y\), a Borel set \(\mathcal E\subset X\times Y\) with nonempty compact sections
$\mathcal E_x:=\{y\in Y:(x,y)\in\mathcal E\}$,
and a fiber-preserving Borel isomorphism \(\iota:\hat X\to\mathcal E\) such that, for every \(x\in X\), the restriction
\(
\iota:\pi^{-1}(x)\longrightarrow \mathcal E_x
\)
is a homeomorphism.

Let \(\mathscr C_b(\hat X)\) denote the space of \emph{bounded fiberwise continuous functions} on \(\hat X\), that is, bounded measurable functions \(\phi:\hat X\to\mathbb R\) whose restriction to each fiber \(\pi^{-1}(x)\) is continuous. We likewise consider \emph{fiberwise upper semicontinuous} functions, allowing values in \([-\infty,\infty)\). For \(u:\hat X\to[-\infty,\infty]\), define
\[
\mathcal M u(x):=\sup_{z\in\pi^{-1}(x)}u(z).
\]
If \(u\) is fiberwise upper semicontinuous, compactness of the fibers gives
\[
\mathcal M u(x)=\max_{z\in\pi^{-1}(x)}u(z)\in[-\infty,\infty).
\]
Let $B(Z)$ be the Banach lattice of bounded measurable functions on a standard Borel space $Z$. A linear operator $Q: B(Z) \to B(Z)$ is called \emph{$\sigma$-Markov} if 
\begin{enumerate}[label=(\roman*)]
  \item  $Q$ is \emph{Markov}: it is positive ($Q\varphi \geq 0$ if $\varphi \geq 0$) and $Q 1_Z = 1_Z$;
  \item $Q$ is \emph{$\sigma$-order continuous}: $\varphi_n \downarrow 0$ pointwise, then $Q\varphi_n \downarrow 0$.
\end{enumerate}
The $\sigma$-order continuity ensures that the Markov operator $Q$ is induced by a transition probability kernel; see Lemma~\ref{lem:kernel-representation}. By duality, $Q$ induces a linear operator $Q^*$ on the space of probability measures on $Z$. We write \(\mathcal I(Q)\) for the convex set of \(Q^*\)-invariant probability measures (i.e., $Q^*\mu = \mu$); its extreme points are the \emph{ergodic} measures.

A sequence \((\phi_n)_{n\ge1}\) of extended real-valued measurable functions (or equivalence classes of functions) on \(Z\) is called \emph{\(Q\)-subadditive} if
\begin{equation}\label{Qsub}
\phi_{n+m}\le \phi_n+Q^n\phi_m
\qquad\text{for all }m,n\ge1.
\end{equation}
We define \emph{\(Q\)-superadditivity} if $(-\phi_n)_{n\geq 1}$ is $Q$-subadditive 
and \emph{$Q$-additive} if equality holds in~\eqref{Qsub}.
For the following results,  we consider  \(\sigma\)-Markov operators  \(\hat P:B(\hat X)\to B(\hat X)\) and \(P:B(X)\to B(X)\). These operators also act on Borel measurable functions with values in $[-\infty,\infty)$; see Lemma~\ref{lem:extend-markov-univ}. 
For \(\mu\in\mathcal I(P)\), let \(\mathcal I_\mu(\hat P)\) denote the set of probability measures \(\hat\mu\) on \(\hat X\) such that \(\hat P^*\hat\mu=\hat\mu\) and \(\pi_*\hat\mu=\mu\). We call $\hat\mu\in\mathcal I_\mu(\hat P)$ an
\emph{ergodic lift of $\mu$} if it is an extreme point of the
convex set $\mathcal I_\mu(\hat P)$. We also write
\[
\pi^*:B(X)\to B(\hat X),\qquad \pi^*\varphi:=\varphi\circ\pi,
\]
for the pullback operator. This
operator is an isometric embedding of Banach lattices.

\begin{mainthm} \label{mainthm:A}
    Let $P: B(X) \to B(X)$ and $\hat P: B(\hat X) \to B(\hat X)$ be $\sigma$-Markov operators such that  $\hat P\circ \pi^* = \pi^* \circ P$. 
    If $(\hat \phi_n)_{n \ge 1}$ is a $\hat{P}$-additive sequence of measurable functions $\hat\phi_n: \hat X \to [-\infty,\infty)$, then the sequence $(\phi_n)_{n\geq 1}$  defined by $\phi_n \coloneqq \mathcal{M}\hat \phi_n$ is $P$-subadditive. 
\end{mainthm}
Let \(Q\) be a Markov operator and consider $\nu \in \mathcal I(Q)$. If \(\Psi=(\psi_n)_{n\ge1}\) is a \(Q\)-subadditive sequence with \(\psi_1^+:=\max\{0,\psi_1\}\in L^1(\nu)\), using subadditive Fekete's lemma, we define its \emph{growth rate} with respect to \(\nu\in\mathcal I(Q)\) by
\[
\Lambda(\nu;Q,\Psi)
\coloneqq
\lim_{n\to\infty}\frac1n\int \psi_n\,d\nu  
=
\inf_{n\ge1}\frac1n\int \psi_n\,d\nu\in [-\infty,\infty).
\]
If \(\Psi\) is \(Q\)-additive, then
\[
\Lambda(\nu;Q,\Psi)=\int \psi_1\,d\nu
\]
so the growth rate reduces to the spatial average of the generator. In general, for a merely subadditive process, \(\Lambda(\nu;Q,\Psi)\) should be viewed as the asymptotic finite-time spatial growth, and hence as an abstract Lyapunov exponent. Moreover, if \(\nu\) is ergodic, then the Kingman theorem for Markov operators~\cite[Thm.~D]{BM} yields
\begin{equation} \label{eq:BM}
\Lambda(\nu;Q,\Psi)=\lim_{n\to\infty}\frac1n\psi_n \in [-\infty,\infty)
\qquad\text{for \(\nu\)-a.e.}
\end{equation}
As a consequence of Theorem~\ref{mainthm:A}, given $\mu \in \mathcal{I}(P)$,  one can introduce the growth rate $\Lambda(\mu):=\Lambda(\mu;P,\Phi)$ for the fiberwise-maximum sequence $\Phi=(\phi_n)_{n\ge 1}$. The following result gives a variational principle for this quantity with the convention that $\sup \emptyset = \max \emptyset = -\infty$. 

\begin{mainthm}\label{thm:VP}
Under the assumptions of Theorem~\ref{mainthm:A}, let \(\hat\phi_1:\hat X\to[-\infty,\infty)\) be a fiberwise upper semicontinuous function and define 
\[
\hat\phi_n:=\sum_{k=0}^{n-1}\hat P^k\hat\phi_1
\qquad \text{and} \qquad
\phi_n:=\mathcal M\hat\phi_n \qquad \text{for $n\geq 1$}.
\]
Let \(\mu\in\mathcal I(P)\) be such that 
$ \phi_1^+:=\max\{0,\phi_1\}\in L^1(\mu)$.  
Assume that there is a Borel set \(K\subset\hat X\) such that, for every \(\psi\in\mathscr C_b(\hat X)\),
\begin{enumerate}
\item[(i)] every fiberwise discontinuity point of \(\hat P\psi\) belongs to \(K\);
\item[(ii)] \(\hat\phi_1(z)=-\infty\) for every \(z\in K\).
\end{enumerate}
Then
\[
\Lambda(\mu) 
=
\sup_{\hat\mu\in\mathcal I_\mu(\hat P)}\int\hat\phi_1\,d\hat\mu
=
\max\left\{
\int\hat\phi_1\,d\hat\mu:
\hat\mu\in\mathcal I_\mu(\hat P)\ \text{ergodic}
\right\}.
\]
Moreover, if \(\Lambda(\mu)>-\infty\), every maximizing lift \(\hat\mu\in\mathcal I_\mu(\hat P)\) satisfies \mbox{\(\hat\mu(K)=0\).}
\end{mainthm}

The intertwining condition between $P$ and $\hat P$ admits an intrinsic reformulation. Let 
\begin{equation}\label{eq:set-F}
\mathcal B_\pi
\coloneqq
\{u\in B(\hat X):u(z)=u(z') \text{ whenever } \pi(z)=\pi(z')\}
\end{equation}
be the closed subspace of bounded \(\pi\)-fiberwise constant functions. The pullback \(\pi^*\) identifies \(B(X)\) isometrically, and as a lattice, with \(\mathcal B_\pi\). By Lemma~\ref{lem:markov-factorization}, a \(\sigma\)-Markov operator \(\hat P:B(\hat X)\to B(\hat X)\) satisfies
\(
\hat P(\mathcal B_\pi)\subset \mathcal B_\pi
\)
if and only if there exists a unique \(\sigma\)-Markov operator \(P:B(X)\to B(X)\) such that
\(
\hat P\circ \pi^*=\pi^*\circ P.
\)
In that case,
$P=(\pi^*)^{-1}\circ \hat P\circ \pi^*$.
This intrinsic reformulation makes Theorem~\ref{mainthm:A} applicable directly to any bundle Markov operator that preserves fiberwise constant functions. In the next subsection, we apply it to projective Markov operators associated with the exterior power of random linear bundle morphisms. 

\subsection{Random dynamics and bundle morphisms}\label{s:random-morphisms}
Throughout, let \((X,\mathscr B)\) be a standard measurable space, called the \emph{phase space}. 
A discrete random dynamical system is determined by a measurable map $f:T\times X\to X$,
where \((T,\mathscr A)\) is a standard Borel space,  called the \emph{state space}. Given \(\omega=(\omega_i)_{i\ge0}\in\Omega:=T^{\mathbb N}\), we write
\[
f_\omega^0=\mathrm{id},
\qquad
f_\omega^n=f_{\omega_{n-1}}\circ\cdots\circ f_{\omega_0},
\qquad n\ge1,
\]
where \(f_t:=f(t,\cdot)\).

\subsubsection{Noise}
Randomness is specified by an ergodic shift-invariant probability measure \(\mathbb P\) on \(\Omega\). The basic examples are ergodic Markov measures. Such a measure is determined by a transition probability kernel \(Q(t,A)\) on \(T\times\mathscr A\) and an ergodic stationary law \(p\). In particular, Bernoulli measures \(\mathbb P=p^{\mathbb N}\) correspond to  \(Q(t,A)=p(A)\) for every \(t\in T\) and \(A\in\mathscr A\). In all these cases, the noise is \emph{place-independent}: the law of the next state \(\omega_n\) depends at most on the previous state \(\omega_{n-1}\), and not on the current position in \(X\).


A more general situation arises when the noise depends on the current position in phase space. Given \(x\in X\), define
\[
x_0=x,
\qquad
x_n=f_{\omega_{n-1}}(x_{n-1})=f_\omega^n(x),
\qquad n\ge1.
\]
In an \emph{iid place-dependent random iteration}, one has $x_0\sim \mu$ and $\omega_n\sim p_{x_n}(\cdot)$, 
where \(x\mapsto p_x\) is a measurable family of probability measures on \(T\), and \(\mu\) is an initial law on \(X\). More generally, a \emph{Markov place-dependent random iteration} is given by
\begin{equation}\label{eq:kernel-q}
(\omega_0,x_0)\sim \nu,
\qquad
\omega_n\sim Q_{x_n}(\omega_{n-1},\cdot),
\end{equation}
where \(Q_x(t,A)\) is a measurable family of transition probability kernels on \(T\times\mathscr A\), indexed by \(x\in X\), and \(\nu\) is an initial law on \(T\times X\). The iid place-dependent case is recovered by taking \(d\nu=p_x\,d\mu(x)\) and \(Q_x(t,\cdot)=p_x(\cdot)\).

\subsubsection{Driving measure}
We consider the \emph{one-step skew-product}
\[
\mathcal F:\Omega\times X\to\Omega\times X,
\qquad
\mathcal F(\omega,x)=(\sigma(\omega),f_{\omega_0}(x)),
\]
where \(\sigma\) is the shift on \(\Omega\). The orbit \(x_n=f_\omega^n(x)\) is precisely the fiberwise orbit of \((\omega,x)\) under \(\mathcal F\). The relevant invariant objects are ergodic \(\mathcal F\)-invariant probability measures \(m\) on \(\Omega\times X\). For Bernoulli random maps, these measures are product measures \(m=\mathbb P\times\mu\), where \(\mathbb P\) is Bernoulli. For Markovian random maps, the first marginal of \(m\) on \(\Omega\) is the corresponding Markov measure \(\mathbb P\), and one has a disintegration
$dm=\mu_{\omega_0}\,d\mathbb P(\omega)$, 
where the fiber measures depend only on the first coordinate \(\omega_0\); see~\cite{BMNNT}. In both cases, the initial law on~\(T\times X\)~is
\[
\nu=(\pi_0)_*m,
\qquad
\pi_0(\omega,x)=(\omega_0,x).
\]

For place-dependent random iterations, the construction of \(m\) requires an auxiliary Markov chain on the extended space $(Z,\mathscr{C})=(T\times X,\mathscr{A}\otimes \mathscr{B})$. Define
\begin{equation}\label{eq:barf}
\bar f:T\times Z\to Z,
\qquad
\bar f_s(t,x)=(s,f_t(x)).
\end{equation}
Let \(\tilde{\mathbb P}\) be the Markov measure on \(\tilde\Omega:=Z^{\mathbb N}\), determined by the transition kernel
\begin{equation}\label{eq:Markov-chain}
q(z,C)=\int 1_C\circ \bar{f}_s(t,x) \,Q_{f_t(x)}(t,ds),
\qquad z=(t,x)\in Z,
\end{equation}
on $Z\times \mathscr{C}$ and the stationary law \(\nu\). We then define the \emph{driving measure}
$$m=\tilde{\pi}_*\tilde{\mathbb{P}} \ \ \text{on $\Omega \times X$} \ \ \ \text{where \ \
$\tilde{\pi}: \tilde{\Omega}\to \Omega \times X$, \ \  $\tilde{\pi}((t_n,x_n)_{n\geq 0})=((t_n)_{n\geq 0},x_0)$.}
$$
Moreover, by construction, $\nu=(\pi_0)_*m$. 

\subsubsection{Random linear bundle morphisms}
Let \(\pi:\mathcal E\to X\) be a measurable rank-\(d\) vector bundle with a measurable Euclidean norm \(\|\cdot\|\), and for \(k=1,\dots,d\), write \(\wedge^k\mathcal E\) for its \(k\)-th exterior power bundle. A measurable map
$F:T\times\mathcal E\to\mathcal E$
covering \(f\) is called a \emph{random linear bundle morphism} if, for each \(t\in T\), the map \(F_t:\mathcal E\to\mathcal E\) is a linear bundle morphism covering \(f_t\). Equivalently,
\[
F_{t,x}:\mathcal E_x\to\mathcal E_{f_t(x)}
\quad\text{is linear}
\qquad\text{and}\qquad
\pi\circ F_t=f_t\circ\pi.
\]

To introduce Lyapunov exponents, we fix an invariant probability
measure on the base. Ergodicity is needed only when one wants the
pointwise exponents to be almost surely constant. There are two
equivalent approaches:
\begin{enumerate}[leftmargin=1cm]
\item[(a)] by using the initial law \(\nu\) on \(T\times X\);
\item[(b)] by passing to the skew-product \(\mathcal F\) and using the driving measure \(m\) on \(\Omega\times X\).
\end{enumerate}
The classical approach (b) proceeds by extending the bundle to \(\bar\pi:\bar{\mathcal E}\to \Omega\times X\), where \(\bar{\mathcal E}:=\Omega\times\mathcal E\), and reinterpreting \(F\) as a bundle morphism covering \(\mathcal F\) via  \(F_{(\omega,x)}:=F_{\omega_0,x}\). Its iterates are
\[
F^n_{(\omega,x)}=
F_{\mathcal F^{n-1}(\omega,x)}\circ\cdots\circ F_{(\omega,x)},
\qquad n\ge1,
\]
and we assume
\begin{equation}\label{eq:integrability-m}
\int \log^+\|F_{(\omega,x)}\|\,dm<\infty.
\end{equation}
Assume first that \(m\) is ergodic. Following~\cite{ruelle1979ergodic}, a standard consequence of Kingman's subadditive ergodic theorem is the existence of numbers 
\begin{equation} \label{eq:lypunov-exponents}
        \infty>\lambda_1(m) \ge \lambda_2(m) \ge \dots \ge \lambda_d(m) \ge -\infty,
\end{equation}
called the \emph{Lyapunov exponents} of $m$, such that for every $k=1,\dots,d$ and $m$-a.e.~$(\omega,x)\in \Omega \times X$
\[
    \lambda_k(m) = \lim_{n \to \infty} \frac{1}{n} \log \sigma_k(F^n_{(\omega,x)}) 
    \quad \text{and} \quad 
    \sum_{i=1}^k \lambda_i(m) = \lim_{n \to \infty} \frac{1}{n} \log \|\wedge^k F^n_{(\omega,x)}\|,
\]
where $\sigma_k(L)$ denotes the $k$-th singular value of a linear map $L$. Furthermore, the sum of exponents satisfies 
\[
   \sum_{i=1}^k \lambda_i(m) = \lim_{n \to \infty} \frac{1}{n} \int \log \|\wedge^k F^n_{(\omega,x)}\| \, dm = \inf_{n \ge 1} \frac{1}{n} \int \log \|\wedge^k F^n_{(\omega,x)}\| \, dm.
\] 
For the pointwise Lyapunov spectrum in the non-ergodic case, see Appendix~\ref{ss:lyapunov-spectrum}.
The drawback of this approach is that it introduces the auxiliary path space \(\Omega\). In particular, without ergodicity one obtains pointwise exponents \(\lambda_k(\omega,x)\), which a priori depend on the full path~\(\omega\). However, for Bernoulli random maps, this dependence is known to disappear~\cite[Thm.~3.2]{liu2006smooth}. 

 Approach (a) works directly on the state-phase space \(Z=T\times X\). 
This approach avoids the auxiliary path space and directly shows, in the non-ergodic case, that the pointwise exponents depend only on the current noise state and position, not on the full path \(\omega\).
Below, we describe this construction in detail.

\subsubsection{Lyapunov exponents}\label{ss:spectrum-lyapunov-random}
Associated with the Markov place-dependent random iteration of \(f\), or equivalently with \(\bar f\) in~\eqref{eq:barf}, we have 
\begin{equation}\label{eq:fully-place-P}
P\varphi(t,x)
=
\int \varphi\bigl(s,f_t(x)\bigr)\,Q_{f_t(x)}(t,ds), \ \  (t,x)\in T\times X, \  \varphi \in B(Z).
\end{equation}
We assume that the initial law \(\nu\) is \(P^*\)-invariant, and that
\begin{equation}\label{eq:integrability-nu}
\int \log^+\|F_{t,x}\|\,d\nu<\infty.
\end{equation}
Since \(\nu=(\pi_0)_*m\) and \(F_{(\omega,x)}=F_{t,x}\) when \(\omega_0=t\), this is equivalent to~\eqref{eq:integrability-m}.

As an application of Theorem~\ref{mainthm:A} and Kingman’s subadditive theorem for Markov operators~\cite[Thm.~D]{BM} we give the following: 
\begin{mainprop} \label{mainprop:kingman1}
For every $k=1,\dots,d$, there is a $P$-invariant function $\Lambda_k$ with $\Lambda_k^+ \in L^1(\nu)$,
\[
\Lambda_k(\omega_0,x)=\lim_{n\to\infty}
\sup_{[v]\in\PP(\wedge^k\mathcal E_x)}
\frac1n
\mathbb E\biggl[
\log\frac{\|\wedge^kF^n_{(\omega,x)}v\|}{\|v\|}
\,\biggm|\,
(\omega_0,x,[v])
\biggr]
\quad\text{for \(\nu\)-a.e.\ \((\omega_0,x)\)}
\]
and 
\[
\Lambda_k(\nu):=\int \Lambda_k \, d\nu = \lim_{n\to \infty} \frac1n\int \sup_{[v]\in\PP(\wedge^k\mathcal E_x)}
\mathbb E\biggl[
\log\frac{\|\wedge^kF^n_{(\omega,x)}v\|}{\|v\|}
\,\biggm|\,
(\omega_0,x,[v])
\biggr] \, d\nu.
\]
Moreover, if $\nu$ is ergodic, then $\Lambda_k(\nu)=\Lambda_k(\omega_0,x)$ for $\nu$-a.e.~$(\omega_0,x)$.
\end{mainprop}

We define the \emph{pointwise Lyapunov exponents} of \(F\) with respect to \(\nu\) by setting \(\Lambda_0(t,x)=0\) and defining for each \(k=1,\dots,d\), 
\begin{equation} \label{eq:def-Lyapunov-law}
\begin{aligned}
\lambda_k(t,x) &\coloneqq \Lambda_k(t,x)-\Lambda_{k-1}(t,x)
\ \ \text{if }\Lambda_{k-1}(t,x)>-\infty,  \quad 
\lambda_k(t,x)\coloneqq -\infty
\ \ \text{otherwise.}
\end{aligned}
\end{equation}
Consequently,
\[
\Lambda_k(t,x)=\sum_{i=1}^k\lambda_i(t,x),
\qquad
\lambda_k(\nu):=\int \lambda_k(t,x)\,d\nu, 
\qquad 
\Lambda_k(\nu)=\sum_{i=1}^k\lambda_i(\nu).
\]

Next step is to show that these new approach to introduce Lyapunov exponents with respect to the initial law of a random linear bundle morphisms rescats the classical Lyapunov exponents using the driving measure.

\subsubsection{Representation and variational formula}
In the invertible case, the induced cocycle on the projective bundle associated with the \(k\)-th exterior power is globally defined, and the corresponding logarithmic potential is finite. The singular case is more delicate: kernel directions admit no projective image, and the associated logarithmic potential becomes singular on the degenerate locus. We resolve this by adjoining a \emph{cemetery section} and extending the induced projective dynamics by sending every degenerate direction to this equivariant section. More precisely, let $\hat{\mathcal E}_k$ be the projective bundle \(\PP(\wedge^k\mathcal E)\) with a \emph{cemetery section}  \(\Delta=\sqcup_{x\in X} \Delta_x \) adjoined when $F$ is singular. We denote by $\pi_k:\hat{\mathcal E}_k \to X $ the bundle projection.  We write
$$ \hat Z_k:=T\times \hat{\mathcal E}_k, \qquad
\pi_k(t,x,[v])=(t,x).$$
Define the random projective bundle map
$\hat F_k:T\times \hat Z_k\to \hat Z_k$
covering \(\bar f\) by
\begin{equation} \label{eq:random-bundle-map}
\hat F_{k,s}(t,x,[v])=(s,f_t(x),[\wedge^kF_{t,x}v]),
\qquad s\in T,
\end{equation}
with  
$[\wedge^kF_{t,x}v]=\Delta_{f_t(x)}$ if $v\in \mathrm{ker}(\wedge^kF_{t,x})$ or $[v]=\Delta_x$.  The set of points \((t,x,[v])\) such that \(\wedge^kF_{t,x} v=0\), together with the cemetery section \(T\times \Delta\) are called \emph{degeneracy locus} of $\hat F_k$.
Associated with the Markov place-dependent random iteration of \(\hat F_k\), we obtain the Markov operator \(\hat P_k:B(\hat Z_k)\to B(\hat Z_k)\),
\begin{equation}\label{eq:fully-place-Phat}
\hat P_k\psi(t,x,[v])
=
\int \psi\bigl(\hat F_{k,s}(t,x,[v])\bigr)
\,Q_{f_t(x)}(t,ds).
\end{equation}
We also define the potential
\begin{equation}\label{eq:fully-place-phihat}
\hat\phi_k(t,x,[v])
=
\log\frac{\|\wedge^kF_{t,x}v\|}{\|v\|}
\in[-\infty,\infty),
\end{equation}
with the convention that \(\hat\phi_k(t,x,[v])=-\infty\) if \(v\in\ker(\wedge^kF_{t,x})\) or \([v]=\Delta_x\). 

\begin{mainthm}\label{mainthm:B}
Let \(F:T\times\mathcal E\to\mathcal E\) be a random linear bundle morphism covering a random map \(f:T\times X\to X\). Consider a measurable family \(Q_x(t,A)\), $x\in X$,  of transition probability kernels on \(T\times\mathscr A\), and let \(\nu\) be a \(P^*\)-invariant probability measure on \(T\times X\) satisfying~\eqref{eq:integrability-nu}, where \(P\) is defined by~\eqref{eq:fully-place-P}. If \(m\) is the induced driving measure on \(\Omega\times X\), then
\[
\lambda_k(\omega,x)=\lambda_k(\omega_0,x)
\quad\text{for \(m\)-a.e.\ \((\omega,x)\),}
\quad
\lambda_k(m)=\lambda_k(\nu),
\quad
k=1,\dots,d,
\]
and
\[
\sum_{i=1}^k\lambda_i(m) =\Lambda_k(\nu)
=
\sup\biggl\{
\int 
\hat\phi_k
\,d\hat\nu
:\ \hat\nu\in\mathcal I(\hat P_k),\ (\pi_k)_*\hat\nu=\nu
\biggr\},
\]
where $\hat P_k$ and $\hat \phi_k$ are defined in~\eqref{eq:fully-place-Phat} and~\eqref{eq:fully-place-phihat}, respectively. Moreover, the supremum is attained on an ergodic \((\hat P_k)^*\)-invariant lift of \(\nu\), and any maximizing measure has zero mass on the degeneracy locus whenever \(\lambda_k(m)>-\infty\).
\end{mainthm}

Theorem~\ref{mainthm:B} identifies \(\sum_{i=1}^k\lambda_i(m)\) with the growth rate \(\Lambda_k(\nu)\). Combining this identification with Proposition~\ref{mainprop:kingman1}  yields the following asymptotic representation. Here the expectation $\mathbb{E}[\cdot]$ and the conditional expectation $\mathbb{E}[\cdot | \cdot ]$ are taken with respect to the driving measure $m$.

\begin{maincor}\label{maincor:B}
Under the assumptions of Theorem~\ref{mainthm:B}, if \(\nu\) is ergodic, then
\begin{align*}
\sum_{i=1}^k &\lambda_i(m)
=
\lim_{n\to\infty}\frac1n
\mathbb E\biggl[
\sup_{[v]\in\PP(\wedge^k\mathcal E_x)}
\log \frac{\|\wedge^kF^n_{(\omega,x)}v\|}{\|v\|}
\biggr]
\\[0.25cm]
&=
\lim_{n\to\infty}
\sup_{[v]\in\PP(\wedge^k\mathcal E_x)}
\frac1n
\mathbb E\biggl[
\log\frac{\|\wedge^kF^n_{(\omega,x)}v\|}{\|v\|}
\,\biggm|\,
(\omega_0,x,[v])
\biggr]
\quad\text{for \(\nu\)-a.e.\ \((\omega_0,x)\).}
\end{align*}
\end{maincor}

The asymptotic representation in Corollary~\ref{maincor:B} entails a notable exchange of supremum and expectation: the asymptotic growth rate obtained by first maximizing over directions and then averaging over the noise equals the rate obtained by first averaging along individual directions and then maximizing. This exchange is nontrivial, the sup and the expectation do not commute in general, but here it emerges as a direct consequence of the variational structure. Moreover, it replaces the evaluation of operator norms at each iterate with the tracking of individual directions under conditional expectations, a formulation naturally suited to Monte Carlo approximation (cf.~\cite{EckmannRuelle85}).

\subsection{Further consequences and organization}

The preceding results admit several complementary formulations and
specializations. First, replacing projective directions in
\(P(\wedge^k\mathcal E_x)\) by \(k\)-dimensional subspaces of
\(\mathcal E_x\) gives an equivalent Grassmannian variational
principle in terms of the Jacobian. 
This formulation is stated and proved in
Proposition~\ref{cor:met-linear-grassmann}.  

In the Bernoulli case, the invariant law on the state-phase space has
the form \(p\times\mu\), where \(\mu\) is stationary on \(X\).
Theorem~\ref{mainthm:B} then reduces to a variational principle over
stationary projective lifts of \(\mu\). It also yields annealed
pointwise representations and recovers the classical independence of
the Lyapunov spectrum from the noise realization. These consequences
are stated in Proposition~\ref{prop:bernoulli-variational} and
Corollary~\ref{cor:bernoulli-sup-int}. 

For edge-Markovian systems, the map and the bundle morphism depend on
a transition \(t\to s\), rather than only on the current state. By
coding edges as vertices and comparing the corresponding invariant
lifts, we obtain a variational principle involving an annealed
projective potential. This is the content of
Proposition~\ref{prop:edge-variational}.

Finally, when the phase space consists of one point, the preceding
results give Furstenberg-type formulas for iid, vertex-Markov, and
edge-Markov products of possibly singular matrices. These formulas
hold in arbitrary dimension under the standard logarithmic moment
condition.

\subsubsection{Organization of the paper.}
\S\ref{sec:proof-thm-B} proves Proposition~\ref{mainprop:kingman1},
Theorem~\ref{mainthm:B}, and the Grassmannian formulation in
Proposition~\ref{cor:met-linear-grassmann}. \S\ref{sec:Prop-IV} treats the
Bernoulli case, including Proposition~\ref{prop:bernoulli-variational}
and Corollary~\ref{cor:bernoulli-sup-int}. \S\ref{sec:Prop-VI} proves the
edge-Markovian result, Proposition~\ref{prop:edge-variational}. In \S\ref{s:linear} we specify the previous results to the classical framework of linear cocycles. The abstract variational principles are proved in~\S\ref{sec:ThmA}.
Appendix~\ref{ss:lyapunov-spectrum} recalls the classical pointwise Lyapunov spectrum in the
nonergodic case. Appendix~\ref{appendix:Young-measure} develops the required Young-measure
topology on standard Borel bundles with compact fibers.
Appendix~\ref{appendix-A} contains the uniform counterpart of
Theorem~\ref{thm:VP}, and Appendix~\ref{app:dual-min-growth} establishes the dual
variational principle for minimal growth.

\section{Lyapunov exponents for random bundle morphisms} \label{sec:proof-thm-B}

In this section we prove Proposition~\ref{mainprop:kingman1} and
Theorem~\ref{mainthm:B}, and then establish the Grassmannian variational formulation 
in Proposition~\ref{cor:met-linear-grassmann}. We first verify the
factorization and regularity properties of the projective lift. 

\subsection{Factorization and regularity}
{We first establish the factorization property for operators defined by transition kernels, thereby extending Lemma~\ref{lem:markov-factorization} to the random-map setting. We begin by recalling the notion of a kernel: given measurable spaces $(T,\mathscr{A})$ and $(X,\mathscr{B})$, a
\emph{transition probability kernel} is a map $\kappa:X\times\mathscr{A}\to[0,1]$ such that
$\kappa(x,\cdot)$ is a probability measure on $(T,\mathscr{A})$ for every $x\in X$, and
$x\mapsto \kappa(x,A)$ is measurable for every $A\in\mathscr{A}$.}

\begin{lemma}\label{lem:markov-factor-preserves-Fpi}
Let $\pi:\hat X\to X$ be a surjective measurable map between standard Borel spaces, and let
$(T,\mathscr{A})$ be a measurable space. Let $\kappa(x,\cdot)$ be a transition probability kernel
on $X\times\mathscr{A}$. Let $f:T\times X\to X$ and $F:T\times\hat X\to\hat X$ be random maps such that
$\pi\circ F_t=f_t\circ\pi$ for all $t\in T$. Define operators on $B(X)$ and $B(\hat X)$ by
\[
(P\varphi)(x)=\int \varphi\bigl(f_t(x)\bigr)\,\kappa(x,dt),
\qquad
(\hat P\psi)(\hat x)=\int \psi\bigl(F_t(\hat x)\bigr)\,\kappa(\pi(\hat x),dt).
\]
Then $P$ and $\hat P$ are $\sigma$-Markov operators and $\hat P\circ\pi^*=\pi^*\circ P$.
\end{lemma}

\begin{proof}
For $x\in X$ and $\hat x\in\hat X$ let $f^x:T\to X$ and $F^{\hat x}:T\to\hat X$ be the evaluation maps
$f^x(t)=f_t(x)$ and $F^{\hat x}(t)=F_t(\hat x)$. Define probability kernels on $X$ and $\hat X$ by
$q(x,\cdot):=(f^x)_*\kappa(x,\cdot)$ and $\hat q(\hat x,\cdot):=(F^{\hat x})_*\kappa(\pi(\hat x),\cdot)$.
Then, by change of variables,
\[
P\varphi(x)=\int \varphi(y)\,q(x,dy),\qquad
\hat P\psi(\hat x)=\int \psi(\hat y)\,\hat q(\hat x,d\hat y),
\]
so $P$ and $\hat P$ are $\sigma$-Markov by Lemma~\ref{lem:kernel-representation}. {To verify the intertwining relation, let $\varphi\in B(X)$ and set $x=\pi(\hat x)$. Using $\pi\circ F_t=f_t\circ\pi$,}
\[
(\hat P\pi^*\varphi)(\hat x)
=\int (\varphi\circ\pi)\bigl(F_t(\hat x)\bigr)\,\kappa(x,dt)
=\int \varphi\bigl(f_t(x)\bigr)\,\kappa(x,dt)
=(\pi^*P\varphi)(\hat x),
\]
hence $\hat P\circ\pi^*=\pi^*\circ P$.
\end{proof}

We now verify the hypotheses needed to apply Theorem~\ref{thm:VP} to the random bundle morphisms introduced in \S\ref{s:random-morphisms}. 
%
To do this, we keep the notations $f:T\times X \to X$, $F:T\times \mathcal{E}\to \mathcal{E}$, $Z=T\times X$,  $\hat{\mathcal{E}}_k$ and $\hat Z_k:=T\times \hat{\mathcal E}_k$ and ${\pi}_k:\hat Z_k\to Z$. The random bundle morphism $\hat{F}_{k}$, the operators $P$, $\hat P_k$ and the potential $\hat\phi_k$
are defined by~\eqref{eq:fully-place-P},~\eqref{eq:random-bundle-map}, \eqref{eq:fully-place-Phat} and \eqref{eq:fully-place-phihat}, respectively. To apply Lemma~\ref{lem:markov-factor-preserves-Fpi}, we consider the kernel 
$\kappa((t,x),A):=Q_{f_t(x)}(t,A)$ on $Z\times \mathscr{A}$.  Clearly,
\begin{align*}
    P\varphi(t,x) &=\int \varphi(\bar{f}_s(t,x))\, \kappa((t,x),ds) \ \ \text{and} \\
\hat P_k\psi(t,x,[v]) &=\int \psi(\hat F_{k,s}(t,x,[v]))\, \kappa(\pi_k(t,x,[v]),ds). 
\end{align*}
Moreover, since ${\pi}_k\circ \hat{F}_{k,s}=\bar f_s\circ {\pi}_k$ for all $s\in T$, Lemma~\ref{lem:markov-factor-preserves-Fpi} yields the factorization
\begin{equation}\label{eq:Pk-factorization}
\hat P_k\circ {\pi}_k^*= {\pi}_k^*\circ P.
\end{equation}
The following lemma verifies for the data \(({\pi}_k,\hat P_k,P,\hat\phi_k,\nu)\) the regularity assumptions of Theorem~\ref{thm:VP}. In particular, once proved, all the hypotheses of Theorem~\ref{thm:VP} hold. 


\enlargethispage{0.2cm}
\begin{lemma}\label{lem:regularity-for-VP}
{Let
\(
K_k := \{(t,x,[v]) : \wedge^k F_{t,x}v = 0\}\cup \Delta
\)
be the degeneracy locus. Then}
\begin{enumerate}[label=(\roman*), leftmargin=0.9cm]
\item For every $\psi\in\mathscr C_b(\hat Z_k)$, every fiberwise discontinuity point of $\hat P_k\psi$ belongs to $K_k$;
\item $\hat\phi_k$ is fiberwise upper semicontinuous and $\hat\phi_k|_{K_k}= -\infty$;
\item $(\mathcal M\hat\phi_k)^+\in L^1(\nu)$.
\end{enumerate}
\end{lemma}

\begin{proof}
Fix $(t,x)\in Z$ and identify the fiber $\pi_k^{-1}(t,x)$ with $(\hat{\mathcal E}_{k})_x$. On $(\hat{\mathcal E}_{k})_x$ we have
\[
\hat\phi_k(t,x,[v])=\log\frac{\|\wedge^kF_{t,x}v\|}{\|v\|},
\]
with the convention $\log 0=-\infty$, and $\hat\phi_k(t,x,\Delta_x)=-\infty$ in the singular case. The map
$[v]\mapsto \|\wedge^kF_{t,x}v\|/\|v\|$ is continuous on $(\hat{\mathcal E}_{k})_x\setminus\ker(\wedge^kF_{t,x})$ and vanishes
on $\ker(\wedge^kF_{t,x})$, hence $[v]\mapsto \hat\phi_k(t,x,[v])$ is upper semicontinuous on the whole fiber. Moreover, by definition, $\hat\phi_k=-\infty$ on $K_k$, proving (ii).

We also have $\hat\phi_k(t,x,[v])\le \log\|\wedge^kF_{t,x}\|\le k\log^+\|F_{t,x}\|$ for all $(t,x,[v])$. Therefore \(\mathcal M\hat\phi_k(t,x)\le \log\|\wedge^kF_{t,x}\|\), and hence
\[
(\mathcal M\hat\phi_k)^+(t,x)\le \log^+\|\wedge^kF_{t,x}\|\le k\log^+\|F_{t,x}\|.
\]
The right-hand side is $\nu$-integrable by \eqref{eq:integrability-nu}, so $(\mathcal M\hat\phi_k)^+\in L^1(\nu)$ and we get~(iii).

{It remains to prove (i).} Let $\psi\in\mathscr C_b(\hat Z_k)$ and fix $(t,x)\in Z$. For $[v]\in (\hat{\mathcal E}_{k})_x$ with
$(t,x,[v])\notin K_k$, we have $\wedge^kF_{t,x}v\neq 0$ and the projective map
\[
\Theta_{t,x}:(\hat{\mathcal E}_{k})_x\setminus\ker(\wedge^kF_{t,x})\to (\hat{\mathcal E}_{k})_{f_t(x)},\qquad
\Theta_{t,x}([v])=[\wedge^kF_{t,x}v],
\]
is continuous. If $[v_n]\to[v]$ in the fiber with $(t,x,[v])\notin K_k$, then $\Theta_{t,x}([v_n])\to\Theta_{t,x}([v])$.
Using~\eqref{eq:fully-place-Phat},
\[
(\hat P_k\psi)(t,x,[v])=\int \psi\bigl(s,f_t(x),\Theta_{t,x}([v])\bigr)\,Q_{f_t(x)}(t,ds).
\]
For each fixed $s\in T$, the map $\zeta\mapsto \psi(s,f_t(x),\zeta)$ is continuous on the fiber over $f_t(x)$
because $\psi$ is Carath\'eodory, hence the integrand converges pointwise in $s$ along $[v_n]\to[v]$.
Since $\psi$ is bounded, dominated convergence yields $(\hat P_k\psi)(t,x,[v_n])\to(\hat P_k\psi)(t,x,[v])$.
Thus \(\hat P_k\psi\) is continuous on \(\pi_k^{-1}(t,x)\setminus K_k\). Moreover,
\[
K_k\cap\pi_k^{-1}(t,x)
=
\PP\bigl(\ker(\wedge^kF_{t,x})\bigr)\sqcup\{\Delta_x\}.
\]
The projectivized kernel is closed and \(\Delta_x\) is isolated; hence
\(K_k\cap\pi_k^{-1}(t,x)\) is closed in the fiber. Thus, $\pi_k^{-1}(t,x)\setminus K_k$  is open and since \(\hat P_k\psi\) is
fiberwise continuous at every point outside \(K_k\), it follows that its fiberwise discontinuity points belong to $K_k$ 
which proves~\textup{(i)}.
\end{proof}

\subsection{Proof of Proposition~\ref{mainprop:kingman1}}
\label{ss:proof-mainprop-kingman1}

Fix \(k\in\{1,\dots,d\}\). For \(n\geq1\), define
\begin{equation}\label{eq:def-projective-maximal-process}
\hat\phi_{k,n}
:=
\sum_{j=0}^{n-1}\hat P_k^{\,j}\hat\phi_k,
\qquad
\phi_{k,n}
:=
\mathcal M\hat\phi_{k,n}.
\end{equation}
The sequence \((\hat\phi_{k,n})_{n\geq1}\) is \(\hat P_k\)-additive. Together with the factorization $\hat P_k\circ\pi_k^*=\pi_k^*\circ P$
obtained in~\eqref{eq:Pk-factorization}, Theorem~\ref{mainthm:A} therefore
implies that \((\phi_{k,n})_{n\geq1}\) is \(P\)-subadditive.

We next identify the functions \(\hat\phi_{k,n}\) and \(\phi_{k,n}\) in terms
of finite-time directional growth. To facilitate the proof, we adopt the probabilistic interpretation of~\eqref{eq:kernel-q} as a Markov chain $(Z_n)_{n\ge0}$ on $Z$ with transition kernel $q(z,C)$ as defined in~\eqref{eq:Markov-chain}. Accordingly, we denote by $\mathbb{E}[\ \cdot \ | \ Z_0=(t,x) ]$ the expectation with respect to the path measure of the chain starting at $(\omega_0,x_0)=(t,x)$. This represents the average value over all possible future paths (trajectories) of the noise, specifically conditioned on the fact that we start at the specific configuration $(\omega_0,x_0)=(t,x)$. Similarly, let $(\tilde Z_n)_{n\ge 0}$ denote the lifted Markov chain on the projective state space $\hat Z_k=T\times \hat{\mathcal{E}}_k$.  That is, \(\tilde Z_j=(\omega_j,x_j,[v_j])\) with \(\tilde Z_0=(t,x,[v])\) driven by the random bundle morphisms \(\hat F_{k,s}\), where the parameter \(s\) corresponds to the next noise state. More precisely,
$\tilde Z_{j+1}=\hat F_{k,\omega_{j+1}}(\tilde Z_j)$,
where \(\omega_{j+1}\) is chosen according to \(Q_{x_{j+1}}(\omega_j,\cdot)\) and \(x_{j+1}=f_{\omega_j}(x_j)\).  With this notation, for any measurable function $g$ bounded from above and any $j\geq 1$, the Markov operator satisfies
$$
(\hat P_k^j g)(z) = \mathbb{E}\big[g(\tilde Z_j) \mid \tilde Z_0 = z\big], \quad \text{for } z=(t,x,[v]).
$$
The kernel representation of $\hat P_k$ gives the same identity for nonnegative measurable functions and, for
extended-real-valued functions whose positive part has finite
conditional expectation. We will use as $g$ the function $\hat \phi_k$, and to simplify notation we admit the convention that
\[
\log\frac{\|\wedge^kF^n_{(\omega,x)}v\|}{\|v\|}=-\infty
\]
whenever either \(\tilde Z_j=(\omega_j,x_j,\Delta_{x_j})\) for some \(0\le j\le n-1\), or \(\wedge^kF^n_{(\omega,x)}v=0\).

\begin{lemma} \label{lem:finite-horizon-expansion}
For every $(t,x,[v])\in \hat Z_k$ and $n\geq 1$, we have
\begin{align*}
 \hat\phi_{k,n}(t,x,[v]) &=\mathbb{E}\bigg[ \log\frac{\|\wedge^k F^n_{(\omega,x)}v\|}{\|v\|} \ \Big| \ \tilde Z_0=(t,x,[v])\bigg] =\int
\log\frac{\bigl\|\wedge^k F^n_{(\omega,x)}v\big\|}{\|v\|}
\; dQ^{n}_{t,x}
\end{align*}
where $Q^n_{t,x}$ is the probability measure on $T^{n-1}$ {generated along the random orbit segment \(x_i=f_{\omega_{i-1}}(x_{i-1})\), \(i=1,\dots,n-1\), starting at \((\omega_0,x_0)=(t,x)\), namely}
\begin{equation} \label{eq:Qn}
    dQ^{n}_{t,x}(\omega_1,\dots,\omega_{n-1})
:=
Q_{x_1}(t,d\omega_1)\,Q_{x_2}(\omega_1,d\omega_2)\cdots Q_{x_{n-1}}(\omega_{n-2},d\omega_{n-1}).
\end{equation}
\end{lemma}

\begin{proof}
By definition of \(\hat\phi_{k,n}\) and of the Markov operator \(\hat P_k\), for \(z=(t,x,[v])\),
\[
\hat\phi_{k,n}(z)
=
\sum_{j=0}^{n-1} \mathbb{E}\big[\hat\phi_k(\tilde Z_j) \bigm| \tilde Z_0 = z\big]
=
\mathbb{E}\bigg[ \sum_{j=0}^{n-1} \hat\phi_k(\tilde Z_j) \biggm| \tilde Z_0 = z \bigg].
\]

If \(\tilde Z_j\) hits the cemetery for some \(0\le j\le n-1\), then \(\hat\phi_k(\tilde Z_j)=-\infty\), and therefore
\[
\sum_{j=0}^{n-1}\hat\phi_k(\tilde Z_j)=-\infty,
\]
which agrees with the above convention. Thus, we may assume that \(\tilde Z_j=(\omega_j,x_j,[v_j])\) lies in the projective part for every \(0\le j\le n-1\). In this case,
\[
\hat\phi_k(\tilde Z_j)
=
\log \frac{\| \wedge^k F_{\omega_j, x_j} v_j \|}{\|v_j\|}
=
\log \frac{\|v_{j+1}\|}{\|v_j\|}.
\]
If \(v_n=0\), then
\[
\hat\phi_k(\tilde Z_{n-1})
=
\log \frac{\|v_n\|}{\|v_{n-1}\|}
=
-\infty,
\]
so again
\[
\sum_{j=0}^{n-1}\hat\phi_k(\tilde Z_j)=-\infty,
\]
which agrees with the convention. Otherwise, \(v_n\neq0\), and the sum inside the expectation telescopes:
\[
\sum_{j=0}^{n-1} \log \frac{\|v_{j+1}\|}{\|v_j\|}
=
\log \frac{\|v_n\|}{\|v_0\|}
=
\log \frac{\|\wedge^k F^n_{(\omega,x)}v\|}{\|v\|}.
\]
Hence, in all cases,
\[
\sum_{j=0}^{n-1}\hat\phi_k(\tilde Z_j)
=
\log \frac{\|\wedge^k F^n_{(\omega,x)}v\|}{\|v\|},
\]
which proves the first equality.

For the second equality, we compute the expectation using the transition kernel of the base chain \((Z_j)_{j\ge0}\), where \(Z_j=(\omega_j,x_j)\). Recall that the kernel on \(Z\) is
\[
q((t,x),ds\times dy)=Q_{f_t(x)}(t,ds)\,\delta_{f_t(x)}(dy).
\]
Thus, conditioned on the initial state \(Z_0=(t,x)\) and the noise path \(\omega=(\omega_0,\dots,\omega_{n-1})\), the spatial trajectory is deterministic, namely \(x_{j+1}=f_{\omega_j}(x_j)\). Therefore, the expectation over future paths reduces to integration over the noise coordinates:
\begin{align*}
\mathbb{E}\big[ \Phi(\omega, x) \mid \tilde Z_0=z\big]
&=
\int  \Phi(\omega, x) \, q(Z_0, dZ_1) \cdots q(Z_{n-2}, dZ_{n-1}) \\
&=
\int_{T^{n-1}} \Phi(\omega, x) \, Q_{x_1}(t, d\omega_1) \cdots Q_{x_{n-1}}(\omega_{n-2}, d\omega_{n-1}).
\end{align*}
By the definition in~\eqref{eq:Qn}, this product measure is exactly \(Q^n_{t,x}\). Substituting
$$
\Phi(\omega,x)=\log\frac{\|\wedge^k F^n_{(\omega,x)}v\|}{\|v\|}
$$
completes the proof.
\end{proof}

For \(n=1\), the fiberwise maximum is $\phi_{k,1}(t,x)
=
\mathcal M\hat\phi_k(t,x)
=
\log\|\wedge^kF_{t,x}\|$.
Consequently,
\[
\phi_{k,1}^+(t,x)
=
\log^+\|\wedge^kF_{t,x}\|
\leq
k\log^+\|F_{t,x}\|,
\]
and hence \(\phi_{k,1}^+\in L^1(\nu)\) by~\eqref{eq:integrability-nu}. We may therefore apply Kingman's subadditive theorem for Markov
operators~\cite[Thm.~D]{BM} to the \(P\)-subadditive sequence
\((\phi_{k,n})_{n\geq1}\). It gives a \(P\)-invariant measurable function
\[
\Lambda_k:Z\longrightarrow[-\infty,\infty)
\]
with \(\Lambda_k^+\in L^1(\nu)\) such that
\begin{equation}\label{eq:pointwise-kingman-projective}
\Lambda_k(t,x)
=
\lim_{n\to\infty}\frac1n\phi_{k,n}(t,x)
\qquad
\text{for \(\nu\)-a.e. \((t,x)\),}
\end{equation}
and
\begin{equation}\label{eq:integrated-kingman-projective}
\int\Lambda_k\,d\nu
=
\lim_{n\to\infty}\frac1n\int\phi_{k,n}\,d\nu
=
\inf_{n\geq1}\frac1n\int\phi_{k,n}\,d\nu.
\end{equation}
We set
\[
\Lambda_k(\nu):=\int\Lambda_k\,d\nu.
\]
From Lemma~\ref{lem:finite-horizon-expansion} and
\eqref{eq:pointwise-kingman-projective} we get
\[
\Lambda_k(t,x)
=
\lim_{n\to\infty}
\sup_{[v]\in\PP(\wedge^k\mathcal E_x)}
\frac1n
\mathbb E\biggl[
\log\frac{\|\wedge^kF^n_{(\omega,x)}v\|}{\|v\|}
\,\biggm|\,
(\omega_0,x,[v]) 
\biggr]\qquad \text{for \(\nu\)-a.e.~\((\omega_0,x)\).}
\]
 Likewise,
\eqref{eq:integrated-kingman-projective} gives
\[
\Lambda_k(\nu)
=
\lim_{n\to\infty}\frac1n
\int
\sup_{[v]\in\PP(\wedge^k\mathcal E_x)}
\mathbb E\biggl[
\log\frac{\|\wedge^kF^n_{(\omega,x)}v\|}{\|v\|}
\,\biggm|\,
(\omega_0,x,[v])
\biggr]
\,d\nu.
\]
These are the first two assertions of
Proposition~\ref{mainprop:kingman1}.

If \(\nu\) is ergodic, the \(P\)-invariant function \(\Lambda_k\) is
constant \(\nu\)-almost everywhere. Its constant value is its integral,
namely \(\Lambda_k(\nu)\). Therefore, $\Lambda_k(t,x)=\Lambda_k(\nu)$ for \(\nu\)-a.e.~\((t,x)\),
which proves the final assertion of Proposition~\ref{mainprop:kingman1}.

\subsection{Proof of Theorem~\ref{mainthm:B}} {Fix \(k\in\{1,\dots,d\}\). To prove the theorem, we first show that \(\Sigma_k(m)=\Lambda_k(\nu)\). We use the following descriptions of these quantities:}
\begin{align*}
     \Sigma_k(m) = \lim_{n\to \infty} \frac{1}{n} \int \log \|\wedge^k F^n_{(\omega,x)}\|\, dm
\ \ \text{and}  \ \ 
\Lambda_k(\nu) =\sup \big\{\int \hat\phi_k\, d\hat\nu:  \hat\nu \in \mathcal{I}_\nu(\hat P_k)\big\}.
\end{align*}
Proposition~\ref{mainprop:kingman1} identifies $\Lambda_k(\nu)$
with the growth rate of the $P$-subadditive sequence
$(\phi_{k,n})_{n\geq1}$. By
Lemma~\ref{lem:regularity-for-VP} and
Theorem~\ref{thm:VP}, this growth rate admits the variational
representation above.



\subsubsection{Inequality $\Lambda_k(\nu)\leq \Sigma_k(m)$}

Now we fix $\hat \nu \in \mathcal{I}_\nu(\hat P_k)$. From Lemma~\ref{lem:finite-horizon-expansion}, we have the pointwise inequality
\[
\hat\phi_{k,n}(t,x,[v]) \le \mathbb{E}\big[\Psi_n \mid Z_0=(t,x)\big] \qquad \text{for } (t,x,[v]) \in \hat Z_k,
\]
where $\Psi_n$ is the random variable corresponding to $\log\|\wedge^k F^{n}_{(\omega,x)}\|$. {Here \(\Psi_n\) depends on the noise path \(\omega\) and the initial position \(x\).} Integrating the inequality against $\hat\nu$ and noting that the right-hand side does not depend on the fiber coordinate $[v]$, we use the projection property $(\pi_k)_*\hat\nu = \nu$ to obtain
\begin{equation}\label{eq:proof-ineq-1}
\int \hat\phi_{k,n} \, d\hat\nu \le \int \mathbb{E}\big[\Psi_n \mid Z_0=z\big] \, d\nu.
\end{equation}
{We now identify the right-hand side with the integral against the driving measure \(m\). Recall that \(\tilde{\mathbb P}\) is the Markov measure on the path space \(\tilde\Omega=Z^{\mathbb N}\) with initial law \(\nu\) and transition kernel \(q(z,C)\). By the defining property of the Markov measure, integrating the conditional expectation against the initial law yields the global expectation:}
\[
\int \mathbb{E}\big[\Psi_n \mid Z_0=z\big] \, d\nu = \int (\Psi_n \circ \tilde{\pi}) \, d\tilde{\mathbb{P}},
\]
where we view $\Psi_n$ as a function on $\Omega \times X$ and lift it to $\tilde{\Omega}=Z^{\mathbb{N}}$ via the projection $\tilde{\pi}((t_i, x_i)_{i\ge 0}) = ((t_i)_{i\ge 0}, x_0)$. {Finally, since the driving measure is defined by}
$
m = \tilde{\pi}_* \tilde{\mathbb{P}}$, 
we get
\[
\int (\Psi_n \circ \tilde{\pi}) \, d\tilde{\mathbb{P}} = \int \Psi_n \, dm = \int \log\|\wedge^k F^{n}_{(\omega,x)}\| \, dm.
\]
Combining this with~\eqref{eq:proof-ineq-1} and observing that $\int \hat\phi_{k,n}\,d\hat\nu = n \int \hat\phi_k\,d\hat\nu$ because of $\hat\nu \in \mathcal{I}(\hat P_k)$, we obtain
\[
\int \hat\phi_k\,d\hat\nu \leq \frac{1}{n} \int \log\|\wedge^k F^{n}_{(\omega,x)}\| \, dm.
\]
{Letting \(n\to\infty\), the right-hand side converges to \(\Sigma_k(m)\). Hence,
$\int \hat\phi_k\,d\hat\nu\le \Sigma_k(m)$
for every invariant lift \(\hat\nu\). Taking the supremum over \(\hat\nu\in\mathcal I_\nu(\hat P_k)\) yields
$
\Lambda_k(\nu)\le \Sigma_k(m)$.}

\subsubsection{Inequality $\Sigma_k(m)\leq\Lambda_k(\nu)$}

If $\Sigma_k(m)=-\infty$, there is nothing to prove. Assume from now on that
$\Sigma_k(m)>-\infty$. The reverse inequality will follow from an elementary
averaging estimate on finite-dimensional projective spaces.

We first introduce the probability measure used in that estimate. Let $V$ be
a $D$-dimensional Euclidean space and denote by
$S(V):=\{v\in V:\|v\|=1\}$
its unit sphere. Let $\sigma_V$ be the normalized Riemannian volume measure on
$S(V)$, so that $\sigma_V(S(V))=1$, and consider the canonical quotient map
$q_V:S(V)\longrightarrow P(V)$ given by $q_V(v)=[v]$. We define
\begin{equation}\label{eq:def-projective-spherical-measure}
\rho_V:=(q_V)_*\sigma_V.
\end{equation}
Thus, for every bounded Borel function $g:P(V)\to\mathbb R$,
\[
\int_{P(V)}g([v])\,d\rho_V([v])
=
\int_{S(V)}g([v])\,d\sigma_V(v).
\]
The measure $\rho_V$ is a Borel probability measure on $P(V)$. It is invariant
under the natural action of the orthogonal group $O(V)$, because $\sigma_V$ is
$O(V)$-invariant. In particular, $\rho_V$ does not depend on the choice of an
orthonormal basis of $V$. Equivalently, it is the canonical probability measure
on the compact homogeneous space
\[
P(V)\simeq O(V)/(O(1)\times O(D-1)).
\]
This is the projective measure obtained from normalized spherical measure; see,
for instance, \cite[Chapter~3]{Mattila1995} for the corresponding spherical
measure construction.

\begin{lemma}\label{lem:projective-averaging}
For every integer $D\geq1$, there exists a constant $C_D<\infty$ with the
following property. If $V$ and $V'$ are $D$-dimensional Euclidean spaces and
$L:V\to V'$ is linear, then
\begin{equation*}\label{eq:projective-averaging}
\log\|L\|
\leq
C_D+
\int
\log\frac{\|Lv\|}{\|v\|}
\,d\rho_V([v]).
\end{equation*}
Here the integrand is independent of the nonzero representative $v$ of $[v]$.
If $L=0$, both sides are understood as $-\infty$.
\end{lemma}

\begin{proof}
The assertion is immediate when $L=0$, so assume that $L\neq0$. Let
\[
\|L\|=\sigma_1(L)\geq\sigma_2(L)\geq\cdots\geq\sigma_D(L)\geq0
\]
be the singular values of $L$. Thus there exist orthonormal bases
$e_1,\dots,e_D$ of $V$ and $e'_1,\dots,e'_D$ of $V'$ such that
$Le_i=\sigma_i(L)e'_i$, $i=1,\dots,D$.  Let $v\in S(V)$ and write
\[
v=\sum_{i=1}^D a_i(v)e_i.
\]
Then
\[
\|Lv\|^2
=
\sum_{i=1}^D \sigma_i(L)^2|a_i(v)|^2
\geq
\sigma_1(L)^2|a_1(v)|^2.
\]
Therefore $\|Lv\|\geq |a_1(v)|\,\|L\|$, 
and hence, with the convention $\log0=-\infty$,
\begin{equation}\label{eq:coordinate-lower-bound}
\log\|Lv\|
\geq
\log\|L\|+\log|a_1(v)|.
\end{equation}

The function $v\mapsto |a_1(v)|=|\langle v,e_1\rangle|$ is even, so it descends
to a well-defined Borel function on $P(V)$. Moreover,
\[
\int \bigl|\log|a_1(v)|\bigr|\,d\sigma_V(v)<\infty.
\]
Indeed, if $D=1$, then $|a_1(v)|=1$ and the integral is zero. If $D\geq2$,
standard spherical coordinates give a constant $A_D<\infty$ such that
\[
\sigma_V\bigl(\{v\in S(V):|a_1(v)|<r\}\bigr)\leq A_Dr,
\qquad 0<r<1.
\]
Hence,
\begin{align*}
\int_{S(V)}-\log|a_1(v)|\,d\sigma_V(v)
&=
\int_0^\infty
\sigma_V\bigl(\{|a_1|<e^{-s}\}\bigr)\,ds \leq
A_D\int_0^\infty e^{-s}\,ds
<\infty.
\end{align*}
By the orthogonal invariance of $\sigma_V$, the number
\[
c_D
:=
\int \log|\langle v,e_1\rangle|\,d\sigma_V(v)
=
\int \log|a_1([v])|\,d\rho_V([v])
\]
depends only on $D$, and $c_D>-\infty$. Integrating
\eqref{eq:coordinate-lower-bound} and using \eqref{eq:def-projective-spherical-measure},
we obtain
\[
\int 
\log\frac{\|Lv\|}{\|v\|}
\,d\rho_V([v])
\geq
\log\|L\|+c_D.
\]
The conclusion follows with $C_D:=-c_D$.
\end{proof}

We apply Lemma~\ref{lem:projective-averaging} to the linear map
$L=\wedge^kF^n_{(\omega,x)}
:
\wedge^k\mathcal E_x
\longrightarrow
\wedge^k\mathcal E_{f_\omega^n(x)}$. 
Set
\[
D_k:=\dim(\wedge^k\mathbb R^d)=\binom{d}{k}.
\]
For each $x\in X$, let $\rho_{k,x}$ be the probability measure on
$P(\wedge^k\mathcal E_x)$ defined by \eqref{eq:def-projective-spherical-measure}
using the Euclidean structure of the fiber $\wedge^k\mathcal E_x$.


Lemma~\ref{lem:projective-averaging} now gives, for every $n\geq1$ and
$m$-a.e. $(\omega,x)$,
\begin{equation}\label{eq:projective-averaging-cocycle}
\log\|\wedge^kF^n_{(\omega,x)}\|
\leq
C_{D_k}
+
\int
\log\frac{\|\wedge^kF^n_{(\omega,x)}v\|}{\|v\|}
\,d\rho_{k,x}([v]).
\end{equation}
The positive parts of the functions in
\eqref{eq:projective-averaging-cocycle} are integrable. Indeed, for every
nonzero $v\in\wedge^k\mathcal E_x$,
\begin{align*}
\log^+
\frac{\|\wedge^kF^n_{(\omega,x)}v\|}{\|v\|}
\leq
\log^+\|\wedge^kF^n_{(\omega,x)}\| \leq
k\log^+\|F^n_{(\omega,x)}\| \leq
k\sum_{j=0}^{n-1}\log^+\|F_{\omega_j,x_j}\|,
\end{align*}
where $x_j=f_\omega^j(x)$. The last function is $m$-integrable by
\eqref{eq:integrability-m} and the invariance of $m$. Consequently, conditional
expectations and Fubini's theorem may be applied to the extended-real-valued
functions appearing below.

Conditioning~\eqref{eq:projective-averaging-cocycle} on $Z_0=(\omega_0,x_0)=(t,x)$ gives
\begin{align*}
&\mathbb E\left[
\log\|\wedge^kF^n_{(\omega,x)}\|
\;\middle|\;
Z_0=(t,x)
\right] \leq
C_{D_k}
+
\int
\mathbb E\left[
\log\frac{\|\wedge^kF^n_{(\omega,x)}v\|}{\|v\|}
\;\middle|\;
Z_0=(t,x)
\right]
\,d\rho_{k,x}([v]).
\end{align*}
For every fixed $[v]\in P(\wedge^k\mathcal E_x)$, Lemma~\ref{lem:finite-horizon-expansion}
identifies the conditional expectation in the integrand with
$\hat\phi_{k,n}(t,x,[v])$. Hence
\begin{align*}
\mathbb E\left[
\log\|\wedge^kF^n_{(\omega,x)}\|
\;\middle|\;
Z_0=(t,x)
\right]
&\leq
C_{D_k}
+
\int
\hat\phi_{k,n}(t,x,[v])
\,d\rho_{k,x}([v])\\
&\leq
C_{D_k}
+
\sup_{[v]\in P(\wedge^k\mathcal E_x)}
\hat\phi_{k,n}(t,x,[v])=
C_{D_k}+\mathcal M\hat\phi_{k,n}(t,x).
\end{align*}
Integrating with respect to the initial law $\nu$ and using the construction of
the driving measure~$m$, we obtain
\begin{equation}\label{eq:integrated-projective-average}
\int
\log\|\wedge^kF^n_{(\omega,x)}\|
\,dm(\omega,x)
\leq
C_{D_k}
+
\int \mathcal M\hat\phi_{k,n}\,d\nu.
\end{equation}
Divide \eqref{eq:integrated-projective-average} by $n$ and let $n\to\infty$.
By the definition of $\Sigma_k(m)$ and of $\Lambda_k(\nu)$,
\[
\Sigma_k(m)
\leq
\lim_{n\to\infty}
\frac1n\int_Z\mathcal M\hat\phi_{k,n}\,d\nu
=
\Lambda_k(\nu).
\]
This proves the desired inequality.

\subsubsection{Lyapunov exponent and pointwise Lyapunov exponents}
\label{ss:ergodic-decomposition}
The previous subsections show that $\Sigma_k(m)=\Lambda_k(\nu)$ for every $k=1,\dots,d$. Hence, by the convention used to define Lyapunov exponents from the partial sums,
\[
\lambda_k(m)=\lambda_k(\nu)
\qquad\text{for every }k=1,\dots,d.
\]

Now we show the pointwise version. Recall that for every \(k=1,\dots,d\),
\[
\Sigma_k(\omega,x):=\lim_{n\to\infty}\frac1n\log\|\wedge^kF^n_{(\omega,x)}\|
\qquad\text{for \(m\)-a.e.\ }(\omega,x),
\]
with \(\Sigma_0(\omega,x)=0\), and
\[
\Lambda_k(t,x):=\lim_{n\to\infty}\frac1n(\mathcal M\hat\phi_{k,n})(t,x)
\qquad\text{for \(\nu\)-a.e.\ }(t,x),
\]
with \(\Lambda_0(t,x)=0\). The corresponding exponents \(\lambda_k(t,x)\)  and \(\lambda_k(\omega,x)\)  are defined 
from partial sums using $\Lambda_k(t,x)$ and $\Sigma_k(\omega,x)$  as  in~\eqref{eq:def-Lyapunov-law} and~\eqref{eq:def-lyapunov-sigma} respectively.

Let $\nu=\int \nu_\alpha\,d\eta(\alpha)$ be the ergodic decomposition of $\nu$ into $P$-ergodic
$P^*$-invariant probability measures. For each $\alpha$, let $m_\alpha$ be the driving measure on
$\Omega\times X$ induced by $\nu_\alpha$ through the same construction as $m$; by linearity of the
construction one has the corresponding decomposition
\[
m=\int m_\alpha\,d\eta(\alpha)
\qquad\text{and}\qquad
(\pi_0)_*m_\alpha=\nu_\alpha.
\]

Fix $\alpha$. Since $\nu_\alpha$ is $P$-ergodic, Kingman {subadditive} ergodic theorem for Markov operators~\cite[Thm.~D]{BM} yields that
$\Lambda_k(t,x)= \Lambda_k(\nu_\alpha)$ for $\nu_\alpha$-a.e.\ $(t,x)$.
On the other hand, the associated $m_\alpha$ is $\mathcal F$-ergodic, hence Kingman's theorem also gives
$\Sigma_k(\omega,x)= \Sigma_k(m_\alpha)$ for $m_\alpha$-a.e.\ $(\omega,x)$.
{Applying the equality of growth rates proved above, we obtain $\Sigma_k(m_\alpha)=\Lambda_k(\nu_\alpha)$ for every $k=1,\dots,d$.}
Therefore, for $m_\alpha$-a.e.\ $(\omega,x)$,
\[
\Sigma_k(\omega,x)=\Sigma_k(m_\alpha)=\Lambda_k(\nu_\alpha)=\Lambda_k(\omega_0,x),
\]
because $(\omega_0,x)$ is $\nu_\alpha$-distributed under $m_\alpha$. Integrating over \(\alpha\) yields the equality for \(m\)-a.e.\ \((\omega,x)\). The corresponding statement for \(\lambda_k\) follows by the convention used to define the exponents from the partial sums.

\subsection{Grassmannian formulation}

It is often convenient to formulate Theorem~\ref{mainthm:B} and Corollary~\ref{maincor:B} on the Grassmann bundle. Let $\pi_{G_k}:\mathrm G_k(\mathcal E)\to X$
be the bundle of \(k\)-dimensional subspaces of the fibers of \(\mathcal E\). The random linear bundle morphism \(F\) induces
\[
\hat F_{{{G}_k},s}(t,x,W)\coloneqq (s,f_t(x),F_{t,x}W)
\]
whenever \(W\in \mathrm{G}_k(\mathcal E_x)\) and
\(\dim(F_{t,x}W)=k\). We adjoin an isolated equivariant cemetery
section \(\Delta\) and send every \((t,x,W)\) with
\(\dim(F_{t,x}W)<k\) to \(T\times\Delta\). We call this set of points the \emph{degeneracy locus} of $\hat F_{G_k}$. Associated with the Markov place-dependent random iteration of \(\hat F_{G_k}\), we introduce the analogous Markov operator $\hat P_{G_k}$ to~\eqref{eq:fully-place-Phat}.
We also consider the potential
\[
\hat\phi_{G_k}(t,x,W)\coloneqq \log \mathrm{Jac}_k(F_{t,x}|_W),  \quad \text{where } \  \mathrm{Jac}_k(F_{t,x}|_W):= \|\wedge^k (F_{t,x}|_W)\|
\]
with the convention \(\hat\phi_{G_k}=-\infty\) on \(T\times\Delta\) and whenever \(\dim(F_{t,x}W)<k\).

\begin{mainprop}\label{cor:met-linear-grassmann}
In the setting of Theorem~\ref{mainthm:B},
\begin{align*}
\sum_{i=1}^k \lambda_i(m)
&=  \lim_{n\to\infty} \frac{1}{n} \int
\sup_{W\in\mathrm{G}_k(\mathcal E_x)}
\mathbb E\biggl[
\log \mathrm{Jac}_k(F^n_{(\omega,x)}|_W )
\,\biggm|\,
(\omega_0,x,W)
\biggr] \, d\nu \\
&=
\sup\left\{\int \hat \phi_{G_k}\,d\hat\nu:
\text{$\hat\nu \in \mathcal{I}(\hat P_{G_k})$, \ $(\pi_{G_k})_*\hat\nu=\nu$}\right\}
\end{align*} 
where the limit may be replaced by the infimum.
Moreover, the supremum is attained on an ergodic $\hat P_{{G}_k}$-invariant measure lift of $\nu$, and any maximizing measure has zero mass on the degeneracy locus whenever $\lambda_k(m)>-\infty$. Also, if $\nu$ is ergodic, then
\begin{align*}
\sum_{i=1}^k &\lambda_i(m)
=
\lim_{n\to\infty}
\sup_{W\in\mathrm{G}_k(\mathcal E_x)}
\frac1n
\mathbb E\biggl[
\log \mathrm{Jac}_k(F^n_{(\omega,x)}|_W)
\,\biggm|\,
(\omega_0,x,W)
\biggr]
\quad\text{for \(\nu\)-a.e.\ \((\omega_0,x)\).}
\end{align*}
\end{mainprop}

\begin{proof}[Proof of Proposition~\ref{cor:met-linear-grassmann}]

Fix $k\in\{1,\dots,d\}$. As in the statement, we use $\pi_{G_k}$ also for the projection $\pi_{G_k}:T\times \mathrm G_k(\mathcal E)\longrightarrow Z=T\times X$, $\pi_{G_k}(t,x,W)=(t,x)$,
with the isolated cemetery section adjoined. For $n\geq1$, set
\begin{equation}\label{eq:grassmann-additive-process}
\hat\phi_{G_k,n}
:=
\sum_{j=0}^{n-1}\hat P_{G_k}^{\,j}\hat\phi_{G_k},
\qquad
\phi_{G_k,n}
:=
\mathcal M\hat\phi_{G_k,n}.
\end{equation}
Thus,
\[
\phi_{G_k,n}(t,x)
=
\sup_{W\in \mathrm G_k(\mathcal E_x)}
\hat\phi_{G_k,n}(t,x,W).
\]
We first verify that Theorem~\ref{thm:VP} applies to the Grassmannian lift and then identify the corresponding growth rate with
\[
\Sigma_k(m):=\sum_{i=1}^k\lambda_i(m).
\]

\paragraph{\emph{The Grassmannian growth rate}}
Recall that
$$ \hat F_{G_k,s}(t,x,W)
=
(s,f_t(x),F_{t,x}W)$$
whenever $\dim(F_{t,x}W)=k$, and that the image is the cemetery section otherwise. Since
\[
\pi_{G_k}\circ\hat F_{G_k,s}
=
\bar f_s\circ\pi_{G_k},
\qquad
\bar f_s(t,x)=(s,f_t(x)),
\]
Lemma~\ref{lem:markov-factor-preserves-Fpi} gives
\begin{equation}\label{eq:grassmann-factorization}
\hat P_{G_k}\circ\pi_{G_k}^*
=
\pi_{G_k}^*\circ P.
\end{equation}

Let \(K_{G_k}\) denote the degeneracy locus of \(\hat F_{G_k}\),
namely the set of points \((t,x,W)\) for which
\(\dim(F_{t,x}W)<k\), together with the cemetery section. The
regularity verification is the same as in
Lemma~\ref{lem:regularity-for-VP}, but we record it here for
completeness. Fix \((t,x)\in Z\). On the fiber
\(\mathrm G_k(\mathcal E_x)\), the map
$ W\mapsto\operatorname{Jac}_k(F_{t,x}|_W)$
is continuous and nonnegative. Therefore
\[
K_{G_k}\cap\pi_{G_k}^{-1}(t,x)
=
\left\{
W:\operatorname{Jac}_k(F_{t,x}|_W)=0
\right\}
\sqcup\{\Delta_x\}
\]
is closed in the fiber. Consequently,
\[
W\longmapsto
\hat\phi_{G_k}(t,x,W)
=
\log \mathrm{Jac}_k(F_{t,x}|_W),
\]
with the convention $\log 0=-\infty$, is upper semicontinuous and equals $-\infty$ on $K_{G_k}$. Moreover, outside $K_{G_k}$ the map $W\mapsto F_{t,x}W$ is continuous. Hence, for every $\psi\in\mathscr C_b(T\times\mathrm G_k(\mathcal E))$, the function
\[
\hat P_{G_k}\psi(t,x,W)
=
\int
\psi\bigl(s,f_t(x),F_{t,x}W\bigr)
\,Q_{f_t(x)}(t,ds)
\]
is continuous in \(W\) outside \(K_{G_k}\), by dominated
convergence. Since the complement of \(K_{G_k}\) is open in each
fiber, every fiberwise discontinuity point of \(\hat P_{G_k}\psi\)
belongs to \(K_{G_k}\).

Finally, the classical singular-value characterization of exterior powers gives
\begin{align*}
\mathcal M\hat\phi_{G_k}(t,x)
&=
\sup_{W\in\mathrm G_k(\mathcal E_x)}
\log\mathrm{Jac}_k(F_{t,x}|_W)=
\log\|\wedge^kF_{t,x}\|.
\end{align*}
Therefore,
\[
(\mathcal M\hat\phi_{G_k})^+(t,x)
\leq
\log^+\|\wedge^kF_{t,x}\|
\leq
k\log^+\|F_{t,x}\|,
\]
and the last function is $\nu$-integrable by~\eqref{eq:integrability-nu}. Thus all the assumptions of Theorem~\ref{thm:VP} are satisfied. In particular, the sequence $(\phi_{G_k,n})_{n\geq1}$ is $P$-subadditive and its growth rate
\begin{equation}\label{eq:grassmann-growth-rate}
\Lambda_{G_k}(\nu)
:=
\lim_{n\to\infty}
\frac1n\int \phi_{G_k,n}\,d\nu
=
\inf_{n\geq1}
\frac1n\int \phi_{G_k,n}\,d\nu
\end{equation}
satisfies
\begin{equation}\label{eq:grassmann-VP}
\Lambda_{G_k}(\nu)
=
\sup\left\{
\int\hat\phi_{G_k}\,d\hat\nu:
\hat\nu\in\mathcal I(\hat P_{G_k}),\
(\pi_{G_k})_*\hat\nu=\nu
\right\}.
\end{equation}
The supremum is attained on an ergodic $(\hat P_{G_k})^*$-invariant lift of $\nu$.

We next identify the finite-time process in~\eqref{eq:grassmann-additive-process}. We use the same Markov-chain notation as in the proof of Lemma~\ref{lem:finite-horizon-expansion}. Starting from $(t,x,W)$, let
$(\omega_j,x_j,W_j)$, $j\geq0$,  be the lifted Grassmannian chain, where
\[
(\omega_0,x_0,W_0)=(t,x,W),
\qquad
x_{j+1}=f_{\omega_j}(x_j),
\qquad
W_{j+1}=F_{\omega_j,x_j}W_j,
\]
until the orbit reaches the cemetery section. By the definition of $\hat P_{G_k}$,
\begin{align*}
\hat\phi_{G_k,n}(t,x,W)
&=
\sum_{j=0}^{n-1}
\mathbb E\bigl[
\hat\phi_{G_k}(\omega_j,x_j,W_j)
\bigm|
(\omega_0,x_0,W_0)=(t,x,W)
\bigr]\\
&=
\mathbb E\biggl[
\sum_{j=0}^{n-1}
\hat\phi_{G_k}(\omega_j,x_j,W_j)
\biggm|
(\omega_0,x_0,W_0)=(t,x,W)
\biggr].
\end{align*}
If degeneration occurs before time $n$, both the sum and the logarithmic Jacobian below are interpreted as $-\infty$. Otherwise, the multiplicativity of Jacobians yields
\begin{align*}
\sum_{j=0}^{n-1}
\hat\phi_{G_k}(\omega_j,x_j,W_j)
&=
\sum_{j=0}^{n-1}
\log\mathrm{Jac}_k(F_{\omega_j,x_j}|_{W_j})
=\log\mathrm{Jac}_k(F^n_{(\omega,x)}|_W).
\end{align*}
Consequently,
\begin{equation}\label{eq:grassmann-finite-horizon}
\hat\phi_{G_k,n}(t,x,W)
=
\mathbb E\biggl[
\log\mathrm{Jac}_k(F^n_{(\omega,x)}|_W)
\biggm|
(\omega_0,x_0,W_0)=(t,x,W)
\biggr].
\end{equation}
Since the conditional law of the future depends on the initial lifted state only through $(t,x)$ and the initial plane $W$ is fixed, this is precisely the conditional expectation denoted in the statement by conditioning on $(\omega_0,x,W)$. Taking the supremum over the fiber gives
\begin{equation}\label{eq:grassmann-fiber-maximum}
\phi_{G_k,n}(t,x)
=
\sup_{W\in\mathrm G_k(\mathcal E_x)}
\mathbb E\biggl[
\log\mathrm{Jac}_k(F^n_{(\omega,x)}|_W)
\biggm|
(\omega_0,x,W)
\biggr].
\end{equation}
Thus the first limit appearing in Proposition~\ref{cor:met-linear-grassmann} is exactly the growth rate $\Lambda_{G_k}(\nu)$ defined in~\eqref{eq:grassmann-growth-rate}.

\paragraph{\em Comparison with the projective growth rate}
We first prove
\begin{equation}\label{eq:grassmann-less-projective}
\Lambda_{G_k}(\nu)\leq \Sigma_k(m).
\end{equation}
For each $x\in X$, consider the Pl\"ucker embedding
\[
\iota_x:\mathrm G_k(\mathcal E_x)
\longrightarrow
\mathbb P(\wedge^k\mathcal E_x),
\qquad
\iota_x(W)=[w_1\wedge\cdots\wedge w_k],
\]
where $w_1,\dots,w_k$ is any basis of $W$. This is well defined because a different basis changes the wedge product by a nonzero scalar. Extending $\iota_x$ by sending the cemetery point to the cemetery point, we obtain a fiber-preserving measurable map
\[
\iota:T\times\mathrm G_k(\mathcal E)
\longrightarrow
\hat Z_k,
\qquad
\iota(t,x,W)=(t,x,\iota_x(W)).
\]
The definitions of the projective and Grassmannian lifts imply
\begin{equation}\label{eq:plucker-intertwining}
\iota\circ\hat F_{G_k,s}
=
\hat F_{k,s}\circ\iota
\qquad\text{for every }s\in T,
\end{equation}
and, if $\xi_W=w_1\wedge\cdots\wedge w_k$ is any nonzero decomposable representative of $W$, then
\begin{equation}\label{eq:plucker-potentials}
\hat\phi_{G_k}(t,x,W)
=
\log\frac{\|\wedge^kF_{t,x}\xi_W\|}{\|\xi_W\|}
=
\hat\phi_k\bigl(\iota(t,x,W)\bigr).
\end{equation}

Let $\hat\nu_G\in\mathcal I(\hat P_{G_k})$ satisfy $(\pi_{G_k})_*\hat\nu_G=\nu$, and set $\hat\nu:=\iota_*\hat\nu_G$. The intertwining relation~\eqref{eq:plucker-intertwining} implies that $\hat\nu\in\mathcal I(\hat P_k)$, and clearly $(\pi_k)_*\hat\nu=\nu$. Moreover, by~\eqref{eq:plucker-potentials},
\[
\int\hat\phi_{G_k}\,d\hat\nu_G
=
\int\hat\phi_k\,d\hat\nu.
\]
The variational formula in Theorem~\ref{mainthm:B} therefore gives
\[
\int\hat\phi_{G_k}\,d\hat\nu_G
\leq
\Lambda_k(\nu)
=
\Sigma_k(m).
\]
Taking the supremum over all admissible $\hat\nu_G$ and using~\eqref{eq:grassmann-VP} proves~\eqref{eq:grassmann-less-projective}.

\paragraph{\em The reverse inequality}
We now prove
\begin{equation}\label{eq:projective-less-grassmann}
\Sigma_k(m)\leq\Lambda_{G_k}(\nu).
\end{equation}
The key point is the following finite-dimensional estimate.

\begin{lem}\label{claim:grassmannian-averaging}
There is a constant $C_{d,k}<\infty$, depending only on $d$ and $k$, such that for every pair of $d$-dimensional Euclidean spaces $V,V'$ and every linear map $L:V\to V'$,
\begin{equation}\label{eq:grassmannian-averaging}
\log\|\wedge^kL\|
\leq
C_{d,k}
+
\int_{\mathrm G_k(V)}
\log\mathrm{Jac}_k(L|_W)\,d\rho_V(W),
\end{equation}
where $\rho_V$ is the normalized orthogonally invariant probability measure on $\mathrm G_k(V)$. If $\operatorname{rank}L<k$, both sides are understood as $-\infty$.
\end{lem}

\begin{proof}
If $\operatorname{rank}L<k$, then $\wedge^kL=0$ and $\mathrm{Jac}_k(L|_W)=0$ for every $W\in\mathrm G_k(V)$, so there is nothing to prove. Assume that $\operatorname{rank}L\geq k$. Let
\[
\sigma_1(L)\geq\cdots\geq\sigma_d(L)\geq0
\]
be the singular values of $L$. Choose orthonormal singular-vector bases $e_1,\dots,e_d$ of $V$ and $e'_1,\dots,e'_d$ of $V'$ such that $Le_i=\sigma_i(L)e'_i$.
For a multi-index $I=(i_1<\cdots<i_k)$, write
\[
e_I=e_{i_1}\wedge\cdots\wedge e_{i_k},
\qquad
e'_I=e'_{i_1}\wedge\cdots\wedge e'_{i_k}.
\]
If $W\in\mathrm G_k(V)$, choose an orthonormal basis $w_1,\dots,w_k$ of $W$ and put
$\xi_W=w_1\wedge\cdots\wedge w_k$.
Then $\|\xi_W\|=1$, $\xi_W$ is determined up to sign, and it admits an expansion
\[
\xi_W=\sum_{|I|=k}a_I(W)e_I,
\qquad
\sum_{|I|=k}|a_I(W)|^2=1.
\]
Let $I_0=(1,\dots,k)$. Since
\[
\wedge^kL(e_{I_0})
=
\sigma_1(L)\cdots\sigma_k(L)e'_{I_0},
\qquad
\|\wedge^kL\|
=
\sigma_1(L)\cdots\sigma_k(L),
\]
we obtain
\begin{align*}
\mathrm{Jac}_k(L|_W)
=
\|\wedge^kL(\xi_W)\|
\geq
|a_{I_0}(W)|\,\sigma_1(L)\cdots\sigma_k(L)=
|a_{I_0}(W)|\,\|\wedge^kL\|.
\end{align*}
Hence
\[
\log\mathrm{Jac}_k(L|_W)
\geq
\log\|\wedge^kL\|
+
\log|a_{I_0}(W)|.
\]
The function $W\mapsto a_{I_0}(W)$ is a nonzero Pl\"ucker coordinate. Its absolute value is well defined independently of the choice of the sign of $\xi_W$, and $\log|a_{I_0}|$ is integrable with respect to the smooth invariant probability measure $\rho_V$. By orthogonal invariance, the number
\[
c_{d,k}
:=
\int_{\mathrm G_k(V)}
\log|a_{I_0}(W)|\,d\rho_V(W)
\]
depends only on $d$ and $k$, and $c_{d,k}>-\infty$. Integrating the preceding inequality and setting $C_{d,k}:=-c_{d,k}$ proves~\eqref{eq:grassmannian-averaging}.
\end{proof}

We apply Lemma~\ref{claim:grassmannian-averaging} to
$L=F^n_{(\omega,x)}:
\mathcal E_x\longrightarrow\mathcal E_{f^n_\omega(x)}$.
Choose a Borel orthonormal frame of $\mathcal E$ and use it to identify every fiber $\mathcal E_x$ measurably with $\mathbb R^d$. The normalized invariant measure on $\mathrm G_k(\mathbb R^d)$ then induces a measurable family $(\rho_x)_{x\in X}$, where $\rho_x$ is the normalized orthogonally invariant probability measure on $\mathrm G_k(\mathcal E_x)$. For every $n\geq1$ and $m$-a.e. $(\omega,x)$, Lemma~\ref{claim:grassmannian-averaging} gives
\begin{equation}\label{eq:averaging-applied-to-cocycle}
\log\|\wedge^kF^n_{(\omega,x)}\|
\leq
C_{d,k}
+
\int_{\mathrm G_k(\mathcal E_x)}
\log\mathrm{Jac}_k(F^n_{(\omega,x)}|_W)\,d\rho_x(W).
\end{equation}
The positive parts of the functions in~\eqref{eq:averaging-applied-to-cocycle} are integrable. Indeed,
\[
\log^+\mathrm{Jac}_k(F^n_{(\omega,x)}|_W)
\leq
\log^+\|\wedge^kF^n_{(\omega,x)}\|
\leq
k\sum_{j=0}^{n-1}\log^+\|F_{\omega_j,x_j}\|,
\]
and the last function is $m$-integrable by~\eqref{eq:integrability-m} and the invariance of $m$. Therefore we may take conditional expectations and apply Fubini's theorem to the extended-real-valued functions involved. Conditioning~\eqref{eq:averaging-applied-to-cocycle} on $Z_0=(\omega_0,x_0)=(t,x)$, and then using~\eqref{eq:grassmann-finite-horizon}, gives
\begin{align*}
\mathbb E\bigl[
\log\|\wedge^kF^n_{(\omega,x)}\|
\bigm|
Z_0=(t,x)
\bigr]
&\leq
C_{d,k}
+
\int
\mathbb E\biggl[
\log\mathrm{Jac}_k(F^n_{(\omega,x)}|_W)
\biggm|
(\omega_0,x,W)
\biggr]
\,d\rho_x(W)\\
&\leq
C_{d,k}
+
\sup_{W\in\mathrm G_k(\mathcal E_x)}
\mathbb E\biggl[
\log\mathrm{Jac}_k(F^n_{(\omega,x)}|_W)
\biggm|
(\omega_0,x,W)
\biggr]\\
&=
C_{d,k}+\phi_{G_k,n}(t,x).
\end{align*}
Integrating with respect to $\nu$ and using the construction of the driving measure $m$, we obtain
\begin{equation}\label{eq:grassmann-integrated-comparison}
\int
\log\|\wedge^kF^n_{(\omega,x)}\|\,dm
\leq
C_{d,k}
+
\int\phi_{G_k,n}\,d\nu.
\end{equation}
Divide~\eqref{eq:grassmann-integrated-comparison} by $n$ and let $n\to\infty$. By the classical exterior-power formula and~\eqref{eq:grassmann-growth-rate},
$\Sigma_k(m)
\leq
\Lambda_{G_k}(\nu)$,
which proves~\eqref{eq:projective-less-grassmann}.

Combining~\eqref{eq:grassmann-less-projective} and~\eqref{eq:projective-less-grassmann}, we conclude that
\begin{equation}\label{eq:grassmann-projective-equality}
\Lambda_{G_k}(\nu)
=
\Sigma_k(m)
=
\sum_{i=1}^k\lambda_i(m).
\end{equation}
Together with~\eqref{eq:grassmann-growth-rate},~\eqref{eq:grassmann-VP}, and~\eqref{eq:grassmann-fiber-maximum}, this proves the two formulas in the first part of Proposition~\ref{cor:met-linear-grassmann}, as well as the assertion that the limit may be replaced by the infimum and that the supremum is attained on an ergodic invariant lift.

If $\nu$ is ergodic, Theorem~\ref{mainthm:A} and Kingman's theorem for Markov operators~\cite[Thm.~D]{BM} applied to the $P$-subadditive sequence $(\phi_{G_k,n})_{n\geq1}$ give
\[
\lim_{n\to\infty}
\frac1n\phi_{G_k,n}(t,x)
=
\Lambda_{G_k}(\nu)
=
\sum_{i=1}^k\lambda_i(m)
\qquad
\text{for $\nu$-a.e. $(t,x)$}.
\]
Using~\eqref{eq:grassmann-fiber-maximum}, this is exactly the pointwise formula in the proposition.

Finally, suppose that $\lambda_k(m)>-\infty$. Since the Lyapunov exponents are ordered,
\[
\lambda_1(m)\geq\cdots\geq\lambda_k(m)>-\infty,
\]
and therefore $\Sigma_k(m)>-\infty$. By~\eqref{eq:grassmann-projective-equality}, $\Lambda_{G_k}(\nu)>-\infty$. The last assertion of Theorem~\ref{thm:VP} then implies that every maximizing lift in~\eqref{eq:grassmann-VP} gives zero mass to the degeneracy locus $K_{G_k}$. This completes the proof of Proposition~\ref{cor:met-linear-grassmann}.
\end{proof}

\section{Bernoulli random bundle morphisms}
\label{sec:Prop-IV}

Consider a Bernoulli measure $\mathbb{P}=p^\mathbb{N}$ on $\Omega=T^\mathbb{N}$.
Hence the Markov operator~\eqref{eq:fully-place-P} becomes
\[
P\varphi(t,x)=\int \varphi\bigl(s,f_t(x)\bigr)\,dp(s),
\qquad
(t,x)\in Z=T\times X,\ \varphi\in B(Z).
\]
Moreover, by~\cite[Thm.~3.4 and Lem.~3.7]{BMNNT}, every \(P^*\)-invariant measure \(\nu\) on \(Z\) is of the form \(\nu=p\times\mu\), where \(\mu\) is \emph{\(f\)-stationary}.  That is, $\mu=\int (f_t)_*\mu  \, dp$. We stress that \(\mu\) depends implicitly on \(p\).  We then consider the random map
$$
\bar F_k: T\times \hat{\mathcal{E}}_k \to \hat{\mathcal{E}}_k, \quad \bar F_{k,t}(x,[v])=(f_t(x),[\wedge^kF_{t,x}v]),
$$
with $[\wedge^kF_{t,x}v]=\Delta_{f_t(x)}$ if $v\in \mathrm{ker}(\wedge^kF_{t,x})$ or $[v]=\Delta_x$, and the annealed potential
\begin{equation}\label{eq:bernoulli-phik}
\bar\phi_k(x,[v])
:=
\int
\hat \phi_k(t,x,[v])
\,dp\in[-\infty,\infty),
\qquad
(x,[v])\in \hat{\mathcal E}_k.
\end{equation}
A probability measure \(\hat\mu\) on \(\hat{\mathcal E}_k\) is called \emph{\(\bar F_k\)-stationary} if $\hat \mu = \int  (\bar F_{k,t})_*\hat\mu \, dp$.

Let \(\mu\) be an \(f\)-stationary measure satisfying~\eqref{eq:integrability-nu} with respect to \(\nu=p\times\mu\). Applying the subadditive ergodic theorem for Markov operators~\cite[Thm.~D]{BM}
yields the Lyapunov exponents $$\lambda_1(\mu)\ge\cdots\ge\lambda_d(\mu) \quad \text{with} \quad \lambda_k(\mu)=\int \lambda_k(x)\, d\mu \  \ \text{for $k=1,\dots,d$}
$$
where, setting \(\|\wedge^0F^n_{(\bar\omega,x)}\|:=1\), if $\lambda_k(x)>-\infty$, then
\begin{equation} \label{eq:annealead_exponents}
\lambda_k(x)=\lim_{n\to\infty}\frac1n
\int
\log
\frac{\|\wedge^kF^n_{(\omega,x)}\|}
{\|\wedge^{k-1}F^n_{(\omega,x)}\|}
\,d\mathbb P
\qquad\text{for \(\mu\)-a.e.\ \(x \in X\)}.
\end{equation}
The next result shows that these exponents coincide with those obtained from the driving measure \(m=\mathbb P\times\mu\).

\begin{mainprop}\label{prop:bernoulli-variational}
Let \(F:T\times \mathcal E\to\mathcal E\) be a random linear bundle morphism covering a random map \(f:T\times X\to X\). Consider a probability measure $p$ on \(T\), and let \(\mu\) be an \(f\)-stationary probability measure satisfying~\eqref{eq:integrability-nu} with \(\nu=p\times\mu\). Let \(m=\mathbb P\times\mu\) be the corresponding driving measure on \(\Omega\times X\). Then,
\[
\lambda_k(\omega,x)=\lambda_k(x)
\quad\text{for \(m\)-a.e.\ \((\omega,x)\),}
\qquad
\lambda_k(m)=\lambda_k(\mu),
\quad
k=1,\dots,d,
\]
and

\begin{equation}\label{eq:bernoilli-VP}
\sum_{i=1}^k\lambda_i(m)
=
\sup\Bigl\{
\int \bar\phi_k\,d\hat\mu:
\hat\mu \ \text{is $\bar{F}_k$-stationary and }
(\pi_k)_*\hat\mu=\mu
\Bigr\}.
\end{equation}
Moreover, the supremum is attained by an ergodic \(\bar F_k\)-stationary lift \(\hat\mu\) of \(\mu\), and any maximizing measure has zero mass on the degeneracy locus whenever \(\lambda_k(m)>-\infty\).
\end{mainprop}
\begin{proof}
    {We apply} Theorem~\ref{mainthm:B} with $Q_x(t,ds)=p(ds)$ and $\nu=p\times\mu$.
This yields
\[
\Lambda_k(\nu)
=
\sup\Bigl\{\int \hat\phi_k\,d\hat\nu:\
\hat\nu\in\mathcal I(\hat P_k)\ \text{and}\ (\pi_k)_*\hat\nu=\nu\Bigr\}.
\]

Now invoke the Bernoulli reduction from \cite[Thm.~3.4 and Lem.~3.7]{BMNNT}:
stationary measures for the transition probability associated with the Markovian place-dependent iteration correspond to $\mathbb P$-Markovian invariant measures for the skew-product. Moreover, in the Bernoulli case, any such invariant $\mathbb P$-Markovian measure must be a product.
In particular, any admissible $\hat\nu$ above must be of the form
$\hat\nu = p\times \hat\mu$
for some probability measure $\hat\mu$ on $\hat{\cE}_k$, and the constraint $({\pi}_k)_*\hat\nu=\nu$
is equivalent to $(\pi_k)_*\hat\mu=\mu$.

A direct computation shows that $\hat\nu=p\times\hat\mu$ is $\hat P_k^*$-invariant if and only if
$\hat\mu$ is $\bar{P}_k^*$-invariant: for every bounded measurable $\psi$ on $\hat{\cE}_k$,
\[
\int \psi\,d\hat\mu
=
\int \psi\,d(p\times\hat\mu)
=
\int\hat P_k\psi\,d(p\times\hat\mu)
=
\int \bar{P}_k\psi\,d\hat\mu.
\]
Finally,
\[
\int_{\hat Z_k}\hat\phi_k\,d(p\times\hat\mu)
=
\int_{\hat{\cE}_k}\Bigl(\int_T \hat\phi_k(t,x,[v])\,dp(t)\Bigr)\,d\hat\mu(x,[v])
=
\int_{\hat{\cE}_k}\bar\phi_k\,d\hat\mu,
\]
which gives the desired variational formula {$\Sigma_k(m)=\Lambda_k(\mu)$} where $\Sigma_k(m):=\sum_{i=1}^k \lambda_i(m)$. Attainment, as well as the statement about the degeneracy locus when
$\lambda_k(m)>-\infty$, follows from Theorem~\ref{mainthm:B}.

It remains to prove the coincidence of Lyapunov exponents. First, \(\lambda_k(m)=\lambda_k(\mu)\) follows by taking consecutive differences in the identity \(\Sigma_k(m)=\Lambda_k(\mu)\). For the pointwise statement,
let $\mu=\int \mu_\alpha\,d\eta(\alpha)$ be the ergodic decomposition of $\mu$ for $\bar P$
and set $\nu_\alpha:=p\times\mu_\alpha$.
By the Bernoulli reduction (cf.\ \cite[Theorem~3.4 and Lemma~3.7]{BMNNT}), the measures $\nu_\alpha$ are precisely the $P$-ergodic
components of $\nu$, hence $\nu=\int \nu_\alpha\,d\eta(\alpha)$ is the ergodic decomposition for $P$.

Fix $\alpha$. Since $\nu_\alpha$ is $P$-ergodic, Kingman's theorem for $P$ yields that the pointwise sums
$\Lambda_k(t,x)$ are $\nu_\alpha$-a.e.\ constant and equal to $\Lambda_k(\nu_\alpha)$; therefore
$\lambda_k(t,x)$ is $\nu_\alpha$-a.e.\ constant and equal to $\lambda_k(\nu_\alpha)$.
Likewise, since $\mu_\alpha$ is $\bar P$-ergodic, the pointwise growth rates $\Lambda_k(x)$ are $\mu_\alpha$-a.e.\ constant and equal to
$\Lambda_k(\mu_\alpha)$, hence $\lambda_k(x)$ is $\mu_\alpha$-a.e.\ constant and equal to $\lambda_k(\mu_\alpha)$. Finally, applying the coincidence of {growth rates} to the measure $\mu_\alpha$ gives equality of the exponents on each ergodic component, i.e.,
$\lambda_k(\nu_\alpha)=\lambda_k(\mu_\alpha)$ for all $k$.
Thus, $\lambda_k(t,x)=\lambda_k(\nu_\alpha)=\lambda_k(\mu_\alpha)=\lambda_k(x)$, for $\nu_\alpha$-a.e.\ $(t,x)$. Integrating over $\alpha$ yields that $\lambda_k(t,x)=\lambda_k(x)$ for $\nu$-a.e.~$(t,x)$.
Finally, from Theorem~\ref{mainthm:B} {we have}  $\lambda_k(\omega,x)=\lambda_k(\omega_0,x)$ for $m$-{a.e.~}$(\omega,x)$. Combining both, we get $\lambda_k(\omega,x)=\lambda_k(x)$ for $m$-a.e.~$(\omega,x)$. This completes the proof of Proposition~\ref{prop:bernoulli-variational}.
\end{proof}

Although the $\omega$-independence of the pointwise Lyapunov exponents is classical, see~\cite[Thm.~3.2]{liu2006smooth},  Proposition~\ref{prop:bernoulli-variational} yields a stronger conclusion: it identifies the exponent with the annealed asymptotic formula~\eqref{eq:annealead_exponents}. For the top Lyapunov exponent, an analogous annealed characterization was obtained in~\cite[Thm.~3.2.1]{BM}.

Applying Theorem~\ref{thm:VP} to the factor system
\((\pi_k,\bar P_k,\bar P,\bar\phi_k,\mu)\), where \(\bar P\) and
\(\bar P_k\) are the annealed Koopman operators associated with \(f\)
and \(\bar F_k\), respectively, shows that the variational quantity
in~\eqref{eq:bernoilli-VP} also agrees with the asymptotic growth
obtained by first averaging over the noise and then maximizing over
directions. Since Proposition~\ref{prop:bernoulli-variational} identifies the same variational quantity with the sum of the first $k$ Lyapunov exponents, the following corollary follows immediately.

\begin{maincor}\label{cor:bernoulli-sup-int}
Under the assumptions of Proposition~\ref{prop:bernoulli-variational}, for every \(k=1,\dots, d\),
\begin{align*}
\sum_{i=1}^k\lambda_i(m)
&=
\lim_{n\to\infty}\frac1n
\int
\sup_{[v]\in P(\wedge^k \mathcal E_x)}
\biggl(
\int
\log\frac{\|\wedge^kF^n_{(\omega,x)}v\|}{\|v\|}
\,d\mathbb P
\biggr)\,d\mu
\end{align*}
where the limit may be replaced by the infimum.
Moreover, if $\mu$ is ergodic, then
\begin{align*}
\sum_{i=1}^k\lambda_i(m)
=
\lim_{n\to\infty}  \
\sup_{[v]\in P(\wedge^k \mathcal E_x)} \, \frac1n
 \int
\log\frac{\|\wedge^kF^n_{(\omega,x)}v\|}{\|v\|}
\,d\mathbb P \quad \text{for $\mu$-a.e.~$x\in X$}.
\end{align*}
\end{maincor}

Finally, we also have the following:

\begin{rem} The variational formula~\eqref{eq:bernoilli-VP} and the asymptotic representations in Corollary~\ref{cor:bernoulli-sup-int} still hold by modifying the supremum over the directions $[v] \in P(\wedge^k \mathcal E_x)$ by a supremum over $k$-planes $W\in \mathrm{G}_k(\mathcal E_x)$,  potential $\bar\phi_k(x,[v])$ by $\bar\phi_{G_k}(x,W)=\int \log\mathrm{Jac}_k(F_{t,x}|_W) \, dp$, and $\log(\|\wedge^kF^n_{(\omega,x)}v\|/\|v\|)$ by $\log\mathrm{Jac}_k(F^n_{(\omega,x)}|_W)$.
\end{rem}


\section{Edge Markovian random bundle morphisms}\label{sec:Prop-VI}

In the place-independent Markovian setting, the dependence on the noise can be modeled in two natural ways.  In the \emph{vertex model} of Theorem~\ref{mainthm:B}, the dynamics are assigned to the states, so that the map \(f_t\) depends only on the current vertex \(t\) and the evolution is \(x_{n+1}=f_{t_n}(x_n)\). In the \emph{edge model}, by contrast, the dynamics are assigned to the transitions, so that the map \(f_{t,s}\) depends on the jump \(t\to s\) and the evolution is \(x_{n+1}=f_{t_n,t_{n+1}}(x_n)\). By enlarging the noise space to admissible transitions, the edge model reduces to the vertex framework. Consequently, Theorem~\ref{mainthm:B} yields a variational formula, but with an annealed potential.

To formalize this, let \(f:T\times T\times X\to X\) be a measurable map, which we call an \emph{edge random map}, and write \(f_{t,s}:=f(t,s,\cdot)\). Let
$
F:T\times T\times \mathcal E\to \mathcal E$
be an \emph{edge random linear bundle morphism} covering \(f\), meaning that
\[
F_{t,s,x}:\mathcal E_x\to \mathcal E_{f_{t,s}(x)}
\quad\text{is linear for every }t,s\in T,
\qquad
\pi\circ F_{t,s}=f_{t,s}\circ\pi.
\]

To study its Lyapunov spectrum, we consider the driving measure \(m\) on \(\Omega\times X\) constructed exactly as in the vertex case. Namely, \(m=\tilde\pi_*\tilde{\mathbb P}\), where \(\tilde{\mathbb P}\) is the Markov measure on \(Z^{\mathbb N}\), \(Z=T\times X\), determined by the transition kernel
\[
q((t,x),C)=\int 1_C\bigl(s,f_{t,s}(x)\bigr)\,Q(t,ds),
\qquad
(t,x)\in Z,\ C\in\mathscr C=\mathscr A\otimes\mathscr B,
\]
and the stationary initial law \(\nu\). We also consider the  edge-driven skew-product
\[
\mathcal F:\Omega\times X\to\Omega\times X,
\qquad
\mathcal F(\omega,x)=(\sigma\omega,f_{\omega_0,\omega_1}(x)),
\]
and the iterates of the edge random bundle morphism,
\[
F^n_{(\omega,x)}
=
F_{\mathcal F^{n-1}(\omega,x)}\circ\cdots\circ F_{\mathcal F(\omega,x)}\circ F_{(\omega,x)},
\qquad
F_{(\omega,x)}:=F_{\omega_0,\omega_1,x}.
\]
Under the standard integrability condition~\eqref{eq:integrability-m}, one introduces the Lyapunov exponents $\lambda_1(m)\ge\cdots\ge\lambda_d(m)$ and  $\lambda_1(\omega,x)\ge\cdots\ge\lambda_d(\omega,x)$ as in~\eqref{eq:lypunov-exponents} and Appendix~\ref{ss:lyapunov-spectrum}. Also following~\S\ref{ss:spectrum-lyapunov-random}, the Lyapunov exponents $\lambda_1(\nu)\ge\cdots\ge\lambda_d(\nu)$ and $\lambda_1(t,x)\ge\cdots\ge\lambda_d(t,x)$ are also  defined from the initial law \(\nu\) by applying Kingman’s theorem for Markov operators~\cite[Thm.~D]{BM} and Theorem~\ref{mainthm:A} to \((\pi_k,\hat P_k,P,\hat\phi_k,\nu)\) where, abusing of notation,
\[
\nu=(\pi_0)_*m=P^*\nu,
\qquad
P\varphi(t,x)=\int \varphi\bigl(s,f_{t,s}(x)\bigr)\,Q(t,ds),
\]
the lifted operator is
\begin{equation}\label{eq:averaged-operator}
\hat P_k\psi(t,x,[v])
=
\int \psi\bigl(s,f_{t,s}(x),[\wedge^kF_{t,s,x}v]\bigr)\,Q(t,ds),
\end{equation}
and the potential is
\begin{equation}\label{eq:averaged-potential}
\hat\phi_k(t,x,[v])
=
\int \log\frac{\|\wedge^kF_{t,s,x}v\|}{\|v\|}\,Q(t,ds),
\end{equation}
with the usual conventions in the singular case.

\begin{mainprop}\label{prop:edge-variational}
Let \(F\) be an edge Markovian random bundle morphism satisfying~\eqref{eq:integrability-m} with respect to the driving measure \(m\) constructed from a transition probability kernel \(Q(t,A)\) on \(T\times\mathscr A\) and a stationary initial law \(\nu\) on \(T\times X\). Then, for every $k=1,\dots,d$,
\[
\lambda_k(\omega,x)=\lambda_k(\omega_0,x)
\quad\text{for \(m\)-a.e.\ \((\omega,x)\),}
\qquad
\lambda_k(m)=\lambda_k(\nu),
\quad
k=1,\dots,d,
\]
and
\[
\sum_{i=1}^k \lambda_i(m)
=
\sup\left\{
\int \hat\phi_k\,d\hat\nu:
\hat\nu\in\mathcal I(\hat P_k)\ \text{and}\ (\pi_k)_*\hat\nu=\nu
\right\}
\]
where \(\hat P_k\) and \(\hat\phi_k\) are given by~\eqref{eq:averaged-operator} and~\eqref{eq:averaged-potential}, respectively.
Moreover, the supremum is attained on an ergodic \((\hat P_k)^*\)-invariant lift of \(\nu\) and any maximizing measure has zero mass on the degeneracy locus whenever \(\lambda_k(m)>-\infty\).
\end{mainprop}

\begin{proof}
Let $T^e:=T\times T$ and $\Omega^e:=(T^e)^{\mathbb N}$. Define the canonical coding $\Theta:\Omega\to\Omega^e$ by
\[
\Theta((\omega_i)_{i\geq 0}):=\bigl((\omega_0,\omega_1),(\omega_1,\omega_2),\dots\bigr).
\]
Let $m^e:=(\Theta\times \mathrm{id}_X)_*m$. By definition of the edge cocycle, we have the exact identification $F^n_{(\omega,x)} = (F^e)^n_{(\Theta(\omega),x)}$, where the auxiliary \emph{vertex} system on the noise space $T^e$ is given by
$
f^e_{(t,s)}(x):=f_{t,s}(x)$ and $F^e_{(t,s),x}:=F_{t,s,x}$,
and is driven by the edge-kernel 
$$
Q^e\bigl((t,s),D\bigr):=\int 1_D(s,r) Q(s,dr). \quad (t,s)\in T^e, \ D\in \mathscr{A}\otimes \mathscr{A}.
$$
Consequently, $\lambda_k(m)=\lambda_k(m^e)$.

Let $\nu:=(\pi_0)_*m$ be the stationary law on $T\times X$ and 
consider the probability measure $\nu^e:=(\pi_0^e)_*m^e$ on $T^e\times X$ where  $\pi_0^e(\omega^e,x):=(\omega^e_0,x)$. Since the second coordinate of an edge is sampled according to $Q(t,\cdot)$ from the first coordinate, one has
\[
d\nu^e((t,s),x)=Q(t,ds)\,d\nu(t,x).
\]
Then, since 
$$
 P^e\varphi ((t,s),x)=\int \varphi(z,f^e_{(t,s)}(x))  \, Q^e((t,s),dz)=\int \varphi((s,r),f_{t,s}(x)) \, Q(s,dr) 
$$
we have that for every $\varphi \in B(T^e \times X)$, 
\begin{align*}
    \int P^e\varphi \, d\nu^e 
    &= \int  \varphi\bigl((s,r), f_{t,s}(x)\bigr) \, Q(s,dr)  \, Q(t,ds) \, d\nu(t,x).
\end{align*}
Defining the auxiliary function $\Phi \in B(T \times X)$ by $\Phi(s,y) \coloneqq \int \varphi((s,r), y) \, Q(s,dr)$, and substituting this into the integral above and using that $\nu$ is $P^*$-invariant (i.e., $\int P\Phi \, d\nu = \int \Phi \, d\nu$), we obtain
\begin{align*}
    \int P^e\varphi \, d\nu^e 
    &= \int \Phi(s, f_{t,s}(x)) \, Q(t,ds) \, d\nu(t,x) = \int P\Phi(t,x) \, d\nu (t,x) \\
    &= \int \Phi \, d\nu = \int \varphi\bigl((t,v), x\bigr) \, Q(t,dr)  \, d\nu(t,x)=\int \varphi \, d\nu^e.
\end{align*}
Thus, $(P^e)^*\nu^e = \nu^e$. 

Applying Theorem~\ref{mainthm:B} to the vertex system $(f^e, F^e)$ driven by $Q^e((t,s),D)$ with stationary measure $\nu^e$  yields
\begin{equation}\label{eq:edge-vertex-VP}
\Sigma_k(m):=\sum_{i=1}^k\lambda_i(m) =
\sum_{i=1}^k\lambda_i(m^e)
=\sup\left\{\int \hat\phi^e_k\,d\hat\nu^e:\ \hat\nu^e\in\mathcal I_{\nu^e}(\hat P^e_k)\right\}
\end{equation}
where 
$$\hat\phi^e_k((t,s),x,[v])=\log\frac{\|\wedge^k F_{t,s,x}v\|}{\|v\|}$$ 
and 
\[
\hat P^e_k\psi((t,s),x,[v])
=\int \psi\bigl((s,r),\,f_{t,s}(x),\,[\wedge^kF_{t,s,x}v]\bigr)\,Q(s,dr).
\]
Moreover,  the supremum in~\eqref{eq:edge-vertex-VP} is attained on an ergodic $(\hat P^e_k)^*$-invariant lift measure of $\nu^e$ (with null mass on the degeneracy locus provided $\lambda_k(m)=\lambda_k(m^e)>-\infty$).


{We now compare the supremum in \eqref{eq:edge-vertex-VP} with the variational formula
\[
\Lambda_k(\nu):=\sup \left\{ \int \hat\phi_k \, d\hat{\nu}: \hat\nu \in \mathcal{I}_{\nu}(\hat P_k)  \right\},
\]
where the annealed operator $\hat P_k$ and the annealed potential $\hat\phi_k$ are given by \eqref{eq:averaged-operator} and \eqref{eq:averaged-potential}, respectively. We first show that every $\hat \nu \in \mathcal{I}_{\nu}(\hat P_k)$ induces a measure $\hat \nu^e \in \mathcal{I}_{\nu^e}(\hat P^e_k)$ with the same integral. This gives one inequality between the two suprema:
}


\begin{lem} \label{lem:a}  $\Lambda_k(\nu) \leq \Sigma_k(m)$ \end{lem} \begin{proof} Fix $\hat\nu\in\mathcal I_{\nu}(\hat P_k)$. 
 We define the measure $\hat\nu^e$ on $\hat Z^e_k = T^e \times \hat{\mathcal{E}}_k$ by 
\begin{equation} \label{eq:disint-Q}
    d\hat\nu^e((t,s),x,[v]) := Q(t,ds)\,d\hat\nu(t,x,[v]).
\end{equation}
From this {disintegration} formula and {Fubini's theorem, we obtain}
\begin{align} \label{eq:Fubinni}
    \int \hat\phi^e_k \, d\hat\nu^e 
    &= \int\log\frac{\|\wedge^k F_{t,s,x}v\|}{\|v\|} \, Q(t,ds)  \,d\hat\nu(t,x,[v]) 
    = \int \hat\phi_k  \, d\hat\nu.
\end{align}
{It remains to prove that \(\hat \nu^e \in \mathcal{I}_{\nu^e}(\hat P^e_k)\). We first verify that \(\hat\nu^e\) projects to \(\nu^e\).} Let $\varphi \in B(T^e \times X)$ be a test function {on the base space.} Since $\hat\nu$ projects to $\nu$, we have
\begin{align*}
    \int \varphi \, d\hat\nu^e 
    &= \int \varphi((t,s),x) \, Q(t,ds) \,  d\hat\nu(t,x,[v]) \\ &= \int  \varphi((t,s),x) \, Q(t,ds) \, d\nu(t,x) = \int \varphi \, d\nu^e.
\end{align*}
Thus, $\hat\nu^e$ is a lift of $\nu^e$. {To prove invariance, let \(\psi\in B(\hat Z_k^e)\) and define
\[
\Psi(t,x,[v]) := \int \psi((t,s),x,[v])\,Q(t,ds).
\]
Then, using \eqref{eq:disint-Q},
\begin{align*}
\int \hat P_k^e\psi\,d\hat\nu^e
&=
\int \psi\bigl((s,r),f_{t,s}(x),[\wedge^kF_{t,s,x}v]\bigr)\,Q(s,dr)\,Q(t,ds)\,d\hat\nu(t,x,[v]) \\
&=
\int \hat P_k\Psi(t,x,[v])\,d\hat\nu(t,x,[v]) =
\int \Psi\,d\hat\nu 
=
\int \psi\,d\hat\nu^e,
\end{align*}
since \(\hat\nu\in\mathcal I_\nu(\hat P_k)\). Hence \(\hat\nu^e\in \mathcal I_{\nu^e}(\hat P_k^e)\).}
\end{proof}



\begin{lem} \label{lem:b} $\Lambda_k(\nu) \geq \Sigma_k(m)$ \end{lem} 

\begin{proof} Fix \(\hat\nu^e\in\mathcal I_{\nu^e}(\hat P_k^e)\). Define
\[
G:\hat Z_k^e\to \hat Z_k,
\qquad
G((t,s),x,[v]) := \bigl(s,f_{t,s}(x),[\wedge^kF_{t,s,x}v]\bigr),
\]
and let $\hat\nu := G_*\hat\nu^e$. We first show that \((\pi_k)_*\hat\nu=\nu\). Let \(\varphi\in B(T\times X)\). Then
\begin{align*}
\int \varphi\,d(\pi_k)_*\hat\nu
&=
\int \varphi\circ \pi_k\circ G \, d\hat\nu^e =
\int \varphi\bigl(s,f_{t,s}(x)\bigr)\,d\hat\nu^e((t,s),x,[v])  \\
&=
\int \varphi\bigl(s,f_{t,s}(x)\bigr)\,d\nu^e((t,s),x) 
=
\int \varphi\bigl(s,f_{t,s}(x)\bigr)\,Q(t,ds)\,d\nu(t,x) 
\\ & =
\int P\varphi\,d\nu 
=
\int \varphi\,d\nu,
\end{align*}
because \((\pi_k^e)_*\hat\nu^e=\nu^e\) and \(\nu\in\mathcal I(P)\). Thus \((\pi_k)_*\hat\nu=\nu\).

Next we prove that \(\hat\nu\) is \((\hat P_k)^*\)-invariant. Let \(\psi\in B(\hat Z_k)\). A direct computation gives
\[
(\hat P_k\psi)\circ G = \hat P_k^e(\psi\circ G).
\]
Indeed, for \(((t,s),x,[v])\in\hat Z_k^e\),
\begin{align*}
((\hat P_k\psi)\circ G)((t,s),x,[v])
&=
(\hat P_k\psi)\bigl(s,f_{t,s}(x),[\wedge^kF_{t,s,x}v]\bigr) \\
&=
\int \psi\bigl(r,f_{s,r}(f_{t,s}(x)),[\wedge^kF_{s,r,f_{t,s}(x)}\,\wedge^kF_{t,s,x}v]\bigr)\,Q(s,dr) \\
&=
\hat P_k^e(\psi\circ G)((t,s),x,[v]).
\end{align*}
Therefore,
\begin{align*}
\int \hat P_k\psi\,d\hat\nu
&=
\int (\hat P_k\psi)\circ G\,d\hat\nu^e =
\int \hat P_k^e(\psi\circ G)\,d\hat\nu^e 
=
\int \psi\circ G\,d\hat\nu^e 
=
\int \psi\,d\hat\nu,
\end{align*}
since \(\hat\nu^e\in\mathcal I(\hat P_k^e)\). Hence \(\hat\nu\in\mathcal I_\nu(\hat P_k)\).

We now compare the integrals. By definition of \(G\),
\[
(\hat\phi_k\circ G)((t,s),x,[v])
=
\int \log \frac{\|\wedge^kF_{s,r,f_{t,s}(x)}\,\wedge^kF_{t,s,x}v\|}{\|\wedge^kF_{t,s,x}v\|}\,Q(s,dr).
\]
On the other hand,
\[
(\hat P_k^e\hat\phi_k^e)((t,s),x,[v])
=
\int \log \frac{\|\wedge^kF_{s,r,f_{t,s}(x)}\,\wedge^kF_{t,s,x}v\|}{\|\wedge^kF_{t,s,x}v\|}\,Q(s,dr),
\]
so
\[
\hat\phi_k\circ G = \hat P_k^e\hat\phi_k^e.
\]
Using the invariance of \(\hat\nu^e\), we obtain
\begin{align*}
\int \hat\phi_k\,d\hat\nu
=
\int \hat\phi_k\circ G\,d\hat\nu^e 
=
\int \hat P_k^e\hat\phi_k^e\,d\hat\nu^e 
=
\int \hat\phi_k^e\,d\hat\nu^e.
\end{align*}
Thus every \(\hat\nu^e\in\mathcal I_{\nu^e}(\hat P_k^e)\) produces a measure
\(\hat\nu\in\mathcal I_\nu(\hat P_k)\) with the same integral. This concludes the proof of the inequality.
\end{proof}

\medskip

Combining Lemma~\ref{lem:a} and~\ref{lem:b}, we obtain
$\Sigma_k(m)=\Lambda_k(\nu)$,
which is exactly the variational formula in Proposition~\ref{prop:edge-variational}. Finally, let \(\hat\nu_*^e\in\mathcal I_{\nu^e}(\hat P_k^e)\) be an ergodic maximizing lift in \eqref{eq:edge-vertex-VP}, and set $\hat\nu_*:=G_*\hat\nu_*^e$. 
Then \(\hat\nu_*\in\mathcal I_\nu(\hat P_k)\), and
$
\int \hat\phi_k\,d\hat\nu_*
=
\int \hat\phi_k^e\,d\hat\nu_*^e
=
\Sigma_k(m)$. 
Moreover, \(\hat\nu_*\) is ergodic, since \(G\) is a factor map between the two Markov systems. Therefore the supremum defining \(\Lambda_k(\nu)\) is attained on an ergodic \((\hat P_k)^*\)-invariant lift of \(\nu\). If \(\lambda_k(m)>-\infty\), then
$\int \hat\phi_k\,d\hat\nu_* > -\infty$, 
so \(\hat\nu_*\) gives zero mass to the degeneracy locus \(\{\hat\phi_k=-\infty\}\).

It remains to prove the pointwise identity. The comparison above uses
only the \(P^*\)-invariance of the initial law, and therefore applies to
every ergodic component of \(\nu\). Write
\[
\nu=\int\nu_\alpha\,d\eta(\alpha)
\]
for the ergodic decomposition of \(\nu\), and let \(m_\alpha\) be the
driving measure associated with \(\nu_\alpha\). Hence, $\Sigma_k(m_\alpha)=\Lambda_k(\nu_\alpha)$ for every $k$.
By the same ergodic-decomposition argument used in
\S\ref{ss:ergodic-decomposition}, $\Sigma_k(\omega,x)
=
\Sigma_k(m_\alpha)
=
\Lambda_k(\nu_\alpha)
=
\Lambda_k(\omega_0,x)$  for \(m_\alpha\)-a.e.~\((\omega,x)\).
Integrating over \(\alpha\), 
\[
\Sigma_k(\omega,x)=\Lambda_k(\omega_0,x)
\qquad\text{for \(m\)-a.e.~}(\omega,x).
\]
Taking consecutive differences with the conventions in
\eqref{eq:def-Lyapunov-law} and
\eqref{eq:def-lyapunov-sigma} yields
\[
\lambda_k(\omega,x)=\lambda_k(\omega_0,x)
\qquad\text{for \(m\)-a.e.~}(\omega,x).
\]
This completes the proof of Proposition~\ref{prop:edge-variational}.
\end{proof}

\section{Random matrix products}
\label{s:linear}

When \(X=\{\ast\}\), \(\mu=\delta_\ast\), and \(\mathcal E\simeq\mathbb R^d\), the bundle morphism \(F\) reduces to a linear cocycle $A$ on the space \(\mathrm{Mat}(d,\mathbb R)\) of $d\times d$ matrices. We study matrix products driven by a kernel \(Q(t,ds)\) on \(T\times\mathscr A\). Throughout this subsection, we assume that~\eqref{eq:integrability-nu} holds with \(F=A\) and \(\nu=p\), where \(p\) is a \(Q\)-stationary probability measure, i.e., $p=\int Q(t,\cdot) \ dp$. In this case, \(\hat{\mathcal E}_k\) denotes $\PP(\wedge^k\mathbb R^d)\) with a cemetery point \(\Delta\) adjoined in the singular case.
 

\subsection{Bernoulli case} In this setting we have an iid\ linear cocycle \(A:T\to\mathrm{Mat}(d,\mathbb R)\). Proposition~\ref{prop:bernoulli-variational} extends the classical Furstenberg variational formula from cocycles in \(\mathrm{GL}(d,\mathbb R)\) to singular cocycles:
\begin{align*}
\sum_{i=1}^k \lambda_i(m)
= 
\sup_{\hat\mu} 
\int \log\frac{\|\wedge^kA(t)v\|}{\|v\|}\,dp\,d\hat\mu 
\end{align*}
where $\hat \mu$ ranges over the $\bar{A}_k$-stationary measures and where $\bar{A}_k:T\times \hat{\mathcal{E}}_k \to \hat{\mathcal{E}}_k$ is the projective random map induced by $\wedge^k A$ and sending the degenerate locus to the fixed point $\Delta$. 
Recent work of Duarte and Graxinha~\cite{DG25} also treats non-invertible iid cocycles in \(\mathrm{Mat}(d,\mathbb R)\) under additional structural and integrability hypotheses. More precisely, in~\cite[Lem.~5.2]{DG25} they conclude the integral formula for the top Lyapunov exponent under the finite exponential moment assumption and for quasi-irreducible cocycles. By contrast, Proposition~\ref{prop:bernoulli-variational} yields a variational representation for possibly singular iid cocycles under the standard weak integrability assumption, without quasi-irreducibility or exponential-moment conditions.

On the other hand, Corollary~\ref{cor:bernoulli-sup-int} yields
\[
\sum_{i=1}^k \lambda_i(m)
=
\lim_{n\to\infty}
\sup_{[v]} \frac{1}{n}
\int
\log\frac{\|\wedge^k A^n(\omega)v\|}{\|v\|}
\,d\mathbb P.
\]
A related representation for the top Lyapunov exponent appears in~\cite[Thm.~6.2]{BM}, where it is obtained by a martingale argument via~\cite[Cor.~3.3]{BM}. Here it follows directly from the variational principle.

\subsection{Vertex case} In this setting we have a linear cocycle
$
A:T\to \mathrm{Mat}(d,\mathbb R)
$. 
Theorem~\ref{mainthm:B} provides the variational formula 
\begin{equation}\label{eq:matrix-vertex-formula}
\sum_{i=1}^k \lambda_i(m)
=
\sup_{\hat\nu}
\int
\log\frac{\|\wedge^kA(t)v\|}{\|v\|}
\,d\hat\nu(t,[v]),
\end{equation}
where the supremum is taken (and attained) over all (ergodic) probability measures
\(\hat\nu\) on $T\times \hat{\mathcal{E}}_k$ that are invariant under the Markov
operator
\[
\hat P_k\phi(t,[v])
=
\int
\phi\bigl(s,[\wedge^kA(t)v]\bigr)\,Q(t,ds)
\]
and project to the stationary measure \(p\) on \(T\). Moreover, any maximizing measure has zero mass on the degeneracy locus whenever \(\lambda_k(m)>-\infty\). 

When \(T\) is finite, say \(T=\{1,\dots,n\}\), \(Q(i,\{j\})=P_{ij}\) and $p=p_1\delta_1 + \dots + p_n \delta_n$, every admissible
lift \(\hat\nu\) can be written uniquely as
\[
\hat\nu=\sum_{i=1}^n p_i\,\delta_i\otimes \eta_i,
\]
where each \(\eta_i\) is a probability measure on \(\PP(\wedge^k\mathbb R^d)\). In this
notation, and assuming that $p_i>0$, for $i=1,\dots,n$, 
the invariance condition \(\hat P_k^*\hat\nu=\hat\nu\) is equivalent to
\[
\eta_j
=
\sum_{i=1}^n \frac{p_iP_{ij}}{p_j}\,
(\wedge^kA(i))_*\eta_i. 
\]
that is, \(\eta=(\eta_i)_{i=1}^n\) is a stationary measure vector. Therefore,
\eqref{eq:matrix-vertex-formula} can be rewritten as
\[
\sum_{i=1}^k \lambda_i(m)
=
\sup_{\eta}
\sum_{i=1}^n p_i
\int 
\log\frac{\|\wedge^kA(i)v\|}{\|v\|}
\,d\eta_i([v]),
\]
where the supremum is taken over all stationary measure vectors
\(\eta=(\eta_i)_{i=1}^n\) on \(\PP(\wedge^k\mathbb R^d)\). In particular, for \(k=1\) and
\(A(i)\in \mathrm{GL}(d,\mathbb R)\), this reduces to the representation formula for the top Lyapunov exponent by Malheiro and  Viana  \cite[Prop.~3.5]{MalheiroViana2015}.

\subsection{Edge case}
In this setting we have a linear cocycle $
A:T\times T\to \mathrm{Mat}(d,\mathbb R)$ 
driven by the transition kernel \(Q(t,A)\) on \(T\times\mathscr A\). Proposition~\ref{prop:edge-variational} gives a variational formula for the Lyapunov exponents of possibly singular edge Markovian matrix products, extending the Furstenberg-type formula established in~\cite[Thm.~3.3(iii)]{CDKM22} for Lipschitz invertible cocycles with compact support. 
More precisely, 
\begin{equation}\label{eq:matrix-markov-formula}
\sum_{i=1}^k \lambda_i(m)
=
\sup_{\hat\nu}
\int
\log\frac{\|\wedge^kA(t,s)v\|}{\|v\|}
\,Q(t,ds)\,d\hat\nu(t,[v]),
\end{equation}
where the supremum is taken (and attained)  over all (ergodic) probability measures \(\hat\nu\) on \(T\times \hat{\mathcal{E}}_k\) that are invariant under the annealed operator \(\hat P_k\) from~\eqref{eq:averaged-operator} and project to the stationary measure \(p\) on \(T\). Moreover, any maximizing measure has zero mass on the degeneracy locus whenever \(\lambda_k(m)>-\infty\). 
 This formula also generalizes the result of  Duarte et al.~\cite{DDGK25} for the top Lyapunov exponent of random $\mathrm{Mat}(2, \mathbb{R})$-valued cocycles with singular components and finite alphabet $T$.  By contrast,~\eqref{eq:matrix-markov-formula} does not require either restriction.


\section{Proof of Theorems~\ref{mainthm:A} and~\ref{thm:VP}}
\label{sec:ThmA}


Throughout the proof, let $\pi:\hat X \to X$ be a standard Borel bundle with compact metric fibers. From the definition, we have a compact metric space \(Y\), a Borel set \(\mathcal E\subset X\times Y\) with nonempty compact sections $\mathcal E_x$ and a Borel isomorphism $\iota:\hat X\to \mathcal E$. Moreover, the restriction of $\iota$ 
between fibers $\pi^{-1}(x)$ and $\mathcal{E}_x$ is a homeomorphism.

\subsection{Factorization and integral representation of $\sigma$-Markov operators} 
{The notion of a \emph{factor} (random semiconjugacy) is standard in the theory of random dynamical systems: it is given by a measurable family of surjective maps intertwining the cocycles; see, for instance, the explicit treatment in Kl\"unger~\cite[{\S3}]{klunger2001periodicity}. Passing to the skew-product, this becomes an ordinary semiconjugacy and hence an intertwining relation for the associated Koopman operators. In the Markov operator framework of this paper, we encode the same idea intrinsically: 

\enlargethispage{-1cm}
\begin{lemma}\label{lem:markov-factorization} Let  $\mathcal B_\pi=\pi^*(B(X))$ be the space of $\pi$-fiberwise constant functions in~\eqref{eq:set-F}. 
Consider a $\sigma$-Markov operator $\hat{P}: B(\hat X) \to B(\hat X)$. The following are equivalent:
\begin{enumerate}
    \item $\hat{P}(\mathcal{B}_\pi) \subset \mathcal{B}_\pi$;
    \item There exists an operator $P: B(X) \to B(X)$ such that $\hat{P} \circ \pi^* = \pi^* \circ P$.
\end{enumerate}
In this case, $P$ is unique, it is a $\sigma$-Markov operator, and it is given by
\[
 P = (\pi^*)^{-1} \circ \hat{P} \circ \pi^*,
\]
where $(\pi^*)^{-1}: \mathcal{B}_\pi \to B(X)$ is the inverse of the isomorphism $\pi^*: B(X) \to \mathcal{B}_\pi$.
Moreover, if $P^*$ and $\hat{P}^*$ denote the dual operators on probability measures, it holds
\[
    \pi_* \circ \hat{P}^* = P^* \circ \pi_*.
\]
In particular, if $\hat{\nu}$ is $\hat{P}^*$-invariant, then $\mu \coloneqq \pi_* \hat{\nu}$ is $P^*$-invariant.
\end{lemma}

\begin{proof}
Since $\pi$ is surjective, the pullback $\pi^*: B(X) \to B(\hat X)$ defined by $\pi^*\varphi = \varphi \circ \pi$ is an isometric lattice embedding. Consequently, $\pi^*$ is a bijection from $B(X)$ onto its range $\pi^*(B(X)) = \mathcal{B}_\pi$. Let $(\pi^*)^{-1}: \mathcal{B}_\pi \to B(X)$ denote its inverse. Thus, for every $f\in\mathcal B_\pi$, the function $\varphi=(\pi^*)^{-1}f$ is the unique element of $B(X)$ such that $f=\varphi\circ\pi$.

\noindent {(2) $\implies$ (1):} Suppose there exists $P: B(X) \to B(X)$ satisfying $\hat{P} \circ \pi^* = \pi^* \circ P$. Let $f \in \mathcal{B}_\pi$. By definition, $f = \pi^*\varphi$ for some $\varphi \in B(X)$. Then
\[
\hat{P}f = \hat{P}(\pi^*\varphi) = \pi^*(P\varphi).
\]
Since $P\varphi \in B(X)$, the right-hand side belongs to $\mathcal{B}_\pi$. Thus $\hat{P}(\mathcal{B}_\pi) \subset \mathcal{B}_\pi$.

\noindent {(1) $\implies$ (2):} Suppose $\hat{P}(\mathcal{B}_\pi) \subset \mathcal{B}_\pi$. For any $\varphi \in B(X)$, consider the function $\hat{P}(\pi^*\varphi)$. By hypothesis, this function belongs to $\mathcal{B}_\pi$. Therefore, we can define an operator $P: B(X) \to B(X)$ by $P\varphi \coloneqq (\pi^*)^{-1}(\hat{P}(\pi^*\varphi))$. 
By construction, $\pi^*(P\varphi) = \hat{P}(\pi^*\varphi)$, which is exactly the commutation relation $\pi^* \circ P = \hat{P} \circ \pi^*$.

It remains to prove that if these conditions hold, $P$ is unique and is a $\sigma$-Markov operator. Assume (2).
\begin{enumerate}[label=-,leftmargin=0.35cm]
     \item \emph{Uniqueness of $P$:} For all $\varphi \in B(X)$, $\pi^*(P\varphi) = \hat{P}(\pi^*\varphi)$. Since $\pi^*$ is injective (due to the surjectivity of $\pi$), $P\varphi$ is uniquely determined.
    \item \emph{Linearity and positivity of $P$:} this follows since $\hat{P}$ and $\pi^*$ are linear and positive, and $(\pi^*)^{-1}$ is linear and positive (if $f \in \mathcal{B}_\pi$ and $f \ge 0$, then its factorization $\varphi \ge 0$).
    \item \emph{Preservation of the unity:} Since $\hat{P}{1}_{\hat X} = {1}_{\hat X}$ and ${1}_{\hat X} = \pi^*({1}_X)$, we have $\pi^*(P{1}_X) = \hat{P}(\pi^*{1}_X) = \hat{P}{1}_{\hat X} = {1}_{\hat X} = \pi^*({1}_X)$. Injectivity of $\pi^*$ implies $P{1}_X = {1}_X$.
    \item \emph{$\sigma$-order continuity of $P$:} Let $\varphi_n \downarrow 0$ pointwise in $B(X)$. Then $\pi^*\varphi_n \downarrow 0$ pointwise in $B(\hat X)$. Since $\hat P$ is $\sigma$-order continuous, $\hat P(\pi^*\varphi_n)\downarrow 0$. Using the factorization, $\pi^*(P\varphi_n)\downarrow 0$. Since $\pi$ is surjective, pointwise convergence of the pullbacks implies pointwise convergence of the base functions. Thus $P\varphi_n\downarrow 0$.
\end{enumerate}
The duality on measures follows from the adjoint relation. For any $\varphi \in B(X)$ and $\hat{\nu} \in \mathcal{P}(\hat X)$,
\begin{align*}
    \int \varphi \, d(\pi_* \hat{P}^* \hat{\nu}) &= \int \pi^*\varphi \, d(\hat{P}^* \hat{\nu}) = \int \hat{P}(\pi^*\varphi) \, d\hat{\nu}
\\
&= \int \pi^*(P\varphi) \, d\hat{\nu} = \int P\varphi \, d(\pi_*\hat{\nu}) = \int \varphi \, d(P^*(\pi_*\hat{\nu})).
\end{align*}
Thus $\pi_* \circ \hat{P}^* = P^* \circ \pi_*$.
\end{proof}

{Assuming the Axiom of Choice, the dual of the space $B(Z)$ of bounded Borel functions on a standard Borel space $(Z,\mathscr B)$ contains Banach limits. As discussed in \cite{MOMarkovRep}, Markov operators on $B(Z)$ constructed from such functionals cannot be represented by transition probabilities. Recall that a \emph{transition probability kernel} $q(z,B)$ is a function such that $q(z,\cdot)$ is a Borel probability measure for every $z\in Z$ and $z\mapsto q(z,B)$ is measurable for every $B\in\mathscr B$. The key property that Banach limits fail, and which is needed to obtain an integral representation by a transition kernel, is $\sigma$-order continuity.}
 This contrasts with the setting of Markov operators on $L^\infty(m)$, where weak$^*$ continuity (inherent by duality with $L^1(m)$) automatically implies $\sigma$-order continuity. In such a context, every Markov operator in $L^\infty(m)$ is induced by an $m$-nonsingular transition probability and can be realized as the annealed operator of an iid\ random dynamical system (see \cite[Remark~2.4 and Section~1.3.4]{barrientos2025finitude}).

\begin{lemma}[{\cite[Thm.~19.10]{AliprantisBorder}}]  \label{lem:kernel-representation} Let $(Z,\mathscr B)$ be a standard Borel space. Then, 
 $Q:B(Z)\to B(Z)$ is a $\sigma$-Markov operator if and only if 
there exists a transition probability kernel $q:Z\times\mathscr{B} \to[0,1]$ such that
\[
(Q\varphi)(z)=\int \varphi(w)\,q(z,dw)\qquad \varphi\in B(Z),\ \,z\in Z.
\]
Moreover $q(z,A)=Q1_A(z)$ for every $A\in\mathscr{B}$, and the dual action on Borel probabilities is given by
\[
(Q^*\nu)(A)=\int q(z,A)\,d\nu(z).
\]
\end{lemma}

\subsection{Fiber-sup operator} 
{Given an extended real-valued function $u$ on $\hat X$, we define
\[
(\mathcal M u)(x) := \sup_{z \in \pi^{-1}(x)} u(z), \qquad x\in X.
\]
Since the fibers $\pi^{-1}(x)$ are compact, if $u$ is fiberwise upper semicontinuous, then the supremum is attained, and}
$$
(\mathcal M u)(x) = \max_{z \in \pi^{-1}(x)} u(z) \in [-\infty,\infty).
$$ 
For a general Borel function $u$, the function $\mathcal M u$ need not be Borel measurable. To ensure that $\mathcal M u$ can be integrated when evaluating $P$ (as required by Lemma~\ref{lem:kernel-representation}), we use the following standard notions from measure theory; see~\cite[\S12]{AliprantisBorder}.

Let $(S, \mathscr{S}, \mu)$ be a probability space and $(Z,\mathscr{B})$ a Polish space. A function $f: S \to Z$ is \emph{$\mu$-measurable} if $f^{-1}(B) \in \mathscr{S}_\mu$ for every $B \in \mathscr{B}$, where $\mathscr{S}_\mu$ is the $\mu$-completion of $\mathscr{S}$. The function $f$ is \emph{universally measurable} if it is $\nu$-measurable for every probability measure $\nu$ on~$(S,\mathscr{S})$.

\begin{lemma} \label{lem:propiedades-M} 
The function $\mathcal{M}u:X\to [-\infty,\infty]$ is universally measurable for every Borel measurable function $u: \hat X \to [-\infty,\infty]$. Moreover, for functions $u$, $v$ on $\hat X$ and $\varphi$ on $X$, the following properties  hold:
\begin{enumerate}[label=(\roman*)]
\item \emph{Monotonicity}: if $u \le v$, then $\mathcal{M} u \le \mathcal{M} v$;
\item \emph{Subadditivity}: $\mathcal{M}(u + v) \le \mathcal{M} u + \mathcal{M} v$ (whenever the sum is well-defined);
\item \emph{Normalization}:  $\mathcal{M}(\varphi \circ \pi) = \varphi$.
\end{enumerate}
\end{lemma}

\begin{proof} Item (i) follows immediately from taking suprema over the same fiber.  To show~(ii), fix $x \in X$. For any $z \in \pi^{-1}(x)$, we have $u(z) \le \mathcal{M}u(x)$ and $v(z) \le \mathcal{M}v(x)$. Summing these inequalities (whenever the sum is well-defined) yields
\[
    u(z) + v(z) \le \mathcal{M}u(x) + \mathcal{M}v(x).
\]
Taking the supremum over $z \in \pi^{-1}(x)$, we get $\mathcal{M}(u+v)(x) \le \mathcal{M}u(x) + \mathcal{M}v(x)$. For~(iii), let $\varphi: X \to [-\infty,\infty]$. If $u = \varphi \circ \pi$, then on any fiber $\pi^{-1}(x)$, the function $u$ is constant and equal to $\varphi(x)$. Therefore, the supremum is  $\varphi(x)$.

For the universal measurability claim, we identify $\hat X$ with the Borel subset $\mathcal E \subseteq X \times Y$ via the isomorphism $\iota$. Then $\pi$ corresponds to the projection $\mathrm{pr}_X:X\times Y\to X$. Now, for any $\alpha \in \mathbb{R}$, the level set of $\mathcal Mu$ is
    \[
    A(\alpha):=\{x \in X : (\mathcal M u)(x) > \alpha\} = \mathrm{pr}_X\left( \{(x, y) \in \mathcal{E} : u(x, y) > \alpha\} \right).
    \]
    Since $u$ is a Borel function, the set inside the projection $\mathrm{pr}_X$ is a Borel subset of $X \times Y$. By \cite[Thm.~12.27]{AliprantisBorder}, every Borel set in a Polish space is an analytic set. Furthermore, by \cite[Thm.~12.24]{AliprantisBorder}, the continuous image of an analytic set is analytic. Since $\mathrm{pr}_X$ is continuous, $A(\alpha)$ is an analytic subset of $X$.
    Finally, by \cite[Thm.~12.41]{AliprantisBorder}, every analytic subset of a Polish space is {universally measurable}. Thus, $\mathcal M u$ is universally measurable.
\end{proof}


The fiber-sup operator \(\mathcal M\) sends Borel functions on \(\hat X\) to universally measurable functions on \(X\) (Lemma~\ref{lem:propiedades-M}), and these need not be Borel. Since the domination argument below requires applying \(P\) and \(\hat P\) to such functions, we first extend \(\sigma\)-Markov operators from \(B(Z)\) to universally measurable functions taking values in \([-\infty,\infty)\).

\begin{lemma}\label{lem:extend-markov-univ}
Let \(Q:B(Z)\to B(Z)\) be \(\sigma\)-Markov operator where $Z$ is a standard Borel space. Then \(Q\) admits a natural extension to universally measurable functions \(f:Z\to[-\infty,\infty)\). Moreover, this extension is monotone: if \(f\le g\), then \(Qf\le Qg\).
\end{lemma}

\begin{proof} Let \(q(z,dw)\) be the transition kernel given by Lemma~\ref{lem:kernel-representation}.
Assume first that \(f:Z\to[-\infty,\infty)\) is universally measurable and bounded from above. For each \(z\in Z\), the function \(f\) is measurable with respect to the completion of \(\mathscr B\) under the probability measure \(q(z,\cdot)\). Hence
\[
Qf(z)=\int f(w)\,q(z,dw)
\]
is well defined in \([-\infty,\infty)\).

We claim that \(Qf\) is universally measurable. Let \(\nu\) be a probability measure on \((Z,\mathscr B)\), and define the probability measure \(\eta\) on \(Z\times Z\) by
\[
\eta(A):=\int q(z,A_z)\,d\nu(z),
\qquad
A_z:=\{w\in Z:(z,w)\in A\}.
\]
Since \(f\) is universally measurable, there exists a Borel function \(\tilde f:Z\to[-\infty,\infty)\) such that
\[
\tilde f(w)=f(w)
\qquad\text{for \(\eta\)-a.e. }(z,w)\in Z\times Z.
\]
By Fubini's theorem, this implies that
\[
\tilde f=f
\qquad\text{\(q(z,\cdot)\)-a.e. for \(\nu\)-a.e. }z\in Z.
\]
Therefore,
\[
Qf(z)=Q\tilde f(z)
\qquad\text{for \(\nu\)-a.e. }z.
\]
Since \(\tilde f\) is Borel, the function \(Q\tilde f\) is Borel measurable. Hence \(Qf\) is \(\nu\)-measurable. As \(\nu\) was arbitrary, \(Qf\) is universally measurable.

Monotonicity for bounded-above universally measurable functions is immediate: if \(f\le g\), then
\[
Qf(z)=\int f\,dq(z,\cdot)\le \int g\,dq(z,\cdot)=Qg(z)
\]
for every \(z\in Z\).

Now let \(f:Z\to[-\infty,\infty)\) be universally measurable but not necessarily bounded from above, and define \(f_N:=\min\{f,N\}\). Then \(f_N\uparrow f\) pointwise, and each \(f_N\) is universally measurable and bounded from above, so each \(Qf_N\) is well defined and universally measurable. We define
\[
Qf:=\sup_{N\ge1}Qf_N,
\qquad
Q^n f:=\sup_{N\ge1}Q^n f_N.
\]
Since each \(Q^n f_N\) is universally measurable, so is \(Q^n f\). Moreover, by the monotonicity already proved for bounded-above functions,
\[
f_N\le f_{N+1}
\quad\Longrightarrow\quad
Q^n f_N\le Q^n f_{N+1},
\]
so \(Q^n f_N\uparrow Q^n f\) pointwise as \(N\to\infty\). This proves the lemma.
\end{proof}

\begin{rem}\label{rem:extend-factor}
The extended operator still preserves the factor identity. More precisely, if
$\hat P\circ\pi^*=\pi^*\circ P$ on $B(X)$,
then for every universally measurable \(\varphi:X\to[-\infty,\infty)\) bounded from above,
$\hat P(\varphi\circ\pi)=(P\varphi)\circ\pi$.
Indeed, both sides are computed by integration against the same pushforward kernel on \(X\), and hence agree pointwise. By truncation, the same conclusion extends to $\varphi$.
\end{rem}

\begin{lemma}\label{lem:domination}
Let $P: B(X) \to B(X)$ and $\hat P: B(\hat X) \to B(\hat X)$ be $\sigma$-Markov operators such that  $\hat P\circ \pi^* = \pi^* \circ P$. Let $u: \hat X \to [-\infty, \infty)$ be a Borel measurable function. Then 
    \[
    \mathcal M(\hat P^n u) \le P^n(\mathcal M u) \qquad \text{for all $n\geq 1$}.
    \]
\end{lemma}

\begin{proof}
Let $u:\hat X\to[-\infty,\infty)$ be Borel. For each $N\ge1$ set $u_N:=\min\{u,N\}$. Then $u_N$ is Borel
and bounded from above, hence $\hat P^n u_N$ is well-defined for every $n\ge1$ and remains bounded from above.
By Lemma~\ref{lem:propiedades-M}, the function $\mathcal{M}u_N$ is universally measurable and bounded from above, and
by the definition of $\mathcal{M}$ we have $u_N\le (\mathcal{M}u_N)\circ\pi$ pointwise on $\hat X$. Applying the positive operator
$\hat P$ and using the (extended) factor property gives
\[
   \hat P u_N \le \hat P\bigl((\mathcal{M}u_N)\circ\pi\bigr) = \bigl(P(\mathcal{M}u_N)\bigr)\circ\pi.
\]
Applying $\mathcal{M}$ and using its monotonicity and normalization (Lemma~\ref{lem:propiedades-M}) yields
\[
   \mathcal{M}(\hat P u_N)\le \mathcal{M}\!\left(\bigl(P(\mathcal{M}u_N)\bigr)\circ\pi\right)=P(\mathcal{M}u_N).
\]
Since $\hat P u_N$ is again bounded from above, the same argument applies to $\hat P u_N$ in place of $u_N$
shows $\mathcal{M}(\hat P^2 u_N)\le P(\mathcal{M}(\hat P u_N))$, and combining with the previous inequality and monotonicity of $P$
gives $\mathcal{M}(\hat P^2 u_N)\le P^2(\mathcal{M}u_N)$. Iterating this reasoning yields, for every $n\ge1$,
\[
   \mathcal{M}(\hat P^n u_N)\le P^n(\mathcal{M}u_N).
\]
Now $u_N\uparrow u$ pointwise, hence by positivity of $\hat P^n$ we have $\hat P^n u_N\uparrow \hat P^n u$
(pointwise, by the truncation definition), and therefore $\mathcal{M}(\hat P^n u_N)\uparrow \mathcal{M}(\hat P^n u)$ by monotonicity
of $\mathcal{M}$. On the other hand, $\mathcal{M}u_N=\min\{\mathcal{M}u,N\}$, so $$P^n(\mathcal{M}u_N)\uparrow P^n(\mathcal{M}u)$$ by positivity of $P^n$ and the
truncation definition. Passing to the limit $N\to\infty$ in $\mathcal{M}(\hat P^n u_N)\le P^n(\mathcal{M}u_N)$ gives
\[
   \mathcal{M}(\hat P^n u)\le P^n(\mathcal{M}u)\qquad\text{for all }n\ge1,
\]
as claimed.
%
\end{proof}



The following  lemma is the main tool used to construct maximizing lifts in the proof of Theorem~\ref{thm:VP}.
The proof is based on Aumann's measurable selection theorem~\cite[Cor.~18.27]{AliprantisBorder}.

\begin{lemma}\label{lem:measurable-max}
Let $\mu$ be a probability measure on $X$ and consider a fiberwise upper semicontinuous function $u: \hat X \to [-\infty, \infty)$. Define
    \[
    K_u(x) := \{z \in \hat X_x : u(z) = (\mathcal M u)(x)\}.
    \]
    Then, there is a measurable function $s: X \to \hat X$ such that  $s(x) \in K_u(x)$ for $\mu$-a.e.~$x$.
\end{lemma}

\begin{proof}
    We identify $\hat X$ with the Borel subset $\mathcal{E} \subseteq X \times Y$ via the isomorphism $\iota$. The projection $\pi$ corresponds to the continuous projection $\mathrm{pr}_X: X \times Y \to X$. Now, 
    following the terminology of~\cite[\S1.3]{AliprantisBorder},
    we consider the correspondence $\mathcal{K}_u: X \twoheadrightarrow Y$  obtained by viewing \(K_u\) as a subset of \(\mathcal E\).  That is, the assignment to each $x \in X$~the~set 
    $$\mathcal{K}_u(x):=\{ y\in Y: (x,y) \in \iota(K_u(x))\}.$$ 
    Its graph is 
    \[
    \mathrm{Gr}(\mathcal{K}_u) :=\{(x, y) \in X \times Y : y \in \mathcal{K}_u(x)\} = 
    \{(x, y) \in \mathcal{E} : u(\iota^{-1}(x, y)) =(\mathcal M u)(x)\}.
    \]
    Since $u$ is Borel and $\mathcal M u$ is $\mu$-measurable, this graph belongs to the product $\sigma$-algebra $\Sigma_\mu \otimes \mathscr B$. Since the fibers are compact and $u$ is fiberwise upper semicontinuous, $\mathcal{K}_u(x)$ is nonempty for every $x$. We apply the Aumann Measurable Selection Theorem~{\cite[Cor.~18.27]{AliprantisBorder}}. This result ensures the  existence of a measurable function $\tilde{s}: X \to Y$ such that $\tilde{s}(x) \in \mathcal{K}_u(x)$ for $\mu$-a.e.~$x\in X$. The section $s(x) = \iota^{-1}(x, \tilde{s}(x))$ satisfies the requirement.
\end{proof}

\subsection{Proof of Theorem~\ref{mainthm:A}} 
If $(\hat\phi_n)_{n\ge1}$ is $\hat P$-additive, i.e.
$\hat\phi_{n+m}=\hat\phi_n+\hat P^n\hat\phi_m$, then by subadditivity of $\mathcal{M}$,
\[
   \phi_{n+m}=\mathcal{M}(\hat\phi_n+\hat P^n\hat\phi_m)\le \mathcal{M}\hat\phi_n+\mathcal{M}(\hat P^n\hat\phi_m)
\]
Applied Lemma~\ref{lem:propiedades-M}} to $u=\hat\phi_m$,
\[
   \mathcal{M}(\hat P^n\hat\phi_m)\le P^n(\mathcal{M}\hat\phi_m)=P^n\phi_m.
\]
Thus $\phi_{n+m}\le \phi_n+P^n\phi_m$ for all $n,m\ge1$, so $(\phi_n)_{n\ge1}$ is $P$-subadditive.

\subsection{Proof of Theorem~\ref{thm:VP}}


By Theorem~\ref{mainthm:A}, the sequence $\phi_n=\mathcal M\hat\phi_n$ is $P$-subadditive. In particular,
$\phi_n\le \phi_{n-1}+P^{n-1}\phi_1\le \sum_{j=0}^{n-1}P^j\phi_1$.
Taking positive parts and using positivity of $P$, we get
\[
\phi_n^+\le \sum_{j=0}^{n-1}(P^j\phi_1)^+\le \sum_{j=0}^{n-1}P^j(\phi_1^+).
\]
Since $\phi_1^+\in L^1(\mu)$ and $\mu\in\mathcal I(P)$, it follows that
\[
\int \phi_n^+\,d\mu \le \sum_{j=0}^{n-1}\int P^j(\phi_1^+)\,d\mu
= n\int \phi_1^+\,d\mu<\infty,
\]
hence $a_n:=\int \phi_n\,d\mu\in[-\infty,\infty)$ for all $n$. Moreover, integrating the $P$-subadditivity of $\phi_n$
gives $a_{n+m}\le a_n+a_m$. Therefore, by Fekete's lemma,
\[
\Lambda(\mu):=\inf_{n\ge1}\frac1n a_n=\lim_{n\to\infty}\frac1n a_n\in[-\infty,\infty).
\]


First we show that $\Lambda(\mu)$ is an  upper bound of the integrals $\int \hat \phi_1 \, d\hat \mu$ with $\hat \mu \in \mathcal I_\mu(\hat P)$. To do this, fix $\hat \mu \in \mathcal{I}_\mu(\hat{P})$. By the definition of $\mathcal{M}$, we have $\hat{\phi}_n \le \phi_n \circ \pi$ pointwise. Integrating with respect to $\hat \mu$ and using the marginal condition $\pi_* \hat \mu = \mu$, we get
\[
\int \hat{\phi}_n \, d\hat \mu \le \int \phi_n \circ \pi \, d\hat \mu = \int \phi_n \, d\mu.
\]
By the $\hat{P}$-additivity of $(\hat\phi_n)_{n\geq 1}$ and since  $\hat \mu\in \mathcal I(\hat P)$, we have $\int \hat{\phi}_n \, d\hat \mu = n \int \hat{\phi}_1 \, d\hat \mu$. Dividing by $n$ and letting $n \to \infty$ yields
\[
\int \hat{\phi}_1 \, d\hat \mu \le \Lambda(\mu).
\]
This concludes the required upper bound.

Now, we construct a maximizing measure, that is, a measure $\hat\nu \in \mathcal I_\mu(\hat P)$ such that $\Lambda(\mu) \leq  \int \hat \phi_1 \, d\hat \nu$. 
Fix $n \ge 1$. By Lemma~\ref{lem:measurable-max}, there is a measurable section $s_n: X \to \hat{X}$ such that $s_n(x) \in \pi^{-1}(x)$ and $\hat{\phi}_n(s_n(x)) = \phi_n(x)$ for $\mu$-a.e.~$x$. Define the lift measure and the Cesàro average
\[
\eta_n := (s_n)_* \mu \in \Prob_\mu(\hat{X}) \quad \text{and}\quad 
\hat{\nu}_n \coloneqq \frac{1}{n} \sum_{k=0}^{n-1} (\hat{P}^*)^k \eta_n.
\]
Note that $\hat{\nu}_n \in \Prob_\mu(\hat{X})$ because  $P^*\mu = \mu$ and $\pi_* \circ \hat P^* = P^* \circ \pi_*$ by Lemma~\ref{lem:markov-factorization}. Moreover, 
\[
\int \hat{\phi}_1 \, d\hat{\nu}_n = \frac{1}{n}  \int \sum_{k=0}^{n-1} \hat{P}^k \hat{\phi}_1 \, d\eta_n = \frac{1}{n} \int \hat{\phi}_n \, d\eta_n = \frac{1}{n} \int \hat{\phi}_n \circ s_n \, d\mu = \frac{1}{n} \int \phi_n \, d\mu.
\]
Taking the limit, we arrive at
\[
\lim_{n \to \infty} \int \hat{\phi}_1 \, d\hat{\nu}_n = \Lambda(\mu).
\]
By Proposition~\ref{prop:compact-metrizable}, $\Prob_\mu(\hat{X})$ is compact in the Young-measure topology. Let $(\hat{\nu}_{n_j})_{j\geq 1}$ be a convergent subsequence to $\hat\nu \in \Prob_\mu(\hat X)$. Since $\hat{\phi}_1$ is fiberwise upper semicontinuous function\ and bounded from above by $ \phi_1\circ\pi \le g\circ\pi$ with $g=\phi_1^+\in L^1(\mu)$, Proposition~\ref{lem:semicontinuity} implies
\begin{equation}\label{eq:energy-usc}
\Lambda(\mu) = \lim_{j \to \infty} \int \hat{\phi}_1 \, d\hat{\nu}_{n_j} \le \int \hat{\phi}_1 \, d\hat{\nu}.
\end{equation}
This establishes attainment, provided $\hat{\nu} \in \mathcal{I}(\hat P)$. 
If $\Lambda(\mu) = -\infty$,  any invariant measure satisfies the inequality trivially. If such an invariant measure does not exist, then by convention 
\[
\sup_{\hat\mu\in\mathcal I_\mu(\hat P)}
\int\hat\phi_1\,d\hat\mu
=
\max\left\{
\int\hat\phi_1\,d\hat\mu:
\hat\mu\in\mathcal I_\mu(\hat P)
\text{ ergodic}
\right\}
= 
-\infty
\]
and the inequality still holds.  Assume $\Lambda(\mu) > -\infty$. Then \eqref{eq:energy-usc} implies $\int \hat{\phi}_1 \, d\hat{\nu} > -\infty$. Condition (ii) ($\hat{\phi}_1|_K \equiv -\infty$) forces $\hat{\nu}(K) = 0$. To prove that $\hat \nu \in \mathcal{I}_\mu(\hat P)$ in this case, we show that $\int \hat P\psi \, d\hat \nu = \int \psi \, d\hat \nu$ for every $\psi \in \mathscr{C}_b(\hat X)$.   

First, consider the term $\int \psi \, d\hat{\nu}_{n_j}$. Since $\psi$ is a bounded Carathéodory function, the convergence $\hat{\nu}_{n_j} \to \hat{\nu}$ in the Young-measure topology implies by definition that
\begin{equation} \label{eq:psij}
        \lim_{j \to \infty} \int \psi \, d\hat{\nu}_{n_j} = \int \psi \, d\hat{\nu}.
\end{equation}
On the other hand, by Condition (i), every fiberwise discontinuity point of the function $\hat{P}\psi$ belongs to $K$.
Since $\hat{\nu}(K)=0$, we have that $\hat P \psi\in B(\hat X)$ is fiberwise continuous outside of $\hat\nu$-null set. Hence, since $\hat\nu_{n_j} \to \hat \nu$ in the Young-measure topology, by Proposition~\ref{lem:ae-continuity}, 
\begin{equation} \label{eq:Phij}
        \lim_{j \to \infty} \int \hat{P}\psi \, d\hat{\nu}_{n_j} = \int \hat{P}\psi \, d\hat{\nu}.
\end{equation}
To conclude, let us prove that 
\begin{equation} \label{eq:lim0}
    \lim_{j \to \infty} \int (\hat{P}\psi - \psi) \, d\hat{\nu}_{n_j} = 0.
\end{equation}
Combining this with~\eqref{eq:psij} and~\eqref{eq:Phij}, we conclude $\int (\hat{P}\psi - \psi) \, d\hat{\nu} = 0$ as required.

To prove~\eqref{eq:lim0}, by the definition of $\hat{\nu}_n$ and the duality, we have
\begin{align*}
    \int (\hat{P}\psi - \psi) \, d\hat{\nu}_n 
    &= \frac{1}{n} \sum_{k=0}^{n-1} \left( \int \hat{P}^{k+1}\psi \, d\eta_n - \int \hat{P}^k\psi \, d\eta_n \right) \\
    &= \frac{1}{n} \left( \int \hat{P}^n\psi \, d\eta_n - \int \psi \, d\eta_n \right).
\end{align*}
Since $\psi$ is bounded (say by $C$), the Markov property implies $\|\hat{P}^n \psi\|_\infty \le C$. Therefore,
\[
    \left| \int (\hat{P}\psi - \psi) \, d\hat{\nu}_n \right| \le \frac{2 C}{n}.
\]
 Taking the limit as $n \to \infty$ (and thus along the subsequence $n_j$), we obtain~\eqref{eq:lim0} and conclude the construction of a maximizing measure.  

 \begin{remark}[Maximizers avoid the singular set]\label{rem:maximizer-avoids-K}
Assume Condition~\textup{(ii)}: $\hat\phi_1|_K\equiv-\infty$.
If $\hat\nu\in\Prob_\mu(\hat X)$ satisfies $\int \hat\phi_1\,d\hat\nu>-\infty$, then $\hat\nu(K)=0$.
In particular, if $\Lambda(\mu)>-\infty$, every maximizing lift $\hat\nu\in\mathcal I_\mu(\hat P)$
(i.e.\ $\int \hat\phi_1\,d\hat\nu=\Lambda(\mu)$) satisfies $\hat\nu(K)=0$.
\end{remark}

It remains to prove the assertion concerning ergodic maximizers when
$\Lambda(\mu)>-\infty$. Consider the set of maximizing lifts
\[
\mathscr M_\mu
:=
\left\{
\hat\rho\in\mathcal I_\mu(\hat P):
\int\hat\phi_1\,d\hat\rho=\Lambda(\mu)
\right\}.
\]
This set is nonempty by the preceding construction. It is also convex:
if $\hat\rho_1,\hat\rho_2\in\mathscr M_\mu$ and $a\in[0,1]$, then
$a\hat\rho_1+(1-a)\hat\rho_2$ is an invariant lift of $\mu$ and
\[
\int\hat\phi_1\,d\bigl(a\hat\rho_1+(1-a)\hat\rho_2\bigr)
=
a\Lambda(\mu)+(1-a)\Lambda(\mu)
=
\Lambda(\mu).
\]

We claim that $\mathscr M_\mu$ is compact in the Young-measure
topology. Let $\hat\rho_j\in\mathscr M_\mu$ and suppose that
$\hat\rho_j\to\hat\rho\in\operatorname{Prob}_\mu(\hat X)$. By
Proposition~\ref{lem:semicontinuity},
\[
\Lambda(\mu)
=
\limsup_{j\to\infty}\int\hat\phi_1\,d\hat\rho_j
\leq
\int\hat\phi_1\,d\hat\rho.
\]
In particular, $\int\hat\phi_1\,d\hat\rho>-\infty$, so
$\hat\rho(K)=0$. Conditions~{\rm(i)}--{\rm(ii)} and
Proposition~\ref{lem:ae-continuity} then allow invariance to pass to
the limit, and hence $\hat\rho\in\mathcal I_\mu(\hat P)$. The upper
bound already proved for invariant lifts gives $\int\hat\phi_1\,d\hat\rho\leq\Lambda(\mu)$. 
Thus equality holds and $\hat\rho\in\mathscr M_\mu$. Therefore
$\mathscr M_\mu$ is closed in the compact space
$\operatorname{Prob}_\mu(\hat X)$, and hence is compact.

We have then that $\mathscr M_\mu$ has an extreme point
$\hat\rho_*$. We show that $\hat\rho_*$ is also extreme in
$\mathcal I_\mu(\hat P)$. Suppose that
$
\hat\rho_*
=
a\hat\nu_1+(1-a)\hat\nu_2$ with
$0<a<1$, where $\hat\nu_1,\hat\nu_2\in\mathcal I_\mu(\hat P)$. Since every
invariant lift satisfies $
\int\hat\phi_1\,d\hat\nu_i\leq\Lambda(\mu)$ for $i=1,2$, 
and $
\Lambda(\mu)
=
a\int\hat\phi_1\,d\hat\nu_1
+
(1-a)\int\hat\phi_1\,d\hat\nu_2$,
both integrals must equal $\Lambda(\mu)$. Hence
$\hat\nu_1,\hat\nu_2\in\mathscr M_\mu$, and the extremality of
$\hat\rho_*$ in $\mathscr M_\mu$ implies
$ \hat\nu_1=\hat\nu_2=\hat\rho_*$.
Thus $\hat\rho_*$ is extreme in $\mathcal I_\mu(\hat P)$, and
therefore ergodic. This completes the proof.

\appendix
\section{Classical pointwise Lyapunov exponents}
\label{ss:lyapunov-spectrum}

We record the classical construction of the Lyapunov spectrum that is used in the paper.  It is the exterior-power argument of Ruelle~\cite{ruelle1979ergodic}, with Kingman's theorem written for the Koopman Markov operator.  No projective compactification is needed here: singular bundle morphisms are handled simply by the convention \(\log 0=-\infty\).

Let \(f\) be a measure-preserving transformation of a standard probability space \((X,\mathscr B,\mu)\), let \(\pi:\mathcal E\to X\) be a measurable rank-\(d\) Euclidean vector bundle, and let \(F:\mathcal E\to\mathcal E\) be a linear bundle morphism covering \(f\).  Write
\[
F_x^n=F_{f^{n-1}(x)}\circ\cdots\circ F_x,
\qquad n\geq1,
\]
and assume
\begin{equation}
\label{eq:appendix-integrability}
\int \log^+\|F_x\|\,d\mu(x)<\infty.
\end{equation}

For \(k=1,\ldots,d\), set
\[
\phi_{k,n}(x):=\log\|\wedge^kF_x^n\|\in[-\infty,\infty),
\qquad n\geq1.
\]
If \(P\varphi=\varphi\circ f\) is the Koopman Markov operator, then
$\phi_{k,n+m}\leq \phi_{k,n}+P^n\phi_{k,m}$, because \(F_x^{n+m}=F_{f^n(x)}^m\circ F_x^n\).  Moreover,
\[
(\phi_{k,1})^+
=\log^+\|\wedge^kF_x\|
\leq k\log^+\|F_x\|,
\]
so the positive part is integrable.  Kingman's subadditive theorem for Markov operators~\cite[Theorem~D]{BM} therefore gives an \(f\)-invariant measurable function \(\Sigma_k:X\to[-\infty,\infty)\) such that
\begin{equation}
\label{eq:appendix-top-sums}
\Sigma_k(x)
=
\lim_{n\to\infty}\frac1n\log\|\wedge^kF_x^n\|
\qquad\text{for \(\mu\)-a.e. \(x\)},
\end{equation}
and
\begin{equation}
\label{eq:appendix-integrated-top-sums}
\Sigma_k(\mu):=\int\Sigma_k\,d\mu
=
\lim_{n\to\infty}\frac1n\int\log\|\wedge^kF_x^n\|\,d\mu
=
\inf_{n\geq1}\frac1n\int\log\|\wedge^kF_x^n\|\,d\mu.
\end{equation}
If \(\mu\) is ergodic, then \(\Sigma_k(x)=\Sigma_k(\mu)\) for \(\mu\)-a.e. \(x\).

Set \(\Sigma_0(x)=0\).  The pointwise Lyapunov exponents are defined by
\begin{equation}
\label{eq:def-lyapunov-sigma}
\lambda_k(x):=
\begin{cases}
\Sigma_k(x)-\Sigma_{k-1}(x),&\text{if \(\Sigma_{k-1}(x)>-\infty\)},\\
-\infty,&\text{if \(\Sigma_{k-1}(x)=-\infty\)},
\end{cases}
\qquad k=1,\ldots,d.
\end{equation}
Then
\begin{equation}
\label{eq:appendix-sum-spectrum}
\Sigma_k(x)=\sum_{i=1}^k\lambda_i(x)
\qquad\text{for \(\mu\)-a.e. \(x\)}.
\end{equation}

To identify these quantities with the usual singular-value exponents, let
\(
\sigma_1(L)\geq\cdots\geq\sigma_d(L)\geq0
\)
denote the singular values of a linear map \(L\).  Since
$\|\wedge^kL\|=\sigma_1(L)\cdots\sigma_k(L)$,
if \(\Sigma_{k-1}(x)>-\infty\), equations~\eqref{eq:appendix-top-sums} and~\eqref{eq:def-lyapunov-sigma} give
\[
\lambda_k(x)
=
\lim_{n\to\infty}\frac1n\log\sigma_k(F_x^n).
\]
If \(k\geq2\) and \(\Sigma_{k-1}(x)=-\infty\), then
\[
\sigma_k(F_x^n)
\leq \sigma_{k-1}(F_x^n)
\leq \|\wedge^{k-1}F_x^n\|^{1/(k-1)}.
\]
Dividing logarithms by \(n\) and using~\eqref{eq:appendix-top-sums} again shows that the same limit is \(-\infty\).  Consequently,
\begin{equation}
\label{eq:appendix-singular-value-spectrum}
\lambda_k(x)
=
\lim_{n\to\infty}\frac1n\log\sigma_k(F_x^n),
\qquad k=1,\ldots,d,
\end{equation}
for \(\mu\)-a.e. \(x\).  In particular,
\[
\infty>\lambda_1(x)\geq\cdots\geq\lambda_d(x)\geq-\infty.
\]

Finally, define
\[
\lambda_k(\mu):=\int\lambda_k(x)\,d\mu(x)\in[-\infty,\infty).
\]
The positive parts are integrable because \(\lambda_k\leq\lambda_1=\Sigma_1\).  Hence finite sums may be integrated in~\eqref{eq:appendix-sum-spectrum}, and
\begin{equation}
\label{eq:appendix-integrated-spectrum}
\Sigma_k(\mu)=\sum_{i=1}^k\lambda_i(\mu).
\end{equation}
Equivalently, if \(\Sigma_{k-1}(\mu)>-\infty\), then
\[
\lambda_k(\mu)=\Sigma_k(\mu)-\Sigma_{k-1}(\mu).
\]
If \(\Sigma_{k-1}(\mu)=-\infty\), then at least one of \(\lambda_1(\mu),\ldots,\lambda_{k-1}(\mu)\) is \(-\infty\); the pointwise ordering implies \(\lambda_k(\mu)=-\infty\).  Thus the growth rates of the exterior powers are precisely the partial sums of the classical Lyapunov spectrum, also for non-invertible or singular bundle morphisms.

\section{The Young-measure topology}
\label{appendix:Young-measure}

{Let $\pi:\hat X\to X$ be a standard Borel bundle with compact metric fibers. Thus $\hat X$ and $X$ are standard Borel spaces, and there exist a compact metric space $Y$, a Borel set $\mathcal E\subset X\times Y$, and a measurable isomorphism
$\iota:\hat X\to \mathcal E$
such that
$\pi=\mathrm{proj}_X\circ \iota$ and, for every $x\in X$, the map $\iota$ sends the fiber $\hat X_x:=\pi^{-1}(x)$ homeomorphically onto the nonempty compact section $\mathcal E_x:=\{y\in Y:(x,y)\in\mathcal E\}$.} 
\begin{rem}{For comparison, many sources introduce such objects intrinsically as \emph{Borel fields of compact metric spaces}~\cite{BorelField}. One assumes that there is a Borel map}
\[
    d:\hat X\times_\pi \hat X \longrightarrow [0,\infty),\qquad 
    \hat X\times_\pi \hat X:=\{(\hat z,\hat z')\in\hat X^2:\ \pi(\hat z)=\pi(\hat z')\},
\]
whose restriction to $\hat X_x\times\hat X_x$ is a compact metric $d_x$ generating the fiber topology, and that
there exists a sequence of Borel sections $s_n:X\to\hat X$ such that $\{s_n(x):n\in\mathbb N\}$ is dense in
$(\hat X_x,d_x)$ for every $x\in X$.
{Under this intrinsic hypothesis, one obtains an embedding $\hat X\hookrightarrow X\times[0,1]^{\mathbb N}$ by the usual
``distance-to-sections'' coordinates. Conversely, the concrete realization $\hat X\cong\mathcal E\subset X\times Y$
implies the existence of a fundamental sequence of Borel sections by the Castaing representation theorem~\cite[Theorem III.9]{Castaing77}.}
\end{rem}

 We write
\[
    \Prob_\mu(\hat X)
    :=\{\nu\in\Prob(\hat X):\ \pi_*\nu=\mu\}.
\]
We endow this set with the topology obtained by testing against Borel bounded fiberwise continuous observables. 
Concretely, for $\nu\in\Prob_\mu(\hat X)$ 
we say that $\nu_n\to\nu$ in the \emph{Young-measure topology} (or \emph{stable topology}) if
\[
    \int \psi \,d\nu_n \ \longrightarrow\
    \int \psi \,d\nu
    \qquad\text{for every }\psi \in \mathscr{C}_b(\hat X).
\]
Here, since the Borel isomorphism $\iota: \hat X \to \mathcal{E}$ is a fiberwise homeomorphism, we have that   
$\mathscr{C}_b(\hat X)=\iota^*(\mathscr{C}_b(X\times Y))$ where $\mathscr{C}_b(X\times Y)$ denotes the set of measurable functions $f:X \times Y \to \mathbb{R}$ that are bounded and the section $f(x,\cdot) :Y\to \mathbb{R}  $ continuous for every $x\in X$. These functions are called in the literature \emph{bounded Carath\'eodory integrands}~\cite[Def.~3.7]{florescu2012young}. This topology corresponds with the stable topologies on Young measures~\cite[pg.~20-21]{castaing2004young} which are all of them equivalent in our context~\cite[Rem.~2.1.1]{castaing2004young}.

\begin{proposition}\label{prop:compact-metrizable}
The space \(\Prob_\mu(\hat X)\), endowed with the Young-measure topology, is compact and metrizable.
\end{proposition}

\begin{proof}
We argue by reduction to the trivial product $X\times Y$.
According to~\cite[Prop.~3.25 and Thm.~3.58]{florescu2012young}, since $X$ is standard Borel (hence it has countably generated $\sigma$-algebra) and $Y$ is compact metric (hence compact metrizable Suslin space), $\Prob_\mu(X\times Y)$ is a compact metrizable space with the Young-measure topology. 

Let
$\Prob_\mu(\mathcal E):=\{\eta\in\Prob_\mu(X\times Y): \eta(\mathcal E)=1\}$. 
We claim that $\Prob_\mu(\mathcal E)$ is closed in $\Prob_\mu(X\times Y)$ for the Young-measure topology.
Indeed, the indicator $1_{\mathcal E}(x,\cdot)$ is upper semicontinuous on $Y$ for each $x$
(because $\mathcal E_x$ is compact, hence closed), and $(x,y)\mapsto   1_{\mathcal E}(x,y)$ is measurable.
Thus $  1_{\mathcal E}$ is a bounded upper semicontinuous integrand.  If $\eta_n\to\eta$ stably and $\eta_n(\mathcal E)=1$ for all $n$, then the Portmanteau theorem~\cite[Thm.~2.1.3 (item 3)]{castaing2004young}
yields
\[
    1=\limsup_{n\to\infty}\eta_n(\mathcal E)
    =\limsup_{n\to\infty}\int   1_{\mathcal E}\,d\eta_n
    \le \int   1_{\mathcal E}\,d\eta
    =\eta(\mathcal E),
\]
hence $\eta(\mathcal E)=1$, proving the claim. In particular, $\Prob_\mu(\mathcal E)$ is compact. 

Finally, we transfer these results from the trivial bundle $\mathrm{proj}: X\times Y \to X$ to the bundle $\pi:\hat X\to X$. The map $\iota:\hat X \to \mathcal{E}$ is a bimeasurable bijection  with
$\pi=\mathrm{proj}_X\circ\iota$, hence pushforward is a bijection $\iota_*:\Prob_\mu(\hat X)\to \Prob_\mu(\mathcal E)$. By construction, the Young-measure topology on $\Prob_\mu(\hat X)$ is the initial topology for $\iota_*$ viewed as a map into the stable topology on $\Prob_\mu(X\times Y)$; therefore
$\iota_*$ is a homeomorphism of $\Prob_\mu(\hat X)$ onto the closed subspace $\Prob_\mu(\mathcal E)$.
Since $\Prob_\mu(X\times Y)$ is compact metrizable and $\Prob_\mu(\mathcal E)$ is closed,
it follows that $\Prob_\mu(\hat X)$ is compact metrizable as well.
\end{proof}

In the previous proof, we used the characterization of the convergence in the Young-measure topology by means of the upper semicontinuity of the functional 
\[
    \mathcal J_f:\Prob_\mu(X\times Y) \to [-\infty,\infty),
    \qquad \mathcal J_f(\nu):=\int f \,d\nu
\]
for every fiberwise upper semicontinuous function  
$f: X\times Y  \to [-\infty,\infty)$  with $|f(x,y)| \leq g(x)$ for some $ g\in L^1(\mu)$, i.e., \emph{$L^1$-bound upper semicontinuous integrand} in the sense of~\cite{AliprantisBorder,florescu2012young}. That is,   $\nu_n\to\nu$ in the
Young-measure topology, if~and~only~if
\[
    \limsup_{n\to\infty}\int f\,d\nu_n\ \le\ \int f\,d\nu
\]
for every $L^1$-bound upper semicontinuous integrand $f$~\cite[Thm.~2.1.3 (item~3)]{castaing2004young}. 
In our setting, this equivalence still holds for upper semicontinuous integrands dominated from above by an $L^1$ function.

\enlargethispage{1cm}
 \begin{prop}\label{lem:semicontinuity}
Let $u:\hat X\to[-\infty,\infty)$ be  fiberwise upper semicontinuous.
Assume there exists $g\in L^1(\mu)$ such that $u\leq g\circ \pi$. Then the functional
\[
    \mathcal J_u:\Prob_\mu(\hat X)\to[-\infty,\infty),\qquad \mathcal J_u(\nu):=\int u\,d\nu,
\]
 is upper semicontinuous for the Young-measure topology.
\end{prop}

 \begin{proof}
By the definition of the Young-measure topology, it suffices to prove the result on $\Prob_\mu(X\times Y)$. 
Thus, assume that $u: X \times Y \to [-\infty,\infty)$ is an upper semicontinuous integrand dominated from above by $g\in L^1(\mu)$.  Fix $M\ge 1$ and set $u_M:=\max\{u,-M\}$.  For each $x$, $u_M(x,\cdot)$ is upper semicontinuous on $Y$ and $-M\le u_M(x,y)\le g(x)$. In particular, $|u_M(x,y)|\le g(x)+M$ for all $(x,y)$. 
Hence $u_M$ is an $L^{1}$-bounded upper semicontinuous integrand.
Therefore, $\limsup_{n\to\infty}\int u_M\,d\nu_n\ \le\ \int u_M\,d\nu$. Since $u\le u_M$, we have $\int u\,d\nu_n\le \int u_M\,d\nu_n$ for every $n$, and hence
$$
\limsup_{n\to\infty}\int u\,d\nu_n\ \le\ \int u_M\,d\nu.
$$
Now let $M\to\infty$.
Because $u_M\downarrow u$ pointwise and $u_M\le g$, the nonnegative functions $w_M:=g-u_M  \uparrow g-u$
increase monotonically, and by monotone convergence,
\[
\int w_M\,d\nu\ \uparrow\ \int (g-u)\,d\nu\in[0,\infty].
\]
Since $\int g\,d\nu=\int g\,d\mu<\infty$, this implies $\int u_M\,d\nu\downarrow \int u\,d\nu$ (possibly $-\infty$).
Taking $\inf_{M}$ in the previous limsup bound yields
\[
\limsup_{n\to\infty}\int u\,d\nu_n\ \le\ \inf_{M}\int u_M\,d\nu\ =\ \int u\,d\nu. \qedhere
\]
\end{proof}



 The following result is a classical consequence of the Portmanteau theorem in the context of weak convergence of probability measures: convergence holds for any bounded integrand that is continuous outside a set of measure zero. While standard for probability measures, the equivalence between the convergence in Young-measure topology with the convergence of integrals for every bounded integrand that is fiberwise continuous almost everywhere with respect to the limit measure
  is not always explicitly detailed in the literature. Therefore, we include a proof here.

\begin{prop}
\label{lem:ae-continuity}
Let $(\nu_n)$ be a sequence in $\Prob_\mu(\hat X)$ and let $\nu \in \Prob_\mu(\hat X)$. The following are equivalent:
\begin{enumerate}[label=\textup{(\alph*)}]
    \item $\nu_n \to \nu$ in the Young-measure topology.
    \item For every bounded measurable function $f: \hat X \to \mathbb{R}$ such that the set of its fiberwise discontinuity points
    \[
        D_f := \{z \in \hat X :   f|_{\pi^{-1}(\pi(z))}  \text{ is discontinuous at } z\}
    \]
    is a $\nu$-null set, we have
    \[
        \int f \,d\nu_n \longrightarrow \int f \,d\nu.
    \]
\end{enumerate}
\end{prop}

\begin{proof}
The implication \textup{(b)}\(\Rightarrow\)\textup{(a)} follows directly
from the definition of the Young-measure topology, since $D_f = \emptyset$ for every \(f\in\mathscr C_b(\hat X)\). 

Assume \textup{(a)} and let \(f\) satisfy the hypothesis in
\textup{(b)}. As in the proof of Proposition~\ref{prop:compact-metrizable}, we argue by reduction to the trivial product. Let $\iota: \hat X \to \mathcal{E} \subset X \times Y$ be the Borel isomorphism satisfying $\pi = \mathrm{proj}_X \circ \iota$.  Define $\tilde{f}: \mathcal{E} \to \mathbb{R}$ by $\tilde{f} = f \circ \iota^{-1}$. For \((x,y)\in\mathcal E\), define the lower and upper envelopes relative to the compact section \(\mathcal E_x\) by 
\begin{align*}
\tilde{f}_*(x,y) := \lim_{r \downarrow 0} \inf_{\substack{z\in \mathcal{E}_z \\ d(z,y) < r}} \tilde{f}(x,z), 
    \qquad 
    \tilde{f}^*(x,y) := \lim_{r \downarrow 0} \sup_{\substack{z\in \mathcal{E}_z \\ d(z,y) < r}}  \tilde{f}(x,z).
\end{align*}
The envelopes are bounded and universally measurable.
Moreover, \(\tilde{f}_*(x,\cdot)\) is lower semicontinuous and
\(\tilde{f}^*(x,\cdot)\) is upper semicontinuous on \(\mathcal E_x\). Pulling
them back through the fiberwise homeomorphism \(\iota\), we obtain bounded universally measurable
functions \(f_*,f^*:\hat X\to\mathbb R\) that are respectively
fiberwise lower and upper semicontinuous and satisfy
$f_*\le f\le f^*$. 

Set $M:=\|f\|_\infty$. By construction, $-M\le f_*\le f\le f^*\le M$. In particular, both $f^*$ and $-f_*$ are bounded above by
$g\circ\pi$, where $g\equiv M\in L^1(\mu)$, since $\mu$ is a
probability measure. Proposition~\ref{lem:semicontinuity}, applied
first to $f^*$ and then to $-f_*$, therefore yields 
\[
\limsup_{n\to\infty}\int f^*\,d\nu_n
\le
\int f^*\,d\nu
\quad \text{and} \quad 
\liminf_{n\to\infty}\int f_*\,d\nu_n
\ge
\int f_*\,d\nu.
\]
At every fiberwise continuity point of \(f\), one has
\(f_*=f=f^*\). Since \(\nu(D_f)=0\),
\[
\int f_*\,d\nu
=
\int f\,d\nu
=
\int f^*\,d\nu.
\]
Therefore,
\[
\int f\,d\nu
\le
\liminf_{n\to\infty}\int f\,d\nu_n
\le
\limsup_{n\to\infty}\int f\,d\nu_n
\le
\int f\,d\nu,
\]
which proves the desired convergence.
\end{proof}


\section{A uniform ergodic-optimization consequence}
\label{appendix-A}

{Assume that \(X\) and \(\hat X\) are compact metric spaces and that \(\pi:\hat X\to X\) is a continuous surjection. In this setting, the hypotheses of Theorem~\ref{thm:VP} are automatically satisfied for Markov operators on \(C(X)\) and \(C(\hat X)\). Indeed, if \(Z\) is compact metric and \(Q:C(Z)\to C(Z)\) is a Markov operator, then \(Q\) admits a transition-kernel representation and therefore extends canonically to a \(\sigma\)-Markov operator on \(B(Z)\); see \cite[Thm.~1.6 and Thm.~1.7]{foguel1973ergodic}. Applying this to \(P:C(X)\to C(X)\) and \(\hat P:C(\hat X)\to C(\hat X)\), the intertwining identity \(\hat P\circ\pi^*=\pi^*\circ P\) extends from \(C(X)\) to \(B(X)\), so hypotheses {\rm(i)}--{\rm(ii)} of Theorem~\ref{thm:VP} hold with \(K=\emptyset\). Moreover, continuity of \(\pi\) is equivalent to \emph{upper hemicontinuity} of the set-valued map \(x\mapsto \pi^{-1}(x)\). That is, for every $x \in X$ and every open set $V \subset \hat{X}$ containing the fiber $\pi^{-1}(x)$, there exists a neighborhood $U$ of $x$ such that $\pi^{-1}(z) \subset V$ for all $z \in U$. Hence, by Berge's Maximum Theorem~\cite[Thm.~2, Chapter~VI, \S3]{Berge63}, it follows that \(\mathcal M\) sends upper semicontinuous functions on \(\hat X\) to upper semicontinuous functions on \(X\). Combining Theorem~\ref{thm:VP} with the uniform Kingman theorem for Markov operators \cite[Thm.~C]{BM}, we obtain the following corollary.}

\begin{corollary}
\label{thm:Uniform-Limit-Topological}
Let \(\pi:\hat X\to X\) be a continuous surjection between compact metric spaces. 
Let \(P:C(X)\to C(X)\) and \(\hat P:C(\hat X)\to C(\hat X)\) be Markov operators satisfying \(\hat P\circ \pi^*=\pi^*\circ P\). 
Given an upper semicontinuous potential \(\hat\phi_1:\hat X\to[-\infty,\infty)\), define
\[
\hat\phi_n:=\sum_{j=0}^{n-1}\hat P^{\,j}\hat\phi_1,
\qquad
\phi_n:=\mathcal M\hat\phi_n.
\]
Then, for every \(\mu\in\mathcal I(P)\), the conclusions of Theorem~\ref{thm:VP} apply, and each \(\phi_n\) is upper semicontinuous.
Moreover,
\[
\Lambda := \sup_{\mu \in \mathcal I(P)} \Lambda(\mu)
=
\max \left\{\int \hat\phi_1 \, d\hat{\mu} : \hat{\mu} \in \mathcal{I}(\hat{P}) \ \text{ergodic} \right\},
\]
and the following properties hold,
\begin{enumerate}
    \item {\emph{Uniform growth rate:}}
    \begin{align*}
        \Lambda &= \lim_{n \to \infty} \max_{{x} \in {X}} \frac{1}{n}  \phi_n({x}) = \inf_{n \geq 1} \max_{{x} \in {X}} \frac{1}{n}  \phi_n({x}).
    \end{align*}
    \item {\emph{Pointwise growth rates:}}
    \[
        \Lambda = \sup_{x \in X} \limsup_{n \to \infty} \frac{1}{n} \phi_n(x)
        = \sup_{\hat{x} \in \hat{X}} \limsup_{n \to \infty} \frac{1}{n} \hat\phi_n(\hat{x}).
    \]
\end{enumerate}    
\end{corollary}

\section{Dual variational principles and minimal growth}
\label{app:dual-min-growth}

This appendix records the lower-growth counterparts of Theorems~\ref{mainthm:A}
and~\ref{thm:VP}.  

\subsection{The abstract lower variational principle}

For a fiberwise lower semicontinuous function
$u:\hat X\to(-\infty,\infty]$, define
\[
 \underline{\mathcal M}u(x)
 :=\min_{z\in\pi^{-1}(x)}u(z).
\]
The minimum is attained because the fibers are compact.  Notice that
\begin{equation}\label{eq:fiber-min-duality}
 \underline{\mathcal M}u=-\mathcal M(-u).
\end{equation}
We use the convention
$u^-:=\max\{-u,0\}$ for the negative part of an extended real-valued
function. Also throughout this appendix, we use the conventions
$ \inf\emptyset=\min\emptyset:=\infty$.

\begin{theorem}
\label{prop:fiber-min-superadditive}
Let $P:B(X)\to B(X)$ and $\hat P:B(\hat X)\to B(\hat X)$ be
$\sigma$-Markov operators satisfying
$\hat P\circ\pi^*=\pi^*\circ P$.
If $(\hat\psi_n)_{n\ge1}$ is a $\hat P$-additive sequence of measurable
functions $\hat\psi_n:\hat X\to(-\infty,\infty]$, then the sequence $(\psi_n)_{n\geq 1}$, where 
$\psi_n:=\underline{\mathcal M}\hat\psi_n$,
is $P$-superadditive. 
\end{theorem}

\begin{proof}
The sequence $(-\hat\psi_n)_{n\ge1}$ is $\hat P$-additive.  By
Theorem~\ref{mainthm:A},
$\mathcal M(-\hat\psi_n)$ is $P$-subadditive.  The conclusion follows
from~\eqref{eq:fiber-min-duality}.
\end{proof}

Let $\mu\in\mathcal I(P)$.  If $(\psi_n)_{n\ge1}$ is
$P$-superadditive and $\psi_1^-\in L^1(\mu)$, then the superadditive
Fekete lemma gives
\begin{equation}\label{eq:def-lower-growth-rate}
 \underline\Lambda(\mu;P,\Psi)
 :=\lim_{n\to\infty}\frac1n\int\psi_n\,d\mu
 =\sup_{n\ge1}\frac1n\int\psi_n\,d\mu
 \in(-\infty,\infty].
\end{equation}
If $\mu$ is ergodic, the superadditive form of Kingman's theorem gives
\begin{equation}\label{eq:lower-kingman-pointwise}
 \underline\Lambda(\mu;P,\Psi)
 =\lim_{n\to\infty}\frac1n\psi_n(x)
 \qquad\text{for $\mu$-a.e. }x.
\end{equation}

\begin{theorem}
\label{thm:dual-VP}
Assume the setting of Theorem~\ref{thm:VP}.  Let
$\hat\psi_1:\hat X\to(-\infty,\infty]$ be fiberwise lower
semicontinuous and define
\[
 \hat\psi_n:=\sum_{j=0}^{n-1}\hat P^{\,j}\hat\psi_1,
 \qquad
 \psi_n:=\underline{\mathcal M}\hat\psi_n.
\]
Let $\mu\in\mathcal I(P)$ and assume that $\psi_1^-\in L^1(\mu)$.
Suppose that there is a Borel set $K\subset\hat X$ such that 
\begin{enumerate}[label=\textup{(\roman*)}]
\item every fiberwise discontinuity point of \(\hat Ph\) belongs to
 \(K\), for every \(h\in\mathscr C_b(\hat X)\);
 \item $\hat\psi_1(z)=+\infty$ for every $z\in K$.
\end{enumerate}
Then
\begin{equation}\label{eq:dual-VP-fixed-measure}
 \underline\Lambda(\mu)
 :=\lim_{n\to\infty}\frac1n\int\psi_n\,d\mu
 =\sup_{n\ge1}\frac1n\int\psi_n\,d\mu
 =\inf_{\hat\mu\in\mathcal I_\mu(\hat P)}
       \int\hat\psi_1\,d\hat\mu.
\end{equation}
Moreover,
\begin{equation}\label{eq:dual-VP-ergodic-minimizer}
 \underline\Lambda(\mu)
 =\min\left\{
       \int\hat\psi_1\,d\hat\mu:
       \hat\mu\in\mathcal I_\mu(\hat P)\text{ ergodic}
      \right\}.
\end{equation}
If $\underline\Lambda(\mu)<+\infty$, every minimizing lift gives zero
mass to $K$.
\end{theorem}

\begin{proof}
Apply Theorem~\ref{thm:VP} to the potential
$\hat\phi_1:=-\hat\psi_1$.  It is fiberwise upper semicontinuous,
$\mathcal M\hat\phi_n=-\psi_n$, and
$(\mathcal M\hat\phi_1)^+
 =(-\psi_1)^+=\psi_1^-\in L^1(\mu)$.
Condition~\textup{(ii)} becomes $\hat\phi_1=-\infty$ on $K$.  Hence
Theorem~\ref{thm:VP} gives
\[
 -\underline\Lambda(\mu)
 =\sup_{\hat\mu\in\mathcal I_\mu(\hat P)}
   \int(-\hat\psi_1)\,d\hat\mu.
\]
Multiplying by $-1$ proves~\eqref{eq:dual-VP-fixed-measure} and
\eqref{eq:dual-VP-ergodic-minimizer}.  The assertion concerning $K$
follows from the corresponding assertion in Theorem~\ref{thm:VP}.
\end{proof}


\subsection{Uniform minimal growth on compact spaces}

The next result is the exact lower counterpart of
Corollary~\ref{thm:Uniform-Limit-Topological}.

\begin{corollary}
\label{cor:uniform-dual-VP}
Let $\pi:\hat X\to X$ be a continuous surjection between compact
metric spaces.  Let $P:C(X)\to C(X)$ and
$\hat P:C(\hat X)\to C(\hat X)$ be Markov operators satisfying
$\hat P\circ\pi^*=\pi^*\circ P$.  Let
$\hat\psi_1:\hat X\to(-\infty,\infty]$ be lower semicontinuous and
put
\[
 \hat\psi_n:=\sum_{j=0}^{n-1}\hat P^{\,j}\hat\psi_1,
 \qquad
 \psi_n:=\underline{\mathcal M}\hat\psi_n.
\]
Then every $\psi_n$ is lower semicontinuous and, for every
$\mu\in\mathcal I(P)$, Theorem~\ref{thm:dual-VP} applies with
$K=\emptyset$.  Furthermore,
\begin{align}
 \underline\Lambda
 &:=\inf_{\mu\in\mathcal I(P)}\underline\Lambda(\mu)
   =\min\left\{
       \int\hat\psi_1\,d\hat\mu:
       \hat\mu\in\mathcal I(\hat P)\text{ ergodic}
     \right\},
 \label{eq:uniform-dual-invariant}\\[1mm]
 \underline\Lambda
 &=\lim_{n\to\infty}\min_{x\in X}\frac1n\psi_n(x)
   =\sup_{n\ge1}\min_{x\in X}\frac1n\psi_n(x),
 \label{eq:uniform-dual-finite-time}\\[1mm]
 \underline\Lambda
 &=\inf_{x\in X}\liminf_{n\to\infty}\frac1n\psi_n(x)
   =\inf_{\hat x\in\hat X}
      \liminf_{n\to\infty}\frac1n\hat\psi_n(\hat x).
 \label{eq:uniform-dual-pointwise}
\end{align}
\end{corollary}

\begin{proof}
Apply Corollary~\ref{thm:Uniform-Limit-Topological} to
$-\hat\psi_1$ and use~\eqref{eq:fiber-min-duality}.  The lower
semicontinuity of $\psi_n$ follows equivalently from the upper
semicontinuity of $-\psi_n$.
\end{proof}

Equations~\eqref{eq:dual-VP-fixed-measure},
\eqref{eq:uniform-dual-finite-time}, and
\eqref{eq:uniform-dual-pointwise} are the forms that are most useful
for lower semicontinuity and uniform expansion arguments.  In
particular, whenever the finite-time functions depend continuously on
a parameter, the representation as a supremum in
\eqref{eq:dual-VP-fixed-measure} or
\eqref{eq:uniform-dual-finite-time} immediately gives lower
semicontinuity of the minimal growth rate.

\subsection{Projective and Grassmannian formulas}
\label{subsec:dual-projective-grassmannian}

We now state the consequences used for Lyapunov exponents.  The
invertibility assumption below is essential for this lower-growth
formulation: it makes the projective and Grassmannian cocycles globally
defined and the logarithmic potentials finite.

Let $F:T\times\mathcal E\to\mathcal E$ be a random linear bundle
morphism in the setting of \S\ref{s:random-morphisms}, and assume that
$F_{t,x}$ is invertible for $\nu$-a.e. $(t,x)$ and
\begin{equation}\label{eq:dual-integrability-random}
 \int\bigl(\log^+\|F_{t,x}\|+
            \log^+\|F_{t,x}^{-1}\|\bigr)\,d\nu<\infty.
\end{equation}
For $1\le k\le d$, retain the projective operator $\hat P_k$ and
potential $\hat\phi_k$ defined in
\eqref{eq:fully-place-Phat}--\eqref{eq:fully-place-phihat}, now without
a cemetery section.  Define
\begin{equation*}\label{eq:lower-random-growth}
 \underline\Lambda_k(\nu)
 :=\inf_{\hat\nu\in\mathcal I_\nu(\hat P_k)}
       \int\hat\phi_k\,d\hat\nu.
\end{equation*}

\begin{corollary}
\label{cor:lower-projective-random}
Under~\eqref{eq:dual-integrability-random},
\begin{align*}
 \underline\Lambda_k(\nu)
 &=\min\left\{
      \int\hat\phi_k\,d\hat\nu:
      \hat\nu\in\mathcal I(\hat P_k),\
      (\pi_k)_*\hat\nu=\nu,\
      \hat\nu\text{ ergodic}
    \right\},
 \\
 &=\lim_{n\to\infty}\frac1n
   \int
   \min_{[v]\in\mathbb P(\wedge^k\mathcal E_x)}
   \mathbb E\left[
      \log\frac{\|\wedge^kF^n_{(\omega,x)}v\|}{\|v\|}
      \;\middle|\; (\omega_0,x)=(t,x)
   \right]d\nu
 \\
 &=\sup_{n\ge1}\frac1n
   \int
   \min_{[v]\in\mathbb P(\wedge^k\mathcal E_x)}
   \mathbb E\left[
      \log\frac{\|\wedge^kF^n_{(\omega,x)}v\|}{\|v\|}
      \;\middle|\; (\omega_0,x)=(t,x)
   \right]d\nu.
 \nonumber
\end{align*}
If $\nu$ is ergodic, 
\[
\underline\Lambda_k(\nu)=\lim_{n\to\infty}\frac1n
   \min_{[v]\in\mathbb P(\wedge^k\mathcal E_x)}
   \mathbb E\left[
      \log\frac{\|\wedge^kF^n_{(\omega,x)}v\|}{\|v\|}
      \;\middle|\; (\omega_0,x)=(t,x)
   \right] \quad \text{for $\nu$-a.e.~$(t,x)$}.
\]

\end{corollary}

\begin{proof}
Apply Theorem~\ref{thm:dual-VP} to the projective bundle with
$\hat\psi_1=\hat\phi_k$ and $K=\emptyset$.  The finite-horizon
identity is the same telescoping computation as in
Lemma~\ref{lem:finite-horizon-expansion}, with the fiber maximum
replaced by the fiber minimum.
\end{proof}

Let $\hat P_{G_k}$ and $\hat\phi_{G_k}$ be the Grassmannian operator
and $k$-Jacobian potential introduced before
Proposition~\ref{cor:met-linear-grassmann}.

\begin{corollary}
\label{cor:lower-grassmannian-random}
Under~\eqref{eq:dual-integrability-random},
\begin{align*}
 \underline\Lambda_k(\nu)
 &=\inf\left\{
      \int\hat\phi_{G_k}\,d\hat\nu:
      \hat\nu\in\mathcal I(\hat P_{G_k}),\
      (\pi_{G_k})_*\hat\nu=\nu
    \right\}\\
 &=\lim_{n\to\infty}\frac1n\int
   \min_{W\in\mathrm G_k(\mathcal E_x)}
   \mathbb E\left[
      \log\mathrm{Jac}_k(F^n_{(\omega,x)}|_W)
      \;\middle|\; (\omega_0,x)=(t,x)
   \right]d\nu,
\end{align*}
and the limit may be replaced by the supremum over $n\ge1$.
The infimum is attained on an ergodic invariant lift.  If $\nu$ is
ergodic, the corresponding fiber-minimal finite-time functions divided
by $n$ converge $\nu$-a.e.~to $\underline\Lambda_k(\nu)$.
\end{corollary}

\begin{proof}
The dual variational principle applies directly to the Grassmannian
lift.  The equality with the projective quantity follows from the
linear-algebra identity
\[
 \min_{[v]\in\mathbb P(\wedge^kV)}
 \log\frac{\|\wedge^kLv\|}{\|v\|}
 =\log m(\wedge^kL)
 =\min_{W\in\mathrm G_k(V)}\log\mathrm{Jac}_k(L|_W),
\]
valid for every invertible linear map $L:V\to V'$.
\end{proof}

In the deterministic ergodic case, the common value in the last two
corollaries is the bottom $k$-sum $\lambda_{d-k+1}(\mu)+\cdots+\lambda_d(\mu)$.
Indeed,
\[
 \log m(\wedge^kF_x^n)
 =\sum_{j=d-k+1}^d\log\sigma_j(F_x^n),
\]
and the conclusion follows from the classical pointwise Lyapunov
spectrum in Appendix~\ref{ss:lyapunov-spectrum}. 
\subsection{The Bernoulli lower exponent}
\label{subsec:bernoulli-lower-exponent}

In the Bernoulli setting of \S\ref{sec:Prop-IV}, let $\mu$ be an
$f$-stationary measure and let
\[
 \bar\phi_1(x,[v])
 :=\int_T\log\frac{\|F_{t,x}v\|}{\|v\|}\,dp.
\]
Define
\[
 \lambda_-(\mu)
 :=\inf\left\{
      \int\bar\phi_1\,d\hat\mu:
      \hat\mu\text{ is }\bar F_1\text{-stationary and }
      (\pi_1)_*\hat\mu=\mu
    \right\}.
\]
Corollary~\ref{cor:lower-projective-random}, applied to
$\nu=p\times\mu$, gives
\begin{align*}
 \lambda_-(\mu)
 &=\min\left\{
      \int\bar\phi_1\,d\hat\mu:
      \hat\mu\text{ is ergodic, }\bar F_1\text{-stationary, and }
      (\pi_1)_*\hat\mu=\mu
    \right\}
 \\
 &=\lim_{n\to\infty}\frac1n\int_X
      \min_{[v]\in\mathbb P(\mathcal E_x)}
      \int_\Omega
        \log\frac{\|F^n_{(\omega,x)}v\|}{\|v\|}
      \,d\mathbb P(\omega)\,d\mu(x)
\\
 &=\sup_{n\ge1}\frac1n\int_X
      \min_{[v]\in\mathbb P(\mathcal E_x)}
      \int_\Omega
        \log\frac{\|F^n_{(\omega,x)}v\|}{\|v\|}
      \,d\mathbb P(\omega)\,d\mu(x).
\end{align*}
When the base and bundle are compact and the cocycle depends
continuously on the base variables, Corollary~\ref{cor:uniform-dual-VP}
also gives
\begin{align*}
 \lambda_-(F)
 &:=\inf_{\mu\in\mathcal I(\bar P)}\lambda_-(\mu)
 =\lim_{n\to\infty}
   \min_{(x,[v])\in\mathbb P(\mathcal E)}
   \frac1n\int
      \log\frac{\|F^n_{(\omega,x)}v\|}{\|v\|}
   \,d\mathbb P \\
 &=\sup_{n\ge1}
   \min_{(x,[v])\in\mathbb P(\mathcal E)}
   \frac1n\int
      \log\frac{\|F^n_{(\omega,x)}v\|}{\|v\|}
   \,d\mathbb P
 =\inf_{(x,[v])\in\mathbb P(\mathcal E)}
   \liminf_{n\to\infty}
   \frac1n\int
      \log\frac{\|F^n_{(\omega,x)}v\|}{\|v\|}
   \,d\mathbb P.
\end{align*}

\section*{Acknowledgement}
P.~G.~Barrientos, was supported by grant  PID-2023-147461NB-I00 funded by MCIN, CNPq Projeto Universal 401737/2025-0 and PQ 302738/2025-8 (CNPq). I.~Nisoli was partially supported by the Postgraduate Program in Mathematics at UFRJ, CAPES–
Finance Code 001, CNPq Projeto Universal No. 404943/2023-3, CAPES–PRINT No. 88881.311616/2018-
00, and CAPES–STINT No. 88887.155746/2017-00.


\bibliographystyle{alpha4}
\bibliography{biblio}
\end{document}